\documentclass[12pt]{article}
\usepackage{amssymb, amsmath, amsthm, amscd}
\usepackage{mathtools}
\usepackage[utf8]{inputenc}
\usepackage[all,cmtip]{xy}
\usepackage{enumitem}
\usepackage{setspace}
\usepackage{mathdots}
\usepackage[mathscr]{euscript}
\usepackage[colorlinks=true]{hyperref}
\usepackage{dsfont}
\hypersetup{colorlinks   = true}
\hypersetup{linkcolor=blue}
\usepackage{tikz}

\begin{document}

\mathchardef\mhyphen="2D
\newtheorem{The}{Theorem}[section]
\newtheorem{Lem}[The]{Lemma}
\newtheorem{Prop}[The]{Proposition}
\newtheorem{Cor}[The]{Corollary}
\newtheorem{Rem}[The]{Remark}
\newtheorem{Obs}[The]{Observation}
\newtheorem{SConj}[The]{Standard Conjecture}
\newtheorem{Titre}[The]{\!\!\!\! }
\newtheorem{Conj}[The]{Conjecture}
\newtheorem{Question}[The]{Question}
\newtheorem{Prob}[The]{Problem}
\newtheorem{Def}[The]{Definition}
\newtheorem{Not}[The]{Notation}
\newtheorem{Claim}[The]{Claim}
\newtheorem{Conc}[The]{Conclusion}
\newtheorem{Ex}[The]{Example}
\newtheorem{Fact}[The]{Fact}
\newtheorem{Formula}[The]{Formula}
\newtheorem{Formulae}[The]{Formulae}
\newtheorem{The-Def}[The]{Theorem and Definition}
\newtheorem{Prop-Def}[The]{Proposition and Definition}
\newtheorem{Def-Prop}[The]{Definition and Proposition}
\newtheorem{Lem-Def}[The]{Lemma and Definition}
\newtheorem{Cor-Def}[The]{Corollary and Definition}
\newtheorem{Conc-Def}[The]{Conclusion and Definition}
\newtheorem{Terminology}[The]{Note on terminology}
\newcommand{\C}{\mathbb{C}}
\newcommand{\R}{\mathbb{R}}
\newcommand{\N}{\mathbb{N}}
\newcommand{\Z}{\mathbb{Z}}
\newcommand{\Q}{\mathbb{Q}}
\newcommand{\Proj}{\mathbb{P}}
\newcommand{\Rc}{\mathcal{R}}
\newcommand{\Oc}{\mathcal{O}}
\newcommand{\Vc}{\mathcal{V}}
\newcommand{\Id}{\operatorname{Id}}
\newcommand{\pr}{\operatorname{pr}}
\newcommand{\rk}{\operatorname{rk}}
\newcommand{\del}{\partial}
\newcommand{\delbar}{\bar{\partial}}
\newcommand{\Cdot}{{\raisebox{-0.7ex}[0pt][0pt]{\scalebox{2.0}{$\cdot$}}}}
\newcommand\nilm{\Gamma\backslash G}
\newcommand\frg{{\mathfrak g}}
\newcommand{\fg}{\mathfrak g}
\newcommand{\Oh}{\mathcal{O}}
\newcommand{\Kur}{\operatorname{Kur}}
\newcommand\gc{\frg_\mathbb{C}}
\newcommand\slawek[1]{{\textcolor{red}{#1}}}
\newcommand\dan[1]{{\textcolor{blue}{#1}}}
\newcommand\question[1]{{\textcolor{green}{#1}}}

\begin{center}

  {\Large\bf $m$-Positive Stability of Holomorphic Vector Bundles and Moduli Spaces}

\end{center}

\begin{center}

{\large Dan Popovici}

\end{center}

\vspace{1ex}

\noindent{\small{\bf Abstract.} We first propose a notion of $m$-positivity for higher-rank vector bundles, a variant of which reduces to the classical Griffiths positivity when $m=1$. Based on this, we go on to propose a generalisation of the classical Mumford-Takemoto theory of stability by means of a smooth function that we associate with every proper coherent subsheaf ${\cal F}$ of a given holomorphic vector bundle $E$. This places the emphasis on the holomorphic structure and the Hermitian fibre metric of $E$, rather than on numerical invariants of the smooth structure of $E$, making our stability conditions into relative pointwise $m$-positivity properties of $E$ with respect to its proper coherent subsheaves ${\cal F}$. We establish links with Hermite–Einstein geometry, prove that Hermite–Einstein bundles are uniformly semi-stable, study the resulting moduli spaces, and compare the new notions with the classical Mumford–Takemoto (semi-)stability notions.}

\vspace{1ex}

\section{Introduction}\label{section:introd} {\it Background and motivation}

The classical Mumford-Takemoto stability theory associates a real number, called the {\it slope} $\mu({\cal F})$, with every torsion-free coherent sheaf ${\cal F}$ over an $n$-dimensional compact K\"ahler manifold $(X,\,\omega)$ (or, more generally, a compact complex manifold on which a Gauduchon metric $\omega$ has been fixed). An important feature of the slope, defined as the ratio \begin{eqnarray*}\mu({\cal F})=\mu_\omega({\cal F}):=\frac{\int\limits_Xc_1({\cal F})\wedge\omega^{n-1}}{\rk({\cal F})} = \frac{\deg_\omega({\cal F})}{\rk({\cal F})},\end{eqnarray*} is that it depends only on the underlying $C^\infty$ determinant line bundle $\det{\cal F}$ associated with the sheaf ${\cal F}$ in at least any of the following two cases:

\vspace{1ex}

(a)\, the Hermitian metric $\omega$ on $X$ is {\it balanced} (i.e. $d\omega^{n-1} = 0$);

\vspace{1ex}

(b)\, the Hermitian metric $\omega$ on $X$ is Gauduchon (such metrics always exist) and $X$ is a {\it $\partial\bar\partial$-manifold}.

\vspace{1ex}

\noindent In other words, the degree (hence also the slope) is independent of the holomorphic structure of $\det{\cal F}$ in either of the cases (a) and (b).

Indeed, the first Chern class $c_1({\cal F})$ of ${\cal F}$ is defined to be the first Chern class of $\det{\cal F}$, the image under the canonical group isomorphism $c_1:H^1(X,\,{\cal E}^\star)\longrightarrow H^2(X,\,\Z)$ of the isomorphism class (of $C^\infty$ line bundles on $X$) of $\det{\cal F}$. One then first takes the image of this class in $H^2(X,\,\R)$ under the canonical morphism $H^2(X,\,\Z)\longrightarrow H^2(X,\,\R)$ and then the image $c_1^{dR}({\cal F})$ in the de Rham cohomology group $H^2_{dR}(X,\,\R)$ (which is known to be the de Rham cohomology class of the Chern curvature form of $\det{\cal F}$ with respect to any fibre metric) under the canonical isomorphism $H^2(X,\,\R)\longrightarrow H^2_{dR}(X,\,\R)$. Thus, by the Stokes theorem, the quantity $\int_Xc_1^{dR}({\cal F})\wedge\omega^{n-1}$ is well defined (i.e. is independent of the choice of representative of the de Rham class $c_1^{dR}({\cal F})$) when $d\omega^{n-1} = 0$.

However, if $\omega$ only satisfies the condition $\partial\bar\partial\omega^{n-1} = 0$ (equivalently, $\omega$ is a {\it Gauduchon metric}), one has to use the Bott-Chern version $c_1^{BC}({\cal F})$ (defined as the Bott-Chern cohomology class of the Chern curvature form of $\det{\cal F}$ with respect to any fibre metric) of the first Chern class of $\det{\cal F}$ in order to make the integral $\int_Xc_1^{BC}({\cal F})\wedge\omega^{n-1}$ well defined. Now, the map $H^{1,\,1}_{BC}(X,\,\C)\longrightarrow H^2_{dR}(X,\,\C)$ induced by the identity map of forms is always well defined, but it need not be injective. However, when $X$ is a {\it $\partial\bar\partial$-manifold}, this map is well known to be {\it injective}. Meanwhile, changing the holomorphic structure of $\det{\cal F}$ without changing the underlying $C^\infty$ structure returns the same de Rham class $c_1^{dR}({\cal F})$, hence also the same Bott-Chern class $c_1^{BC}({\cal F})$ if, in addition, $X$ is a {\it $\partial\bar\partial$-manifold}.

\vspace{1ex}

Then, a holomorphic vector bundle $E\to X$ is said to be {\it slope (semi-)stable} if $\mu({\cal F})<\mu(E)$ (respectively $\mu({\cal F})\leq\mu(E)$) for every proper coherent subsheaf ${\cal F}$ of ${\cal O}(E)$ (or of $E$, for short). This notion depends on the holomorphic structure of $E$ through the coherent subsheaves ${\cal F}$ on which it has to be tested (and which change when the holomorphic structure $D''$ on a given $C^\infty$ complex vector bundle $E$ changes). However, the slope does not.

\vspace{2ex}

\noindent {\it Our notions: $m$-positivity and $m$-positive stability}

In this paper, we propose a generalisation of the classical theory of stability by replacing the numerical slope invariant $\mu({\cal F})$ by a $C^\infty$ function $Z_h({\cal F}):X\longrightarrow\R$ associated with every triple $(E,\,h,\,{\cal F})$ consisting of a Hermitian holomorphic vector bundle $(E,\,h)$ of any rank $r\geq 1$ on $X$ and of a coherent subsheaf ${\cal F}\subset{\cal O}(E)$ of any rank $s\in\{1,\dots , r-1\}$. The underlying idea is to make the resulting (semi-)stability notions depend more heavily on the holomorphic structure $D''$ of $E$ and, in addition, on the given Hermitian fibre metric $h$ on $E$. This flexibility leads to phenomena that do not occur in the the classical setting.

Our approach starts with the introduction in $\S$\ref{section:m-pos_v-bundles} of pointwise $m$-positivity notions for holomorphic vector bundles, one of which reduces to the classical Griffiths positivity when $m=1$. Stability is then defined by requiring a suitable relative $m$-positivity condition between the given vector bundle $E$ and each of its proper coherent subsheaves ${\cal F}$.  

Throughout much of the paper, we will assume the existence on the given $n$-dimensional compact complex manifold $X$ of a K\"ahler metric $\omega$ and, for a fixed $m\in\{1,\dots , n\}$, of a form $\Omega\in C^\infty_{n-m,\,n-m}(X,\,\R)$ such that $\partial\bar\partial\Omega = 0$ and $\Omega>0$ (metrically weakly) on $X$. This positivity condition imposed on $\Omega$ was introduced in Definition 3.13. of the appendix to [Pop25] together with other strong and weak strict positivity notions for forms and currents that complement the classical strong and weak semi-positivity notions introduced by Lelong in the 1950's. It means that $\Omega - \varepsilon\,\omega^{n-m}\geq 0$ (weakly) at every point of $X$, for some constant $\varepsilon>0$. 

\subsection{$m$-positivity}\label{subsection:introd_m-positivity}

One of the (semi-)positivity notions we propose in $\S$\ref{subsection:m-pos_def}, at an arbitrary point $x\in X$, in the fairly general setting of an arbitrary Hermitian metric $\omega$ on $X$ is \begin{eqnarray*}(E,\,h) \hspace{1ex}\mbox{is}\hspace{1ex}  \omega\mbox{\bf -m-semi-positive}\hspace{1ex} \mbox{at}\hspace{1ex}  x & \iff & \bigg\{\bigg(i\Theta_h(E)\wedge\omega^{m-1}\wedge\Omega\bigg)\,u,\,u\bigg\}_h\geq 0, \hspace{2ex} u\in E_x, \\
  & & \Omega\in\Lambda^{n-m,\,n-m}T^\star_x X \hspace{1ex}\mbox{with}\hspace{1ex} \Omega> 0 \hspace{1ex} \mbox{(metrically weakly)},\end{eqnarray*} where $i\Theta_h(E)$ is the curvature form of the Chern connection of the Hermitian holomorphic vector bundle $(E,\,h)$.

Moreover, {\bf $\omega$-$m$-positivity} is defined in the same way by requiring a strict inequality ``$>0$''. Choosing $\Omega = \omega\wedge\Gamma$ with $\Gamma\in\Lambda^{n-m-1,\,n-m-1}T^\star_x X$ and $\Gamma>0$ (metrically weakly), we get the implication: \begin{eqnarray*}(E,\,h) \hspace{1ex}\mbox{is}\hspace{1ex}  \omega\mbox{\bf -m-semi-positive}\hspace{1ex} \mbox{at}\hspace{1ex}  x \implies (E,\,h) \hspace{1ex}\mbox{is}\hspace{1ex}  \omega\mbox{\bf -(m+1)-semi-positive}\hspace{1ex} \mbox{at}\hspace{1ex}  x\end{eqnarray*} for every $m\in\{1,\dots , n-2\}$. We notice in Observation \ref{Obs:m-pos_Griffiths-Nakano_comparison} that the strongest of these properties (the one for $m=1$) is equivalent to $(E,\,h)$ being {\bf Griffiths (semi-)positive} at $x$.

  When $\Omega$ has been fixed, these notions are called {\bf $(\omega,\,\Omega)$-$m$-(semi-)positivity}.

  Before proving functoriality properties for these positivity notions in $\S$\ref{subsection:prop_m-pos_operations}, we prove in $\S$\ref{subsection:vanishing} a Kobayashi-style vanishing theorem tailored for our setting in which a K\"ahler metric $\omega$ and a globally defined form $\Omega\in C^\infty_{n-m,\,n-m}(X,\,\R)$ with the properties specified above have been fixed (see Theorem \ref{The:Kobayashi-vanishing_gen} for a more precise statement):

  \begin{The}\label{The:introd_Kobayashi-vanishing_gen} (a)\, $E$ has no non-zero global holomorphic sections if there exists a $C^\infty$ Hermitian fibre metric $h$ on $E$ such that $(E,\,h)$ is {\bf $(\omega,\,\Omega)$-$m$-semi-negative} at every point of $X$ and {\bf $(\omega,\,\Omega)$-$m$-negative} at some point $x_0\in X$.

\vspace{1ex}

(b)\, Any non-zero global holomorphic section of $E$ is {\bf parallel} with respect to the Chern connection $D_h$ of any $C^\infty$ Hermitian fibre metric $h$ on $E$ such that $(E,\,h)$ is {\bf $(\omega,\,\Omega)$-$m$-semi-negative} at every point of $X$.

\end{The}    

  The proof is based on an application of the Bochner technique. Variants of it are used on several occasions throughout the paper.

  \subsection{$m$-positive stability}\label{subsection:introd_m-positive-stability}

  $\bullet$ Central to our approach is the use of the {\bf canonical section} $\sigma_{{\cal F}}\in H^0(X,\,\Lambda^s E\otimes\det{\cal F}^\star)$ induced by every coherent subsheaf ${\cal F}\subset{\cal O}(E)$ of rank $s\in\{1,\dots , r-1\}$ of the given holomorphic vector bundle $E$. (See Fact \ref{Fact:canonical-section}.) Once a smooth Hermitian fibre metric $h$ has been fixed on $E$, we associate with the pair $(\sigma_{{\cal F}},\,h)$ (see Definition \ref{Def:Z-function}) a $C^\infty$ function $Z_h({\cal F})= Z_{\omega,\,\Omega,\,h}^{(m)}({\cal F}):X\longrightarrow\R$, that we call the {\bf stability function}, comparing the curvature of $E$ with the Hermite-Einstein geometry of $\det{\cal F}$ through the canonical section associated with ${\cal F}$. The definition uses the (unique, up to a positive multiplicative constant) $(\omega,\,\Omega)$-Hermite-Einstein fibre metric (a generalisation proposed in [Pop25] of the classical Hermite-Einstein condition) on the line bundle $\det{\cal F}^\star$. Its normalised version, $\widehat{Z}_h({\cal F}):=Z_h({\cal F})/|\sigma_{{\cal F}}|^2:X\longrightarrow\R$, is independent of the choice of $(\omega,\,\Omega)$-Hermite-Einstein fibre metric on $\det{\cal F}^\star$.

 A key observation that shapes our approach is provided by Proposition \ref{Prop:Z-function_alternative} which gives the following dichotomy for every proper coherent subsheaf ${\cal F}\subset{\cal O}(E)$: \begin{eqnarray*}\label{eqn:introd_Z-function_alternative}\mbox{either}\hspace{2ex} Z_h({\cal F})\equiv 0 \hspace{5ex}\mbox{or}\hspace{5ex} Z_h({\cal F})(x_0)>0 \hspace{1ex}\mbox{for some point}\hspace{1ex} x_0\in X.\end{eqnarray*} The same proposition ensures that we always have $\int_X Z_h({\cal F})\,dV_\omega \geq 0$, with equality if and only if $\sigma_{{\cal F}}$ is parallel with respect to the Chern connection of $(\Lambda^s E\otimes\det{\cal F}^\star,\,\Lambda^sh\otimes h^{H-E}_{\det{\cal F}^\star})$.

This leads us to introduce in Definition \ref{Def:m-pos-stability} our notions of stability. For the sake of perspicuity, we drop the pair $(\omega,\,\Omega)$ from the terminology used in this introduction and occasionally elsewhere throughout the paper.

A holomorphic vector bundle $E = (E,\,D'')\longrightarrow X$ of rank $r\geq 1$ is said to be {\bf $m$-positively stable} (respectively {\bf $m$-positively semi-stable}) if for every coherent subsheaf ${\cal F}\subset{\cal O}(E)$ of any rank $s\in\{1,\dots , r-1\}$, there exists a $C^\infty$ Hermitian fibre metric $h$ on $E$ such that \begin{eqnarray*} Z_h({\cal F})> 0 \hspace{5ex} (\mbox{respectively} \hspace{1ex} Z_h({\cal F})\geq 0) \hspace{5ex} \mbox{at every point of}\hspace{1ex} X.\end{eqnarray*} Note that $h$ depends on ${\cal F}$ in this definition.

The {\bf uniform} strengthening of this definiton requires the same conditions to hold with a same $C^\infty$ Hermitian fibre metric $h$ on $E$ for all proper coherent subsheaves ${\cal F}\subset{\cal O}(E)$. In this case, we use interchangeably any of the following formulations (and we may further drop the ``$m$-positively''): \begin{eqnarray*}(E,\,h) \hspace{1ex}\mbox{is {\bf uniformly stable} (resp. {\bf uniformly semi-stable})} & \iff & \\
  h \hspace{1ex}\mbox{is {\bf uniformly stable} (resp. {\bf uniformly semi-stable}) for}\hspace{1ex} D'' \hspace{1ex} \mbox{(or for}\hspace{1ex} E) &\iff & \\
  D'' \hspace{1ex}\mbox{is {\bf uniformly stable} (resp. {\bf uniformly semi-stable}) for}\hspace{1ex} h.\end{eqnarray*} When no metric is specified, we will say that the holomorphic structure $D''$ of $E$ (or $E$ or $(E,\,D'')$) is {\bf uniformly (semi-)stable} if there exists a $C^\infty$ Hermitian fibre metric $h$ on $E$ for which $D''$ has this property.

\vspace{2ex}

The quadratic form that defines the stability function of ${\cal F}\subset{\cal O}(E)$ through an evaluation on the canonical section $\sigma_{\cal F}$ links these stability notions to the $m$-positivity notions introduced in $\S$\ref{subsection:m-pos_def}. 

\vspace{2ex}

$\bullet$ A large portion of the paper is devoted to the study of the properties of these stability notions and to the moduli spaces they induce. We gather some of them in the following statement that merges (parts of) several results obtained in $\S$\ref{section:m-pos-stability_def-prop_stability} and $\S$\ref{section:moduli}.

\begin{The}\label{The:introd_summary} $(1)$\, (cf. Theorem \ref{The:m-pos-stable_implies_simple}) If $(E,\,D'')$ is {\bf $m$-positively stable}, then $(E,\,D'')$ is {\bf simple} (in the sense that every holomorphic endomorphism $T\in H^0(X,\,\operatorname{End} E)$ is an {\bf isometry} -- namely, it equals $\lambda\,\mbox{Id}_E$ for some constant $\lambda\in\C$).

\vspace{1ex}

$(2)$\, (cf. Proposition \ref{Prop:m-pos-stability_G-C-action_invariance}) For every $C^\infty$ $\C$-vector bundle $E$ of rank $r\geq 1$ on $X$, the subset \begin{eqnarray*}{\cal H}_{m\mbox{-stable}}(E)\subset{\cal H}^s(E),\end{eqnarray*} of holomorphic structures $D''$ on $E$ such that the holomorphic vector bundle $(E,\,D'')$ is {\bf $m$-positively stable} is {\bf invariant} under the right action of the complex gauge group ${\cal G}^{\C} = C^\infty(X,\,\mbox{Aut}\,E) = GL(E)\subset C^\infty(X,\,\operatorname{End} E)$ of $C^\infty$ automorphisms of $E$. By ${\cal H}^s(E)$ we mean the set of simple holomorphic structures $D''$ on $E$.

\vspace{1ex}

$(3)$\, (cf. Corollary and Definition \ref{Cor-def:moduli-stable_set}) The set-theoretic {\bf moduli space of $m$-positively stable holomorphic structures} on $E$, defined to be the first quotient in the inclusion: \begin{eqnarray*}{\cal M}_{m\mbox{-stable}}(E):={\cal H}_{m\mbox{-stable}}(E)/{\cal G}^{\cal C}\xhookrightarrow { } {\cal M}^s(E)={\cal H}^s(E)/{\cal G}^{\cal C},\end{eqnarray*} where the second quotient is the moduli space of simple holomorphic structures on $E$, is meaningful.

\vspace{1ex}

$(4)$\, (cf. Corollary \ref{Cor:uniform-stab_G-C-action_invariance} -- a refinement of $(2)$) For every $T\in{\cal G}^{\cal C}$, the following equivalence holds: \begin{eqnarray*}D'' \hspace{1ex} \mbox{is {\bf uniformly stable} for}\hspace{1ex} h \iff  D''_T:=T^{-1}\circ D''\circ T \hspace{1ex} \mbox{is {\bf uniformly stable} for}\hspace{1ex} h_T,\end{eqnarray*} where $h_T$ is the $C^\infty$ Hermitian fibre metric on $E$ defined by $T$ by requiring \begin{eqnarray*}\langle u,\,v\rangle_{h_T}:= \langle Tu,\,Tv\rangle_h, \hspace{5ex} u,v\in E_x,\hspace{1ex} x\in X.\end{eqnarray*}

\vspace{1ex}

$(5)$\, (cf. Corollary and Definition \ref{Cor-Def:m-pos-stable_moduli} and Proposition \ref{Prop:Hausdorff_moduli_h-unif-stable}) For any $C^\infty$ Hermitian fibre metric $h$ on $E$ that is uniformly stable for at least one holomorphic structure $D''$ on $E$, the set \begin{eqnarray*}{\cal H}_{m\mbox{-u-stable}}(E,\,h) &:=& \bigg\{D''\,\mid\, D''\in{\cal H}(E) \hspace{1ex} \mbox{is {\bf uniformly stable for} h}\bigg\}\end{eqnarray*} is invariant under the right action of the group $U(E,\,h)$ of $h$-unitary automorphisms of $E$.

  Moreover, the set-theoretic {\bf moduli space of $h$-uniformly stable holomorphic structures} on $E$, defined to be the quotient: \begin{eqnarray*}{\cal M}_{m\mbox{-u-stable}}(E,\,h):={\cal H}_{m\mbox{-u-stable}}(E,\,h)/U(E,\,h)\end{eqnarray*} is meaningful and is {\bf Hausdorff} when equipped with the quotient topology induced by the $C^\infty$ topology of ${\cal H}_{m\mbox{-u-stable}}(E,\,h)$. 

\end{The}  

\vspace{2ex}

$\bullet$ In $\S$\ref{section:w-h-subbundle_formalism}, we adapt to our setting the Uhlenbeck-Yau formalism of [UY86] and [UY89] to give an alternative point of view on the stability notions introduced in this paper. Recall that, once a $C^\infty$ Hermitian fibre metric $h$ has been fixed on a holomorphic vector bundle $E$, a coherent analytic subsheaf ${\cal F}\subset{\cal O}(E)$ identifies with the orthogonal projection $\pi\in L^2_1(X,\,\operatorname{End}E)$ of $E$ onto the holomorphic vector subbundle $F$ of $E$ that ${\cal F}$ defines outside an analytic subset $S\subset X$ of codimension $\geq 2$. The regularity of $\pi$ is $C^\infty$ on $X\setminus S$ and Sobolev $L^2_1$ on the whole of $X$.

We start from the observation that $\Lambda^s\pi$ is actually $C^\infty$ on the whole of $X$, where $s$ is the rank of ${\cal F}$, and that it relates nicely through the identity \begin{eqnarray*}\Lambda^s\pi = \sigma_{\cal F}\circ\sigma_{\cal F}^\star\in C^\infty(X,\,\operatorname{End}(\Lambda^sE))\end{eqnarray*} to the canonical section $\sigma_{\cal F}$ induced by ${\cal F}$ (see Corollary \ref{Cor:Lambda_s_exact-seq}). We go on to prove the following result (see Theorem and Definition \ref{The-Def:discrepancy-equation} for further details):

\begin{The}\label{The:introd_discrepancy-equation} For every coherent subsheaf ${\cal F}$ of rank $s\in\{1,\dots , r-1\}$ of a Hermitian holomorphic vector bundle $(E,\,h)$ of rank $r\geq 1$, the following identities hold: \begin{eqnarray*}\widehat{Z}_h({\cal F}) = e^{-f}\,\frac{Z_{\omega,\,\Omega,\,h}^{(m)}({\cal F})}{|\sigma_{{\cal F}}|^2_{h_A}} = P_{\omega,\,\Omega}(f) - Q(\pi),\end{eqnarray*} where the unique $C^\infty$ function $f = f_{\cal F}:X\longrightarrow\R$ such that \begin{eqnarray*}h_{\det{\cal F}}^{H-E} = e^{-f}\,h_{\det{\cal F}}\end{eqnarray*} is called the {\bf discrepancy function} of the triple $(E,\,h,\,{\cal F})$.

  Moreover, $Q(\pi)\leq 0$ everywhere on $X$. 

\end{The}  

  The notation used above is: $h_{\det{\cal F}}^{H-E}$ is the (unique, up to a positive multiplicative constant) {\bf $(\omega,\,\Omega)$-Hermite-Einstein} fibre metric on the line bundle $\det{\cal F}$; $h_{\det{\cal F}}$ is the fibre metric on $\det{\cal F}$ induced by the fibre metric $\Lambda^sh$ of $\Lambda^sE$ when viewing $\det{\cal F}$ as a holomorphic {\bf subbundle} of $\Lambda^sE$; $P_{\omega,\,\Omega}$ is a Laplacian-like second-order differential operator on the functions on $X$ (see (\ref{eqn:P_operator_def})); $Q(\pi)$ is a $C^\infty$ function on $X$ depending on $\Lambda^s\pi$, its first order derivatives and the canonical section $\sigma_{\cal F}$.   

The key takeaway from the above theorem is Corollary \ref{Cor:discrepancy-equation} whose part (i) expresses the {\bf uniform (semi-)stability} condition on $(E,\,h)$ as the inequality $P_{\omega,\,\Omega}(f) > Q(\pi)$ (respectively $\geq 0$) holding on $X$ for every coherent ${\cal F}\subset{\cal O}(E)$. In particular, $(E,\,h)$ is {\bf uniformly semi-stable} if $h$ induces an $(\omega,\,\Omega)$-Hermite-Einstein metric on $\det {\cal F}$ for every ${\cal F}$. 

\vspace{1ex}

$\bullet$ In $\S$\ref{section:H-E_link}, we relate our notion of stability with the classical slope stability and Hermite-Einstein theories. The results of that section can be loosely summed up as follows:

\begin{The}\label{The:introd_H-E-stability} The following implications and non-implication hold: \begin{eqnarray}\label{eqn:introd_H-E-stability_implications}\exists\, h \hspace{1ex} \mbox{{\bf $(\omega,\,\Omega)$-Hermite-Einstein} metric on}\hspace{1ex} E \implies (E,\,h) \hspace{1ex} \mbox{is {\bf uniformly semi-stable}};\end{eqnarray} \begin{eqnarray}\label{eqn:introd_slope-stab_implies_unif-stab}E \hspace{1ex} \mbox{\bf slope stable} \implies (E,\,h) \hspace{1ex} \mbox{{\bf uniformly stable} for some}\hspace{1ex} C^\infty \hspace{1ex} \mbox{Hermitian fibre metric}\hspace{1ex} h;\end{eqnarray} \begin{eqnarray}\label{eqn:introd_unif-semi-stab_not-implies_slope-semi}(E,\,h) \hspace{1ex} \mbox{\bf uniformly semi-stable} \not\Longrightarrow (E,\,h) \hspace{1ex} \mbox{{\bf slope semi-stable}}. \end{eqnarray}  

The notions considered in each of these implications are relative to a same pair $(\omega,\,\Omega)$.

\end{The}

Implication (\ref{eqn:introd_H-E-stability_implications}) is proven in Theorem \ref{The:H-E_implies_m-pos-stability} whose second part gives an analogue in our {\it uniformly stable} context of the classical polystability property.

Implication (\ref{eqn:introd_slope-stab_implies_unif-stab}), in which we omitted all the complexities of the notation in order to enhance readability, follows by combining the generalisation to Gauduchon metrics of the main Uhlenbeck-Yau result of [UY86] and [UY89], obtained by Li and Yau in [LY86] (see also [Buc88]), with our Theorem \ref{The:H-E_implies_m-pos-stability}. Indeed, with our assumptions on the forms $\omega$ and $\Omega$, there exists a unique positive definite $(1,\,1)$-form $\rho$ on $X$ such that $\omega^{m-1}\wedge\Omega = \rho^{n-1}/(n-1)!$. Moreover, $\rho$ is $C^\infty$ and $\partial\bar\partial\rho^{n-1} = 0$ at every point of $X$. Thus, $\rho$ is a Gauduchon metric on $X$.

In other words, our notion of {\it uniform stability} generalises the classical Mumford-Takemoto notion of {\it slope stability}. The generalisation is probably strict (as the definitions and the available evidence seem to suggest), but at the moment we cannot produce an example of a vector bundle that is uniformly stable, but not slope stable.

However, implication (\ref{eqn:introd_unif-semi-stab_not-implies_slope-semi}), proven in Theorem \ref{The:unif-semi-stab_not-slope-semi-stab} by producing an explicit counter-example, shows that, at least at the ``semi-stability'' level, our notion is different from the classical Mumford-Takemoto slope semi-stability. 

\vspace{2ex}

\noindent {\bf Acknowledgments.} The author is very grateful to S. Dinew for helpful discussions and to N. Buchdahl for useful comments on the manuscript.

\section{$m$-positivity for higher-rank vector bundles}\label{section:m-pos_v-bundles} Let $(E,\,h)\longrightarrow X$ be a Hermitian holomorphic vector bundle of rank $r\geq 1$ over a complex manifold with $\mbox{dim}_\C X =n$. The Hermitian fibre metric $h$ that has been fixed on $E$ is supposed to be $C^\infty$.

It is standard that the curvature form of the Chern connection $D_h=D'_h + \bar\partial$ of $(E,\,h)$ is a $C^\infty$ $(1,\,1)$-form on $X$ with values in the holomorphic vector bundle $\mbox{End}(E)$ of endomorphisms of $E$: \begin{eqnarray*}i\Theta_h(E)\in C^\infty_{1,\,1}(X,\,\mbox{End}(E)).\end{eqnarray*}

We will use the standard pointwise sesquilinear pairing $\{\,\cdot\,,\,\cdot\,\}_h$ induced by $h$ on $E$-valued differential forms on $X$ (cf. e.g. [Dem97, chapter V]). For any point $x\in X$ and any forms $u\in\Lambda^{p,\,q}T^\star_x X\otimes E_x$ and $v\in\Lambda^{r,\,s}T^\star_x X\otimes E_x$, the form $\{u,\,v\}_h \in\Lambda^{p+s,\,q+r}T^\star_x X\otimes E_x$ is defined by taking the exterior product of the scalar part of $u$ by the conjugate of the scalar part of $v$ and then multiplying the result by the inner product (defined by $h(x)$) of the vector part (i.e. the $E_x$-valued part) of $u(x)$ by the vector part of $v(x)$.

In particular, $\{(i\Theta_h(E)\wedge\omega^{m-1}\wedge\Omega)\,u,\,u\}_h$ is an $\R$-valued $(n,\,n)$-form at any given point $x\in X$ if $u\in E_x$.

\subsection{Definition of vector bundle $m$-positivity}\label{subsection:m-pos_def}

The first notion we propose in this paper is the following alternative to the classical Griffiths/m-/Nakano positivity notions for holomorphic vector bundles.

\begin{Def}\label{Def:m-pos} Let $\omega$ be a Hermitian metric on an $n$-dimensional complex manifold $X$, let $m\in\{1,\dots , n\}$ and let $x\in X$. A Hermitian holomorphic vector bundle $(E,\,h)\longrightarrow X$ of rank $r\geq 1$ over $X$ is said to be:

\vspace{1ex}

(a)\,{\bf $\omega$-$m$-semi-positive} at $x$ if \begin{eqnarray*}\bigg\{\bigg(i\Theta_h(E)\wedge\omega^{m-1}\wedge\Omega\bigg)\,u,\,u\bigg\}_h\geq 0\end{eqnarray*} for every $\R$-valued $(n-m,\,n-m)$-form $\Omega\in\Lambda^{n-m,\,n-m}T^\star_x X$ at $x$ such that $\Omega> 0$ (metrically weakly) and for every $u\in E_x$; 

\vspace{1ex}

(b)\,{\bf $\omega$-$m$-positive} at $x$ if \begin{eqnarray*}\bigg\{\bigg(i\Theta_h(E)\wedge\omega^{m-1}\wedge\Omega\bigg)\,u,\,u\bigg\}_h> 0\end{eqnarray*} for every $\R$-valued $(n-m,\,n-m)$-form $\Omega\in\Lambda^{n-m,\,n-m}T^\star_x X$ at $x$ such that $\Omega> 0$ (metrically weakly) and for every $u\in E_x\setminus\{0\}$;

\vspace{1ex}

(c)\,{\bf $(\omega,\,\Omega)$-$m$-semi-positive} at $x$ for a given $\R$-valued $(n-m,\,n-m)$-form $\Omega\in\Lambda^{n-m,\,n-m}T^\star_x X$ at $x$ such that $\Omega> 0$ (metrically weakly) if \begin{eqnarray*}\bigg\{\bigg(i\Theta_h(E)\wedge\omega^{m-1}\wedge\Omega\bigg)\,u,\,u\bigg\}_h\geq 0\end{eqnarray*} for every $u\in E_x$;

\vspace{1ex}

(d)\,{\bf $(\omega,\,\Omega)$-$m$-positive} at $x$ for a given $\R$-valued $(n-m,\,n-m)$-form $\Omega\in\Lambda^{n-m,\,n-m}T^\star_x X$ at $x$ such that $\Omega> 0$ (metrically weakly) if \begin{eqnarray*}\bigg\{\bigg(i\Theta_h(E)\wedge\omega^{m-1}\wedge\Omega\bigg)\,u,\,u\bigg\}_h> 0\end{eqnarray*} for every $u\in E_x\setminus\{0\}$.

\end{Def}

Parts (a) and (b) of this definition mean that $(E,\,h)$ is $\omega$-$m$-(semi-)positive if and only if the $\mbox{End} (E)$-valued $(m,\,m)$-form $i\Theta_h(E)\wedge\omega^{m-1}$ defines a {\it strongly (semi-)positive} $\mbox{End} (E)$-valued current of bidegree $(m,\,m)$ on $X$. Thus, the theoretical underpinning of this definition is the duality between strongly (semi-)positive currents and weakly (semi-)positive test forms in the complementary bidegree. Of course, $\omega$ disappears from the definitions (and we will suppress it from the names) when $m=1$.

The counterparts of the notions introduced in Definition \ref{Def:m-pos} obtained by replacing the word ``positive'' by the word ``negative'' are defined by replacing $i\Theta_h(E)$ by $-i\Theta_h(E)$. 

\vspace{1ex}

To begin with, we notice that our $m$-(semi-)positivity notions for vector bundles generalise the classical Griffiths (semi-)positivity, hence also the classical Demailly $m$-(semi-)positivity and Nakano (semi-)positivity notions. (See e.g. [Dem97, VII-$\S6$] for the classical definitions.)

\begin{Obs}\label{Obs:m-pos_Griffiths-Nakano_comparison} Let $(E,\,h)\longrightarrow X$ be a Hermitian holomorphic vector bundle of rank $r\geq 1$ over an $n$-dimensional complex manifold. Let $m\in\{1,\dots , n\}$, $x\in X$ and $\omega$ a Hermitian metric on $X$.

  The following equivalences, implications and statement hold:

\vspace{1ex}  

\hspace{2ex} (a)\,\, $(E,\,h)$ is {\bf Griffiths (semi-)positive} at $x$ $\iff$ $(E,\,h)$ is {\bf $1$-(semi-)positive} at $x$;

\vspace{1ex}  

\hspace{2ex} (b)\,\, $(E,\,h)$ is {\bf $\omega$-$m$-(semi-)positive} at $x$ $\implies$ $(E,\,h)$ is {\bf $\omega$-$(m+1)$-(semi-)positive} at $x$;

\vspace{1ex}  

\hspace{2ex} (c)\,\, $(E,\,h)$ is {\bf $\omega$-$m$-(semi-)positive} at $x$ $\iff$ 

\noindent $\iff$ $(E,\,h)$ is {\bf $(\omega,\,\Omega)$-$m$-(semi-)positive} at $x$ for all $\Omega\in\Lambda^{n-m,\,n-m}T^\star_x X$ such that $\Omega> 0$ (metrically weakly);

\vspace{1ex}  

\hspace{2ex} (d)\,\, for each $\Omega\in\Lambda^{n-m,\,n-m}T^\star_x X$ such that $\Omega> 0$ (metrically weakly), the condition that $(E,\,h)$ be {\bf $(\omega,\,\Omega)$-$m$-(semi-)positive} at $x$ is strictly {\bf weaker} than the condition that  $(E,\,h)$ be {\bf Griffiths (semi-)positive} at $x$.

\end{Obs}  

\noindent {\it Proof.} (a)\, If one uses the $\C$-linear isomorphism: \begin{eqnarray*}E\longrightarrow\overline{E^\star}, \hspace{5ex} u\longmapsto\langle u,\,\cdot\rangle_h,\end{eqnarray*} the condition (as spelt out in [Dem97, VII-$\S6$]) that $(E,\,h)$ be Griffiths (semi-)positive at $x$ translates to requiring that, for each fixed $u\in E_x$, the $\R$-valued $(1,\,1)$-form $\{i\Theta_h(E)\,u,\,u\}_h$ at $x$ be (semi-)positive definite. Now, it is standard that in bidegree $(1,\,1)$ the weak and strong positivity notions coincide, so $\{i\Theta_h(E)\,u,\,u\}_h$ is (semi-)positive definite if and only if \begin{eqnarray*}\{i\Theta_h(E)\,u,\,u\}_h\wedge\Omega = \bigg\{\bigg(i\Theta_h(E)\wedge\Omega\bigg)\,u,\,u\bigg\}_h >0 \hspace{3ex} (\mbox{resp.} \hspace{1ex} \geq 0)\end{eqnarray*} for every positive definite $(n-1,\,n-1)$-form $\Omega$ at $x$. This last condition holding for every $u\in E_x$ is precisely the $1$-(semi-)positivity condition on $(E,\,h)$ as spelt out under (a) of Definition \ref{Def:m-pos}.

(b) and (c) follow at once from the definitions.

(d)\, It is standard (and easily checked) that the map $\rho\mapsto\rho_{n-1}:=\rho^{n-1}/(n-1)!$ is a bijection from the set of positive definite $(1,\,1)$-forms at $x$ and the set of positive definite $(n-1,\,n-1)$-forms at $x$. In particular, there exists a unique $(1,\,1)$-form $\rho>0$ at $x$ such that $\rho_{n-1} = \omega^{m-1}\wedge\Omega$.

Meanwhile, for a $(1,\,1)$-form $\alpha$, we have: $\alpha\wedge\rho_{n-1} = \Lambda_\rho(\alpha)\,dV_\rho$, where $\Lambda_\rho$ is the adjoint w.r.t. the pointwise inner product induced by $\rho$ of the multiplication operator $\rho\wedge\cdot$. For a $(1,\,1)$-form $\alpha$, $\Lambda_\rho(\alpha)$ is the trace of $\alpha$ w.r.t. $\rho$. 

Therefore, $(E,\,h)$ is $(\omega,\,\Omega)$-$m$-(semi-)positive at $x$ if and only if for each fixed $u\in E_x$, the trace with respect to $\rho$ of the $\R$-valued $(1,\,1)$-form $\{i\Theta_h(E)\,u,\,u\}_h$ at $x$ is positive (respectively non-negative): \begin{eqnarray*}\Lambda_\rho\bigg(\{i\Theta_h(E)\,u,\,u\}_h\bigg) = \bigg\{\Lambda_\rho\bigg(i\Theta_h(E)\bigg)\,u,\,u\bigg\}_h > 0  \hspace{3ex} (\mbox{resp.} \hspace{1ex} \geq 0).\end{eqnarray*}

On the other hand, Griffiths (semi-)positivity of $(E,\,h)$ at $x$ means that, for each fixed $u\in E_x$, the $\R$-valued $(1,\,1)$-form $\{i\Theta_h(E)\,u,\,u\}_h$ at $x$ (not just its trace) is (semi-)positive definite.  \hfill $\Box$

\subsection{Vanishing theorem}\label{subsection:vanishing}

The first application of our $m$-(semi-)positivity notions for vector bundles shows that Kobayashi's vanishing theorem (which plays a key role in the Kobayashi-Hitchin correspondence) still holds in this more general context. The next statement and its proof generalise to the setting of this paper Theorem 2.8. of [Pop25].

\begin{The}\label{The:Kobayashi-vanishing_gen} Let $(X,\,\omega)$ be a compact K\"ahler manifold with $\mbox{dim}_\C X =n$. Let $m\in\{1,\dots , n\}$ and suppose there exists a form $\Omega\in C^\infty_{n-m,\,n-m}(X,\,\R)$ such that $\partial\bar\partial\Omega = 0$ and $\Omega>0$ (metrically weakly) on $X$.

\vspace{1ex}  

(a)\, If there exists a $C^\infty$ Hermitian fibre metric $h$ on $E$ such that $(E,\,h)$ is {\bf $(\omega,\,\Omega)$-$m$-semi-negative} at every point of $X$ and {\bf $(\omega,\,\Omega)$-$m$-negative} at some point $x_0\in X$, then $H^0(X,\,E) = \{0\}$.

\vspace{1ex}  

(b)\, If there exists a $C^\infty$ Hermitian fibre metric $h$ on $E$ such that $(E,\,h)$ is {\bf $(\omega,\,\Omega)$-$m$-semi-negative} at every point of $X$,
then for every $s\in H^0(X,\,E)$ we have $D_h s=0$ on $X$ (i.e. $s$ is parallel with respect to the Chern connection $D_h$ of $(E,\,h)$).

\end{The}

\noindent {\it Proof.} We will use the second-order differential operator: \begin{eqnarray}\label{eqn:P_operator_def}P_{\omega,\,\Omega}:C^\infty(X,\,\R)\longrightarrow C^\infty(X,\,\R), \hspace{5ex} P_{\omega,\,\Omega}(\varphi):=-\frac{i\partial\bar\partial\varphi\wedge\omega^{m-1}\wedge\Omega}{dV_\omega},\end{eqnarray} that we introduced in [DP25] and re-used in [Pop25].

Let $s\in H^0(X,\,E)$. We have: \begin{eqnarray*}P_{\omega,\,\Omega}|s|^2_h = \frac{i\bar\partial\partial\{s,\,s\}_h\wedge\omega^{m-1}\wedge\Omega}{dV_\omega}.\end{eqnarray*}

The $h$-compatibility of $D_h$ gives the first inequality below, while the holomorphicity of $s$ yields the second: \begin{eqnarray*}\partial\{s,\,s\}_h = \{D_h's,\,s\}_h + \{s,\,\bar\partial s\}_h = \{D_h's,\,s\}_h.\end{eqnarray*} Taking $\bar\partial$, we further get: \begin{eqnarray*}\bar\partial\partial\{s,\,s\}_h = \{\bar\partial D_h's,\,s\}_h - \{D_h's,\,D_h's\}_h = \{\Theta_h(E)s,\,s\}_h - \{D_h's,\,D_h's\}_h.\end{eqnarray*} Hence: \begin{eqnarray*}i\bar\partial\partial\{s,\,s\}_h\wedge\omega^{m-1}\wedge\Omega & = & \bigg\{\bigg(i\Theta_h(E)\wedge\omega^{m-1}\wedge\Omega\bigg)s,\,s\bigg\}_h - i\,\{D_h's,\,D_h's\}_h\wedge\omega^{m-1}\wedge\Omega \\
  & \leq & 0\end{eqnarray*} at every point of $X$. The last inequality follows from the $(\omega,\,\Omega)$-$m$-semi-negativity hypothesis on $(E,\,h)$ at every point of $X$ (see (c) of Definition \ref{Def:m-pos} with $-i\Theta_h(E)$ in place of $i\Theta_h(E)$) and from Lemma 2.9. of [Pop25] which guarantees (for any $C^\infty$ Hermitian fibre metric $h$ on $E$, irrespective of whether its curvature form satisfies a positivity/negativity condition or not) that, for every $\eta\in C^\infty_{1,\,0}(X,\,E)$, the $\R$-valued $C^\infty$ $(n,\,n)$-form $i\,\{\eta,\,\eta\}_h\wedge\omega^{m-1}\wedge\Omega$ is $\geq 0$ at every point of $X$. 

  We have thus proved that, under the assumptions of either (a) or (b), we have: \begin{eqnarray}\label{eqn:P_s-squared_formula}P_{\omega,\,\Omega}|s|^2_h \leq 0\end{eqnarray} at every point of $X$ for every $s\in H^0(X,\,E)$. Since $P_{\omega,\,\Omega}$ is an elliptic second-order differential operator whose formal adjoint $P_{\omega,\,\Omega}^\star$ has no zero$^{th}$-order terms (cf. [DP25, Lemma 2.8. and Corollary 2.9.]) and $X$ is compact, the maximum principle implies that \begin{eqnarray*}P_{\omega,\,\Omega}|s|^2_h\equiv 0 \hspace{3ex}\mbox{and}\hspace{3ex} |s|^2_h = Const \hspace{1ex} \mbox{on}\hspace{1ex} X.\end{eqnarray*}

    Moreover, inequality (\ref{eqn:P_s-squared_formula}) is strict at $x_0\in X$ under the hypothesis of (a) if $s(x_0)\neq 0\in E_{x_0}$. We deduce that $s(x_0)=0$, hence $s(x)=0\in E_x$ for every $x\in X$. This proves (a).

    Under the weaker hypothesis of (b), we get $P_{\omega,\,\Omega}|s|^2_h\equiv 0$, hence also $i\,\{D_h's,\,D_h's\}_h\wedge\omega^{m-1}\wedge\Omega \equiv 0$. This implies that $D_h's \equiv 0$ (hence also $D_hs \equiv 0$ since $\bar\partial s\equiv 0$ by the holomorphicity of $s$) thanks to the second part of Lemma 2.9. of [Pop25] which guarantees (for any $C^\infty$ Hermitian fibre metric $h$ on $E$) the following equivalence for every $\eta\in C^\infty_{1,\,0}(X,\,E)$ at every point $x\in X$: \begin{eqnarray*}i\,\{\eta,\,\eta\}_h\wedge\omega^{m-1}\wedge\Omega = 0 \hspace{2ex}\mbox{at}\hspace{1ex} x \iff \eta(x)=0.\end{eqnarray*} \hfill $\Box$

\subsection{Properties of vector bundle $m$-positivity}\label{subsection:prop_m-pos_operations}

We now show that our m-(semi-)positivity notions behave well under vector bundle operations.

\begin{Prop}\label{Prop:m-pos_properties} Let $(E,\,h_E)\longrightarrow X$ be a Hermitian holomorphic vector bundle of rank $r\geq 1$ over an $n$-dimensional complex manifold $X$. Let $\omega$ be a Hermitian metric on $X$, $m\in\{1,\dots , n\}$, $x\in X$ and $\Omega\in\Lambda^{n-m,\,n-m}T^\star_x X$ an $\R$-valued $(n-m,\,n-m)$-form at $x$ such that $\Omega> 0$ (metrically weakly).

\vspace{1ex}

 (1)\, The following equivalence holds:

\vspace{1ex}

\hspace{-4ex} $(E,\,h_E)$ is  $(\omega,\,\Omega)$-$m$-(semi-)positive at $x$ {\bf if and only if} $(E^\star,\,h_E^\star)$ is  $(\omega,\,\Omega)$-$m$-(semi-)negative at $x$,

\vspace{1ex}

\noindent where $(E^\star,\,h_E^\star)$ is the dual vector bundle of $E$ equipped with the fibre metric induced by $h_E$.

\vspace{1ex}

 (2)\, Let \begin{eqnarray*}0\longrightarrow (S,\,h_S) \longrightarrow (E,\,h_E) \longrightarrow (Q,\,h_Q)\longrightarrow 0\end{eqnarray*} be a short exact sequence of holomorphic vector bundles on $X$, where $h_S$, resp. $h_Q$, are the induced, resp. quotient, Hermitian fibre metrics on $S$, resp. $Q$.

  The following implications hold:

\vspace{1ex}

(a)\, {\bf if} $(E,\,h_E)$ is  $(\omega,\,\Omega)$-$m$-(semi-)negative at $x$, {\bf then} $(S,\,h_S)$ is $(\omega,\,\Omega)$-$m$-(semi-)negative at $x$;

\vspace{1ex}

(b)\, {\bf if} $(E,\,h_E)$ is  $(\omega,\,\Omega)$-$m$-(semi-)positive at $x$, {\bf then} $(Q,\,h_Q)$ is $(\omega,\,\Omega)$-$m$-(semi-)positive at $x$.

\end{Prop}  

\noindent {\it Proof.}  (1)\, If $V$ and $W$ are $\C$-vector spaces, the following canonical isomorphism of $\C$-vector spaces: \begin{eqnarray*}\mbox{Hom}\,(V,\,W) &\longrightarrow & \mbox{Hom}\,(W^\star,\,V^\star), \\
  T & \longmapsto & T^{\dagger}, \hspace{2ex} \mbox{where} \hspace{1ex} T^{\dagger}(w^\star)(v):= w^\star(Tv), \hspace{2ex} \mbox{for all} \hspace{1ex} w^\star\in W^\star, v\in V,\end{eqnarray*} is the well-known {\it transposition} operator. When acting fibrewise in the context of a $\C$-vector bundle $E$, it induces a canonical isomorphism of $\C$-vector bundles: \begin{eqnarray*}\operatorname{End} (E) \longrightarrow \operatorname{End} (E^\star), \hspace{5ex} T \longmapsto T^{\dagger},\end{eqnarray*} under which the image of the curvature form $i\Theta(D_E)\in C^\infty_{1,\,1}(X,\,\operatorname{End} (E))$ of a connection $D_E$ on $E$ is denoted by $(i\Theta(D_E))^{\dagger}\in C^\infty_{1,\,1}(X,\,\operatorname{End} (E^\star))$. It is standard (see e.g. [Dem97, chapter V]) that the curvature forms of the Chern connections on $(E,\,h_E)$ and $(E^\star,\,h_E^\star)$ are related as follows: \begin{eqnarray}\label{eqn:curvature_E-star_dagger}\bigg(i\Theta_{h_E}(E)\bigg)^{\dagger} = -i\Theta_{h_{E^\star}}(E^\star).\end{eqnarray}

This means that \begin{eqnarray*}i\Theta_{h_{E^\star}}(E^\star)(\xi_1,\,\xi_2)(s_2^\star)(s_1) = -s_2^\star\bigg(i\Theta_{h_E}(E)(\xi_1,\,\xi_2)(s_1)\bigg)\end{eqnarray*} for all $\xi_1,\xi_2\in T^{1,\,0}_xX$ and all $s_2^\star\in E^\star_x$, all $s_1\in E_x$ and all $x\in X$.

If for every $s\in E_x$, we set $s^\star:=\langle\,\cdot\,,\,s\rangle_{h_E}\in E^\star_x$, this amounts to \begin{eqnarray*}\bigg\langle i\Theta_{h_{E^\star}}(E^\star)(\xi_1,\,\xi_2)(s_2^\star),\, s_1^\star\bigg\rangle_{h_{E^\star}} = - \bigg\langle s_2,\, i\Theta_{h_E}(E)(\xi_1,\,\xi_2)(s_1)\bigg\rangle_{h_E}.\end{eqnarray*}

Multiplying by $\omega^{m-1}\wedge\Omega$, we get: \begin{eqnarray*}\label{eqn:Chern-curvatures_E-E-star_bis}\bigg\{\bigg(i\Theta_{h_{E^\star}}(E^\star)\wedge\omega^{m-1}\wedge\Omega\bigg)(s_2^\star),\, s_1^\star\bigg\}_{h_{E^\star}} = - \bigg\{s_2,\, \bigg(i\Theta_{h_E}(E)\wedge\omega^{m-1}\wedge\Omega\bigg)(s_1)\bigg\}_{h_E}\end{eqnarray*} for all $s_2^\star\in E^\star_x$, $s_1\in E_x$ and $x\in X$.

The statement under (1) follows from this by taking $s_1=s_2$ and using the definitions.

\vspace{1ex}

(2)\, Let $\beta\in C^\infty_{1,\,0}(X,\,\mbox{Hom}\,(S,\,Q))$ be the second fundamental form of $S$ in $E$ and let $\beta^\star\in C^\infty_{0,\,1}(X,\,\mbox{Hom}\,(Q,\,S))$ be its adjoint with respect to the pointwise inner product induced by $h_S$ and $h_Q$. Denoting by $i\Theta_{h_E}(E)_{|S}$ (resp. $i\Theta_{h_E}(E)_{|Q}$) the $(S\to S)$-part (resp. $(Q\to Q)$-part) of $i\Theta_{h_E}(E)\in C^\infty_{1,\,1}(X,\,\operatorname{End}  (E))$ with respect to the (in general non-holomorphic) $C^\infty$ splitting $E\simeq S\oplus Q$ induced by $h_E, h_S, h_Q$, the Chern curvature forms of $(S,\,h_S)$ and $(Q,\,h_Q)$ are given by the well-known formulae (\ref{eqn:curvatures_S-Q_exact-seq}) - see e.g. [Dem97, V-$\S14$].

Writing $\beta = \sum\limits_{j=1}^n dz_j\otimes\beta_j$, with $\beta_j\in\mbox{Hom}\,(S,\,Q)$, in a coordinate neighbourhood of $x$, we get $\beta^\star = \sum\limits_{k=1}^n d\bar{z}_k\otimes\beta_k^\star$, where $\beta_k^\star\in\mbox{Hom}\,(Q,\,S)$ is the adjoint of $\beta_k$. Meanwhile, let $\gamma>0$ be the (unique) positive definite $(1,\,1)$-form on $X$ such that $\gamma_{n-1} = \gamma^{n-1}/(n-1)! = \omega^{m-1}\wedge\Omega$. Choosing the local holomorphic coordinates $z_1,\dots , z_n$ centered at $x$ such that $\gamma = \sum_{l=1}^n idz_l\wedge d\bar{z}_l$ at $x$, we get: \begin{eqnarray*}\omega^{m-1}\wedge\Omega = \sum\limits_{l=1}^n \widehat{idz_l\wedge d\bar{z}_l}    \hspace{6ex} \mbox{at} \hspace{1ex} x,\end{eqnarray*} where $\widehat{idz_l\wedge d\bar{z}_l}$ stands for the wedge product of all $idz_j\wedge d\bar{z}_j$ with $j = 1, \dots , l-1, l+1,\dots , n$.

Consequently, we get the following expressions (the second of which holds only at $x$): \begin{eqnarray*}i\beta^\star\wedge\beta\wedge\omega^{m-1}\wedge\Omega = -\sum\limits_{j,k=1}^n idz_j\wedge d\bar{z}_k\wedge\omega^{m-1}\wedge\Omega\otimes\beta_k^\star\beta_j = - dV_n\otimes\sum_{j=1}^n\beta_j^\star\beta_j,\end{eqnarray*} where $dV_n:= idz_1\wedge d\bar{z}_1\wedge\dots\wedge idz_n\wedge d\bar{z}_n$. Hence, \begin{eqnarray*}\bigg\{\bigg(i\beta^\star\wedge\beta\wedge\omega^{m-1}\wedge\Omega\bigg)\,u_x,\,u_x\bigg\}_{h_S(x)} = -\sum_{j=1}^n\bigg\langle(\beta_j^\star\beta_j)\,u_x,\,u_x\bigg\rangle_{h_S(x)}\,dV_n = -\sum_{j=1}^n\bigg|\beta_j\,u_x\bigg|^2_{h_Q(x)}\,dV_n \leq 0.\end{eqnarray*} From this and the equality $i\Theta_{h_S}(S)\wedge\omega^{m-1}\wedge\Omega = i\Theta_{h_E}(E)_{|S}\wedge\omega^{m-1}\wedge\Omega + i\beta^\star\wedge\beta\wedge\omega^{m-1}\wedge\Omega$, we conclude that \begin{eqnarray*}\bigg\{\bigg(i\Theta_{h_S}(S)\wedge\omega^{m-1}\wedge\Omega\bigg)\,u_x,\,u_x\bigg\}_{h_S(x)}\leq \bigg\{\bigg(i\Theta_{h_E}(E)_{|S}\wedge\omega^{m-1}\wedge\Omega\bigg)\,u_x,\,u_x\bigg\}_{h_S(x)}\end{eqnarray*}  for each $u_x\in S_x$. This proves (a).

(b) is proved in a similar way after observing that \begin{eqnarray*}\bigg\{\bigg(i\beta\wedge\beta^\star\wedge\omega^{m-1}\wedge\Omega\bigg)\,v_x,\,v_x\bigg\}_{h_Q(x)} = \sum_{j=1}^n\bigg|\beta_j^\star\,v_x\bigg|^2_{h_S(x)}\,dV_n \geq 0\end{eqnarray*} for each $v_x\in Q_x$ and after using the general equality $i\Theta_{h_Q}(Q)\wedge\omega^{m-1}\wedge\Omega = i\Theta_{h_E}(E)_{|Q}\wedge\omega^{m-1}\wedge\Omega + i\beta\wedge\beta^\star\wedge\omega^{m-1}\wedge\Omega$.  \hfill $\Box$

\vspace{2ex}

Next, we show that one of the above positivity notions behaves well under pullbacks as well.

\begin{Prop}\label{Prop:m-pos_pullback} Let $\pi:Y\longrightarrow X$ be a holomorphic submersion between complex manifolds with $\mbox{dim}_\C X =n$ and $\mbox{dim}_\C Y = N = n+l$. Suppose there exists $\gamma\in C^\infty_{1,\,1}(Y,\,\R)$ such that the restriction $\gamma_{|\pi^{-1}(z)}$ is positive definite everywhere on the fibre $Y_z:=\pi^{-1}(z)$ for every $z\in X$.

  Let $(E,\,h)\longrightarrow X$ be a Hermitian holomorphic vector bundle of rank $r\geq 1$ over $X$. Let $\omega$ be a Hermitian metric on $X$, $m\in\{1,\dots , n\}$, $x\in X$ and $\Omega\in\Lambda^{n-m,\,n-m}T^\star_x X$ an $\R$-valued $(n-m,\,n-m)$-form at $x$ such that $\Omega> 0$ (metrically weakly).

\vspace{1ex}  

 (a)\, Let $\rho\in\Lambda^{1,\,1}T^\star_x X$ be the unique positive definite $(1,\,1)$-form at $x\in X$ such that $\rho^{n-1} = \omega^{m-1}\wedge\Omega$. For every $\varepsilon>0$, set $\tilde\rho_\varepsilon:=\pi^\star\rho + \varepsilon\gamma_{|\pi^{-1}(x)}\in C^\infty_{1,\,1}(\pi^{-1}(x),\,\R)$.

  Then, the following implication holds:

  \vspace{1ex}

  {\bf if} $(E,\,h)$ is  $(\omega,\,\Omega)$-$m$-(semi-)positive at $x\in X$, {\bf then} $(\pi^\star E,\,\pi^\star h)$ is $\tilde\rho_\varepsilon$-$(N-1)$-(semi-)positive at every point $y\in\pi^{-1}(x)$ for every $\varepsilon>0$.

  \vspace{1ex}  

  (b)\, For every $\varepsilon>0$, set $\widetilde\omega_\varepsilon:=\pi^\star\omega + \varepsilon\gamma \in C^\infty_{1,\,1}(Y,\,\R)$ and $\widetilde\Omega_\varepsilon:=\pi^\star\Omega + \varepsilon\gamma^{n-m}_{|\pi^{-1}(x)}\in C^\infty_{n-m,\,n-m}(\pi^{-1}(x),\,\R)$.

  Then, if $n-m\leq l$, the following implication holds:

  \vspace{1ex}

  {\bf if} $(E,\,h)$ is  $(\omega,\,\Omega)$-$m$-(semi-)positive at $x\in X$ and $\omega$-$n$-(semi-)positive at $x\in X$, {\bf then} $(\pi^\star E,\,\pi^\star h)$ is $(\widetilde\omega_\varepsilon,\,\widetilde\Omega_\varepsilon)$-$(l+m)$-(semi-)positive at every point $y\in\pi^{-1}(x)$ for every $\varepsilon>0$.

\end{Prop}  

\noindent {\it Proof.} (a)\, The hypothesis means that, for every $u\in E_x\setminus\{0\}$ (resp. for every $u\in E_x$), we have: \begin{eqnarray*}\bigg\{\bigg(i\Theta_h(E)\wedge\rho^{n-1}\bigg)\,u,\,u\bigg\}_h>0 \hspace{2ex} (\mbox{resp.} \hspace{1ex} \geq 0).\end{eqnarray*} In other words, the linear map $i\Theta_h(E)\wedge\rho^{n-1}/dV_\rho:E_x\longrightarrow E_x$ is positive definite (resp. semi-positive definite) with respect to the inner product $\langle\,\cdot\,,\,\cdot\,\rangle_{h(x)}$.

On the other hand, $i\Theta_{\pi^\star h}(\pi^\star E) = \pi^\star(i\Theta_h(E))$ and \begin{eqnarray*}i\Theta_{\pi^\star h}(\pi^\star E)\wedge\tilde\rho_\varepsilon^{N-1} = {N-1 \choose n-1}\,\varepsilon^l\,\pi^\star\bigg(i\Theta_h(E)\wedge\rho^{n-1}\bigg)\wedge(\gamma_{|\pi^{-1}(x)})^l\end{eqnarray*} at every point $y\in\pi^{-1}(x)$. Indeed, $l$ is the dimension of the fibres $Y_z=\pi^{-1}(z)$ as complex manifolds, so $l$ is the top power that the vertical $(1,\,1)$-form $\gamma_{|\pi^{-1}(x)}$ can non-trivially be raised to on $\pi^{-1}(x)$, while $n-1$ is the top power of $\rho$ that can non-trivially be multiplied with $i\Theta_h(E)$ on the $n$-dimensional manifold $X$. Therefore, this equality yields: \begin{eqnarray*}\bigg\{\bigg(i\Theta_{\pi^\star h}(\pi^\star E)\wedge\tilde\rho_\varepsilon^{N-1}\bigg)\,u,\,u\bigg\}_{\pi^\star h} =  {N-1 \choose n-1}\,\varepsilon^l\,\bigg\{\bigg(i\Theta_h(E)\wedge\rho^{n-1}\bigg)\,u,\,u\bigg\}_h\wedge(\gamma_{|\pi^{-1}(x)})^l\end{eqnarray*} for every $y\in\pi^{-1}(x)$ and every $u\in(\pi^\star E)_y = E_x$.

    Now, $\gamma_{|\pi^{-1}(x)}>0$ as a scalar form on the fibre $\pi^{-1}(x)$, so the above equality shows that the (semi-)positive definiteness of $i\Theta_h(E)\wedge\rho^{n-1}/dV_\rho:E_x\longrightarrow E_x$ implies the same property for $i\Theta_{\pi^\star h}(\pi^\star E)\wedge\tilde\rho_\varepsilon^{N-1}/dV_{\tilde\rho_\varepsilon}:(\pi^\star E)_y\longrightarrow(\pi^\star E)_y$. This proves the contention.

    \vspace{1ex}

    (b)\, Given the restrictions on the exponents of the horizontal form $\pi^\star\omega$ and the vertical form $\gamma_{|\pi^{-1}(x)}$ in terms of $n, m, l$ that were explained in the proof of (a), we get: \begin{eqnarray*} i\Theta_{\pi^\star h}(\pi^\star E)\wedge\widetilde\omega_\varepsilon^{l+m-1}\wedge\widetilde\Omega_\varepsilon & = & \\
 \pi^\star\bigg[{l+m-1 \choose m-1}\,\varepsilon^l\,\bigg(i\Theta_h(E)\wedge\omega^{m-1}\wedge\Omega\bigg) & + & {l+m-1 \choose n-1}\,\varepsilon^{l+1-n+m}\,\bigg(i\Theta_h(E)\wedge\omega^{n-1}\bigg)\bigg]\wedge(\gamma_{|\pi^{-1}(x)})^l\end{eqnarray*} at every point $y\in\pi^{-1}(x)$. This proves the contention. \hfill $\Box$

\vspace{2ex}

Next, recall the following standard construction (see e.g. [Dem97]). Let $E$ be a holomorphic vector bundle of rank $r\geq 1$ over an $n$-dimensional complex manifold $X$. One denotes by $\pi:\Proj(E^\star)\longrightarrow X$ the associated projectivised bundle whose fibre at every point $x\in X$ is $\Proj(E^\star)_x = \Proj(E^\star_x)$, the set of $\C$-lines in $E^\star_x$. One then defines the {\it tautological hyperplane vector subbundle} $S$ of $\pi^\star E\longrightarrow\Proj(E^\star)$ whose fibre at every point $(x,\,[\xi])\in\Proj(E^\star)$ (where $x\in X$ and $\xi\in E^\star_x\setminus\{0\}$ can be viewed as a non-zero $\C$-linear map $\xi:E_x\longrightarrow\C$) is the following $\C$-hyperplane of $E_x$: \begin{eqnarray*}S_{(x,\,[\xi])}:=\ker\xi\subset E_x = (\pi^\star E)_{(x,\,[\xi])}.\end{eqnarray*} One then considers the quotient $\C$-line bundle ${\cal O}_E(1):=\pi^\star E/S\longrightarrow\Proj(E^\star)$ and gets the following short exact sequence of holomorphic vector bundles over $\Proj(E^\star)$: \begin{eqnarray}\label{eqn:short-exact_O(1)}0\longrightarrow S\longrightarrow\pi^\star E\longrightarrow{\cal O}_E(1)\longrightarrow 0.\end{eqnarray} If, moreover, $E$ is given a $C^\infty$ Hermitian fibre metric $h$, the pullback fibre metric $\pi^\star h$ on $\pi^\star E$ induces, via (\ref{eqn:short-exact_O(1)}), a quotient fibre metric $h_Q$ on ${\cal O}_E(1)$.

\vspace{1ex}

Putting together part (2)(b) of Proposition \ref{Prop:m-pos_properties} and Proposition \ref{Prop:m-pos_pullback}, we get the following:

\begin{Cor}\label{Cor:m-pos_O(1)} Let $(E,\,h)\longrightarrow X$ be a Hermitian holomorphic vector bundle of rank $r\geq 1$ over an $n$-dimensional complex manifold $X$. Let $\omega$ be a Hermitian metric on $X$, $m\in\{1,\dots , n\}$, $x\in X$ and $\Omega\in\Lambda^{n-m,\,n-m}T^\star_x X$ an $\R$-valued $(n-m,\,n-m)$-form at $x$ such that $\Omega> 0$ (metrically weakly).

  Set $N:=n+r-1$ and let $\gamma=\omega_{FS}\in C^\infty_{1,\,1}(\Proj(E^\star),\,\R)$ be the form induced by $h$ whose restriction $\gamma_{|\pi^{-1}(z)}$ is the Fubini-Study metric of $\pi^{-1}(z)=\Proj(E^\star_z)$ for every $z\in X$.

\vspace{1ex}

  Then, the following implications hold (notation is as in Proposition \ref{Prop:m-pos_pullback}): 

\vspace{1ex}

(a)\, {\bf if} $(E,\,h)$ is  $(\omega,\,\Omega)$-$m$-(semi-)positive at $x$, {\bf then}  $\bigg({\cal O}_E(1),\,h_Q\bigg)$ is $\tilde\rho_\varepsilon$-$(N-1)$--(semi-)positive at every point $(x,\,[\xi])\in\pi^{-1}(x)$ for every $\varepsilon>0$;

\vspace{1ex}

(b)\, {\bf if} $n-m\leq r-1$, $(E,\,h)$ is  $(\omega,\,\Omega)$-$m$-(semi-)positive  at $x$ and $\omega$-$n$-(semi-)positive at $x\in X$, {\bf then} $\bigg({\cal O}_E(1),\,h_Q\bigg)$ is $(\widetilde\omega_\varepsilon,\,\widetilde\Omega_\varepsilon)$-$(r-1+m)$-(semi-)positive at every $(x,\,[\xi])\in\pi^{-1}(x)$ for every $\varepsilon>0$.

\end{Cor}

\noindent {\it Proof.} The only thing to notice before applying said results to deduce this one is the existence of $\gamma=\omega_{FS}$ as in the statement. Indeed, once a $C^\infty$ Hermitian fibre metric $h$ has been fixed on $E$, for each $z\in X$ the Hermitian vector space $(E^\star_z,\,h_z^\star)$ induces the Fubini-Study metric $\omega_{FS,\,z}$ on the associated complex projective space $\Proj(E^\star_z)$. When $z$ is made to vary in $X$, the family $(\omega_{FS,\,z})_{z\in X}$ of forms (which depend in a $C^\infty$ way on $z\in X$ because $h$ does) defines a $C^\infty$ $(1,\,1)$-form $\omega_{FS}$ on the complex manifold $\Proj(E^\star)$.

Equivalently, the $C^\infty$ $(1,\,1)$-form $\omega_{FS}$ on $\Proj(E^\star)$ can be seen as the ``vertical'' component of the curvature form $i\Theta_{h_Q}({\cal O}_E(1))$. In particular, $\omega_{FS}$ and $i\Theta_{h_Q}({\cal O}_E(1))$ have the same restriction to each fibre $\Proj(E^\star_z)$ with $z\in X$. Indeed, the restriction of ${\cal O}_E(1)$ to $\Proj(E^\star_z)\simeq\Proj^{r-1}$ identifies with the line bundle ${\cal O}_{\Proj^{r-1}}(1)$ whose curvature form (with respect to the induced fibre metric) is the usual Fubini-Study metric of the projective space $\Proj^{r-1}$. More precisely, $\omega_{FS} = i\beta\wedge\beta^\star$ as $(1,\,1)$-forms on $\Proj(E^\star)$, where $\beta\in C^\infty_{1,\,0}(\Proj(E^\star),\,\mbox{Hom}\,(S,\,{\cal O}_E(1))$ is the second fundamental form of the short exact sequence (\ref{eqn:short-exact_O(1)}) with fibre metrics induced by $h$. Put differently, in the standard formula (cf. e.g. [Dem97, V-$\S14$]): \begin{eqnarray*}i\Theta_{h_Q}({\cal O}_E(1)) = i\Theta_{\pi^\star h}(\pi^\star E)_{|{\cal O}_E(1)} + i\beta\wedge\beta^\star,\end{eqnarray*} the $({\cal O}_E(1)\to{\cal O}_E(1))$-part of the curvature form of $(\pi^\star E,\, \pi^\star h)$ is the ``horizontal'' component (with respect to the fibration $\pi:\Proj(E^\star)\longrightarrow X$) of the curvature form of $({\cal O}_E(1),\,h_Q)$. \hfill $\Box$

\vspace{2ex}

As an aside, we note that the above reasoning implies the well-known fact that the tautological short exact sequence (\ref{eqn:short-exact_O(1)}) never splits holomorphically.

\section{$m$-positive stability: main definitions and basic properties}\label{section:m-pos-stability_def-prop_stability} We propose notions of stability for holomorphic vector bundles that generalise the classical Mumford-Takemoto slope stability theory.

\subsection{Coherent subsheaves and stability functions}\label{subsection:subsheaves_Z-functions}

 Let $E$ be a holomorphic vector bundle of rank $r\geq 1$ on an $n$-dimensional complex manifold $X$. It is well known that one has the following

\begin{Fact} ({\bf standard})\label{Fact:canonical-section} Every coherent subsheaf ${\cal F}\subset{\cal O}(E)$ of rank $s\in\{1,\dots , r-1\}$ of the sheaf of germs of holomorphic sections of $E$ induces a {\bf canonical} non-vanishing global holomorphic section \begin{eqnarray*}\sigma_{{\cal F}}\in H^0(X,\,\Lambda^s E\otimes\det{\cal F}^\star),\end{eqnarray*} where $\det{\cal F}:=(\Lambda^s{\cal F})^{\vee\vee}$ is the determinant line bundle of ${\cal F}$ and $\det{\cal F}^\star$ is its dual line bundle.

\end{Fact}      

\noindent {\it Proof.} Since ${\cal F}$ is a coherent subsheaf of the locally free sheaf ${\cal O}(E)$, it is torsion-free. Since any torsion-free coherent sheaf on a complex manifold is locally free outside an analytic subset of codimension $\geq 2$, there exists such an analytic subset $S\subset X$ such that the restriction ${\cal F}_{|X\setminus S}$ is locally free. Thus, there exists a rank-$s$ holomorphic vector subbundle $F$ of $E_{|X\setminus S}$ such that ${\cal F}_{|X\setminus S} = O(F)$.

Taking $s$-th exterior powers in the inclusion $F\hookrightarrow E_{|X\setminus S}$ of vector bundles over $X\setminus S$, we get the canonical injective vector bundle morphism $\Lambda^s F = \det{\cal F}_{|X\setminus S}\hookrightarrow\Lambda^sE_{|X\setminus S}$ over $X\setminus S$. Since $\mbox{codim}_X S\geq 2$, the classical Hartogs extension theorem for holomorphic functions in codimension $\geq 2$ yields a unique extension of this morphism across $S$: \begin{eqnarray*}\det{\cal F}\longrightarrow\Lambda^sE.\end{eqnarray*} This morphism (of holomorphic vector bundles over $X$) obtained by extension across $S$ remains injective as a sheaf morphism. Indeed, recall that the determinant line bundle $\det{\cal F} = (\Lambda^s{\cal F})^{\vee\vee}$ is defined on the whole of $X$ as the double dual of the top exterior power of ${\cal F}$. It is the unique extension across $S$ of the holomorphic line bundle $\Lambda^sF$ (a priori defined on $X\setminus S$).

Tensoring by the holomorphic line bundle $\det{\cal F}^\star$, we obtain a morphism of locally free ${\cal O}_X$-modules over $X$: \begin{eqnarray*}{\cal O}_X\longrightarrow{\cal O}(\Lambda^sE\otimes\det{\cal F}^\star).\end{eqnarray*} The sought-after canonical global section $\sigma_{{\cal F}}\in H^0(X,\,\Lambda^s E\otimes\det{\cal F}^\star)$ is the image of $1$ under this last morphism. \hfill $\Box$

\vspace{2ex}

In the above context, we propose the following

\begin{Def}\label{Def:Z-function} Let $X$ be a compact complex manifold with $\mbox{dim}_\C X =n$ equipped with a Hermitian metric $\omega$ and a form $\Omega\in C^\infty_{n-m,\,n-m}(X,\,\R)$ such that $\Omega>0$ (metrically weakly) at every point of $X$ for some $m\in\{1,\dots , n\}$. Let $(E,\,h)$ be a holomorphic vector bundle of rank $r\geq 1$ on $X$ endowed with a $C^\infty$ Hermitian fibre metric.

  With every coherent subsheaf ${\cal F}\subset{\cal O}(E)$ of rank $s\in\{1,\dots , r-1\}$, we associate the {\bf stability function} $Z_{\omega,\,\Omega,\,h}^{(m)}({\cal F}):X\longrightarrow\R$ defined by \begin{eqnarray*}Z_{\omega,\,\Omega,\,h}^{(m)}({\cal F}):=\bigg\{\frac{i\Theta_{\Lambda^sh\otimes h^{H-E}_{\det{\cal F}^\star}}(\Lambda^s E\otimes\det{\cal F}^\star)\wedge\omega^{m-1}\wedge\Omega}{dV_\omega}\,\sigma_{{\cal F}},\,\sigma_{{\cal F}}\bigg\}_{\Lambda^sh\otimes h^{H-E}_{\det{\cal F}^\star}},\end{eqnarray*} where $\Lambda^sh$ is the Hermitian fibre metric induced on $\Lambda^s E$ by $h$, $h^{H-E}_{\det{\cal F}^\star} = \det h_{{\cal F}}^\star$ is the (unique, up to a positive multiplicative constant) $(\omega,\,\Omega)$-Hermite-Einstein fibre metric on the line bundle $\det{\cal F}^\star$ and $\sigma_{{\cal F}}\in H^0(X,\,\Lambda^s E\otimes\det{\cal F}^\star)$ is the canonical section of Fact \ref{Fact:canonical-section}.

  The {\bf normalised stability function} $\widehat{Z}_{\omega,\,\Omega,\,h}^{(m)}({\cal F}):X\longrightarrow\R$ is defined to be \begin{eqnarray*}\widehat{Z}_{\omega,\,\Omega,\,h}^{(m)}({\cal F}):= \frac{Z_{\omega,\,\Omega,\,h}^{(m)}({\cal F})}{|\sigma_{{\cal F}}|^2_{\Lambda^sh\otimes h^{H-E}_{\det{\cal F}^\star}}}.\end{eqnarray*}

\end{Def}  

Note that the curvature forms $i\Theta_{h^{H-E}_{\det{\cal F}^\star}}(\det{\cal F}^\star)$ and $i\Theta_{\Lambda^sh\otimes h^{H-E}_{\det{\cal F}^\star}}(\Lambda^s E\otimes\det{\cal F}^\star)$ are independent of the choice of positive multiplicative constant for the fibre metric $h^{H-E}_{\det{\cal F}^\star} = \det h_{{\cal F}}^\star$. Indeed, such a constant contributes an additive constant to the local weights of the fibre metric, hence it contributes the $i\partial\bar\partial$ of this additive constant, which is 0, to the resulting curvature form. We infer that the stability function $Z_{\omega,\,\Omega,\,h}^{(m)}({\cal F})$ is defined only up to a positive multiplicative constant, while its normalised counterpart $\widehat{Z}_{\omega,\,\Omega,\,h}^{(m)}({\cal F})$ is independent of the choice of $(\omega,\,\Omega)$-Hermite-Einstein fibre metric on the line bundle $\det{\cal F}^\star$.

Separately, note that the stability function associated with ${\cal F}\subset{\cal O}(E)$ in Definition \ref{Def:Z-function} is obtained by evaluating on the canonical section $\sigma_{{\cal F}}$ the quadratic form against which the $(\omega,\,\Omega)$-$m$-(semi-)positivity of the vector bundle $\Lambda^s E\otimes\det{\cal F}^\star$ would have to be tested under Definition \ref{Def:m-pos}.

\begin{Prop}\label{Prop:Z-function_alternative} In the setting of Definition \ref{Def:Z-function}, suppose that $\omega$ is K\"ahler and $\partial\bar\partial\Omega = 0$.

  Then, for any coherent subsheaf ${\cal F}\subset{\cal O}(E)$, the associated function falls into one of the following two cases: \begin{eqnarray}\label{eqn:Z-function_alternative}\mbox{either}\hspace{2ex} Z_{\omega,\,\Omega,\,h}^{(m)}({\cal F})\equiv 0 \hspace{5ex}\mbox{or}\hspace{5ex} Z_{\omega,\,\Omega,\,h}^{(m)}({\cal F})(x_0)>0 \hspace{1ex}\mbox{for some point}\hspace{1ex} x_0\in X.\end{eqnarray}

  In either case, it has the property: \begin{eqnarray}\label{eqn:Z-function_integral}\int\limits_X Z_{\omega,\,\Omega,\,h}^{(m)}({\cal F})\,dV_\omega \geq 0.\end{eqnarray}

  Moreover, if $\int_XZ_{\omega,\,\Omega,\,h}^{(m)}({\cal F})\,dV_\omega = 0$, then $D_{\Lambda^sh\otimes h^{H-E}_{\det{\cal F}^\star}}\sigma_{{\cal F}}\equiv 0$ (i.e. $\sigma_{{\cal F}}$ is parallel with respect to the Chern connection $D_{\Lambda^sh\otimes h^{H-E}_{\det{\cal F}^\star}} = D_{\Lambda^sh\otimes h^{H-E}_{\det{\cal F}^\star}}' + \bar\partial$ of $(\Lambda^s E\otimes\det{\cal F}^\star,\,\Lambda^sh\otimes h^{H-E}_{\det{\cal F}^\star})$).

\end{Prop}

\noindent {\it Proof.} We drop the subscripts specifying the fibre metrics to improve readability. Replacing $E$ with $\Lambda^s E\otimes\det{\cal F}^\star$ and $s$ with $\sigma_{{\cal F}}$ in the proof of Theorem \ref{The:Kobayashi-vanishing_gen}, we get, as in that proof, the following equality of $\R$-valued $(n,\,n)$-forms on $X$: \begin{eqnarray}\label{eqn:Z-function_alternative_proof_1}i\bar\partial\partial\{\sigma_{{\cal F}},\,\sigma_{{\cal F}}\}\wedge\omega^{m-1}\wedge\Omega & = & Z_{\omega,\,\Omega,\,h}^{(m)}({\cal F})\,dV_\omega - i\,\{D'\sigma_{{\cal F}},\,D'\sigma_{{\cal F}}\}\wedge\omega^{m-1}\wedge\Omega.\end{eqnarray}

As argued in the proof of Theorem \ref{The:Kobayashi-vanishing_gen}, we always have: \begin{eqnarray}\label{eqn:Z-function_alternative_proof_2}- i\,\{D'\sigma_{{\cal F}},\,D'\sigma_{{\cal F}}\}\wedge\omega^{m-1}\wedge\Omega\leq 0\end{eqnarray} at every point of $X$. Therefore, if $Z_{\omega,\,\Omega,\,h}^{(m)}({\cal F})$ were $\leq 0$ everywhere on $X$ and $<0$ at some point of $X$, equality (\ref{eqn:Z-function_alternative_proof_1}) would imply (by the ellipticity argument used in the proof of Theorem \ref{The:Kobayashi-vanishing_gen}) that $\sigma_{{\cal F}} = 0$ at every point of $X$, a contradiction. This proves the dichotomy (\ref{eqn:Z-function_alternative}).

  Meanwhile, integrating (\ref{eqn:Z-function_alternative_proof_1}) we get the first equality below: \begin{eqnarray*}\int\limits_X Z_{\omega,\,\Omega,\,h}^{(m)}({\cal F})\,dV_\omega & = & \int\limits_X i\bar\partial\partial\{\sigma_{{\cal F}},\,\sigma_{{\cal F}}\}\wedge\omega^{m-1}\wedge\Omega + \int\limits_X i\,\{D'\sigma_{{\cal F}},\,D'\sigma_{{\cal F}}\}\wedge\omega^{m-1}\wedge\Omega \\
    & = & \int\limits_X i\,\{D'\sigma_{{\cal F}},\,D'\sigma_{{\cal F}}\}\wedge\omega^{m-1}\wedge\Omega \geq 0,\end{eqnarray*} where the second equality follows from the Stokes theorem (thanks to our assumptions $d\omega = 0$ and $\partial\bar\partial\Omega = 0$) and the inequality follows from the pointwise non-negativity of the integrand (see (\ref{eqn:Z-function_alternative_proof_2})). This proves (\ref{eqn:Z-function_integral}).

   Moreover, we see that if $\int_XZ_{\omega,\,\Omega,\,h}^{(m)}({\cal F})\,dV_\omega = 0$, then $i\,\{D'\sigma_{{\cal F}},\,D'\sigma_{{\cal F}}\}\wedge\omega^{m-1}\wedge\Omega \equiv 0$ since we always have $i\,\{D'\sigma_{{\cal F}},\,D'\sigma_{{\cal F}}\}\wedge\omega^{m-1}\wedge\Omega \geq 0$ at every point of $X$. As we argued at the end of the proof of Theorem \ref{The:Kobayashi-vanishing_gen} for that $s$, this implies that $D'\sigma_{{\cal F}}\equiv 0$. The last statement follows. \hfill $\Box$

\vspace{2ex}

  Property (\ref{eqn:Z-function_alternative}) leaves the door open to the function $Z_{\omega,\,\Omega,\,h}^{(m)}({\cal F})$ assuming negative values outside a neighbourhood of the point $x_0\in X$. The main definition we propose in this section singles out the vector bundles $E$ for which this kind of negativity does not occur for any coherent subsheaf. Intuitively, these are the vector bundles $E$ that have at least as much positivity as each of their coherent subsheaves ${\cal F}$ in the canonical direction (defined by the canonical section $\sigma_{{\cal F}}$). The next definition is not affected by the positive constant that multiplies the stability function $Z_{\omega,\,\Omega,\,h}^{(m)}({\cal F})$ if a different choice of $(\omega,\,\Omega)$-Hermite-Einstein fibre metric on $\det{\cal F}^\star$ is made. One can also replace $Z_{\omega,\,\Omega,\,h}^{(m)}({\cal F})$ by its normalised (and uniquely defined) counterpart $\widehat{Z}_{\omega,\,\Omega,\,h}^{(m)}({\cal F})$.

\begin{Def}\label{Def:m-pos-stability} Let $E\longrightarrow X$ be a holomorphic vector bundle of rank $r\geq 1$ over a compact complex $n$-dimensional manifold and let $m\in\{1,\dots , n\}$. Suppose there exists a K\"ahler metric $\omega$ on $X$. Let $\Omega\in C^\infty_{n-m,\,n-m}(X,\,\R)$ be a form such that \begin{eqnarray*}\Omega>0 \hspace{1ex} \mbox{(metrically weakly) \hspace{3ex} and} \hspace{3ex} \partial\bar\partial\Omega = 0.\end{eqnarray*}

 (a)\, The vector bundle $E\longrightarrow X$ is said to be {\bf $(\omega,\,\Omega)$-$m$-positively stable} (respectively {\bf $(\omega,\,\Omega)$-$m$-positively semi-stable}) if for every coherent subsheaf ${\cal F}\subset{\cal O}(E)$ of any rank $s\in\{1,\dots , r-1\}$, there exists a $C^\infty$ Hermitian fibre metric $h$ on $E$ such that \begin{eqnarray}\label{eqn:m-pos-stability} Z_{\omega,\,\Omega,\,h}^{(m)}({\cal F})> 0 \hspace{5ex} (\mbox{respectively} \hspace{1ex} Z_{\omega,\,\Omega,\,h}^{(m)}({\cal F})\geq 0) \hspace{5ex} \mbox{at every point of}\hspace{1ex} X, \end{eqnarray} where $Z_{\omega,\,\Omega,\,h}^{(m)}({\cal F})$ is the function introduced in Definition \ref{Def:Z-function}.

\vspace{1ex}

(b)\, A $C^\infty$ Hermitian fibre metric $h$ on $E$ is said to be {\bf uniformly $(\omega,\,\Omega)$-$m$-positively stable} (respectively {\bf uniformly $(\omega,\,\Omega)$-$m$-positively semi-stable}) if for every coherent subsheaf ${\cal F}\subset{\cal O}(E)$ of any rank $s\in\{1,\dots , r-1\}$, property (\ref{eqn:m-pos-stability}) holds.   

\vspace{1ex}

When no confusion is likely, we may omit mentioning ``$(\omega,\,\Omega)$-$m$-positively'' and refer to these notions as simply {\bf uniformly stable}, respectively {\bf uniformly semi-stable}.

\end{Def}

\vspace{2ex}

Since we will sometimes change the holomorphic structure $D''$ on a given $C^\infty$ complex vector bundle $E$, we can refer to property (b) in the above Definition \ref{Def:m-pos-stability} by any of the equivalent formulations spelt out in $\S$\ref{subsection:introd_m-positive-stability} of the introduction.


\begin{Cor-Def}\label{Cor-Def:Z_unif-lower-bounf_F} Let $(E,\,D'',\,h)$ be a Hermitian holomorphic vector bundle of rank $r\geq 1$ over a {\bf compact} complex $n$-dimensional manifold $X$ supposed to carry differential forms $\omega,\,\Omega$ satisfying the properties of Definition \ref{Def:m-pos-stability}.

  For every coherent subsheaf ${\cal F}\subset{\cal O}(E)$ of rank $s\in\{1,\dots , r-1\}$, the real constant \begin{eqnarray}\label{eqn:h-unif-constant}c_h({\cal F}):=\min\limits_X \widehat{Z}_{\omega,\,\Omega,\,h}^{(m)}({\cal F})\in\R \end{eqnarray} is called the {\bf $h$-uniformity constant} of  ${\cal F}$. Moreover, one has the following equivalences: \begin{eqnarray*}\label{eqn:Z_unif-lower-bounf_F}(E,\,D'',\,h) \hspace{1ex} \mbox{is {\bf uniformly stable}} & \iff & c_h({\cal F})>0  \hspace{1ex} \forall\,{\cal F}\subset{\cal O}(E)\hspace{1ex} \mbox{proper coherent subsheaf}\\
    (E,\,D'',\,h) \hspace{1ex} \mbox{is {\bf uniformly semi-stable}} & \iff & c_h({\cal F})\geq 0  \hspace{1ex} \forall\,{\cal F}\subset{\cal O}(E)\hspace{1ex} \mbox{proper coherent subsheaf}.\end{eqnarray*}

\end{Cor-Def}  

\noindent {\it Proof.} The function $\widehat{Z}_{\omega,\,\Omega,\,h}^{(m)}({\cal F}):X\longrightarrow\R$ is $C^\infty$, in particular continuous, hence it attains its minimum on the compact manifold $X$, showing that the constant $c_h({\cal F})$ is well defined.

The stated equivalences follow at once from the definitions. \hfill $\Box$

\subsection{Quotient sheaves and stability functions}\label{subsection:quotient-sheaves_W-functions}

We shall now see that this discussion can be run equivalently with quotient sheaves instead of subsheaves. We start by observing the analogue of Fact \ref{Fact:canonical-section}.

\begin{Fact}\label{Fact:canonical-section_quotient} Let $E$ be a holomorphic vector bundle of rank $r\geq 1$ on an $n$-dimensional complex manifold $X$. Let ${\cal F}\subset{\cal O}(E):={\cal E}$ be a coherent subsheaf of rank $s\in\{1,\dots , r-1\}$.

  Then, the quotient morphism ${\cal E}\to{\cal E}/{\cal F}$ induces a {\bf canonical} non-vanishing global holomorphic section \begin{eqnarray*}\mu_{{\cal F}}\in H^0(X,\,\Lambda^{r-s} E^\star\otimes\det({\cal E}/{\cal F})),\end{eqnarray*} where $\det({\cal E}/{\cal F}):=\Lambda^{r-s}({\cal E}/{\cal F})^{\vee\vee}$ is the determinant line bundle of ${\cal E}/{\cal F}$.

\end{Fact}      

\noindent {\it Proof.} We use the notation of the proof of Fact \ref{Fact:canonical-section}. In particular, the short exact sequence of torsion-free coherent sheaves over $X$: \begin{eqnarray*}0\longrightarrow{\cal F}\longrightarrow{\cal E}\longrightarrow{\cal E}/{\cal F}\longrightarrow 0\end{eqnarray*} induces a short exact sequence of holomorphic vector bundles over $X\setminus S$: \begin{eqnarray*}0\longrightarrow F\longrightarrow E_{|X\setminus S}\longrightarrow E_{|X\setminus S}/F\longrightarrow 0.\end{eqnarray*}

  Taking $(r-s)$-th exterior powers in the surjection $E_{|X\setminus S}\longrightarrow E_{|X\setminus S}/F$, we get the canonical surjective vector bundle morphism $\Lambda^{r-s}E_{|X\setminus S}\longrightarrow\Lambda^{r-s}(E_{|X\setminus S}/F) = \det({\cal E}/{\cal F})_{|X\setminus S}$. Since $\mbox{codim}_X S\geq 2$, the classical Hartogs extension theorem for holomorphic functions in codimension $\geq 2$ yields a unique extension of this morphism across $S$: \begin{eqnarray*}\Lambda^{r-s}E\longrightarrow\det({\cal E}/{\cal F}).\end{eqnarray*} Tensoring by the holomorphic vector bundle $\Lambda^{r-s}E^\star$, we obtain a morphism of locally free ${\cal O}_X$-modules over $X$: \begin{eqnarray*}{\cal O}_X\longrightarrow{\cal O}(\Lambda^{r-s}E^\star\otimes\det({\cal E}/{\cal F})).\end{eqnarray*} The sought-after canonical global section $\mu_{{\cal F}}\in H^0(X,\,\Lambda^{r-s} E^\star\otimes\det({\cal E}/{\cal F}))$ is the image of $1$ under this last morphism. \hfill $\Box$

  \vspace{2ex}

  By analogy with Definition \ref{Def:Z-function}, we propose the following

  \begin{Def}\label{Def:Q-function} The setting is the same as in Definition \ref{Def:Z-function}. With every quotient sheaf ${\cal E}/{\cal F}$ of a coherent subsheaf ${\cal F}\subset{\cal O}(E)$ of rank $s\in\{1,\dots , r-1\}$, we associate the function
  $W_{\omega,\,\Omega,\,h}^{(m)}({\cal E}/{\cal F}):X\longrightarrow\R$ defined by \begin{eqnarray*}W_{\omega,\,\Omega,\,h}^{(m)}({\cal E}/{\cal F}):=\bigg\{\frac{i\Theta_{\Lambda^{r-s}h^\star\otimes\det h_{{\cal E}/{\cal F}}}(\Lambda^{r-s} E^\star\otimes\det({\cal E}/{\cal F}))\wedge\omega^{m-1}\wedge\Omega}{dV_\omega}\,\mu_{{\cal F}},\,\mu_{{\cal F}}\bigg\}_{\Lambda^{r-s}h^\star\otimes\det h_{{\cal E}/{\cal F}}},\end{eqnarray*} where $\Lambda^{r-s}h^\star$ is the Hermitian fibre metric induced on $\Lambda^{r-s} E^\star$ by $h$, $\det h_{{\cal E}/{\cal F}} = h^{H-E}_{\det({\cal E}/{\cal F})}$ is the (unique, up to a positive multiplicative constant) $(\omega,\,\Omega)$-Hermite-Einstein fibre metric on the line bundle $\det({\cal E}/{\cal F})$ and $\mu_{{\cal F}}\in H^0(X,\,\Lambda^{r-s} E^\star\otimes\det({\cal E}/{\cal F}))$ is the canonical section of Fact \ref{Fact:canonical-section_quotient}.

\end{Def}

  The following observation shows that considering quotient sheaves is equivalent to considering subsheaves when discussing {\it $m$-positively (semi-)stable} properties of vector bundles.

\begin{Prop}\label{Prop:subsheaves_quotient-sheaves_equiv} Under the assumptions of Definitions \ref{Def:Z-function} and \ref{Def:Q-function}, for every coherent subsheaf ${\cal F}\subset{\cal O}(E)$ of rank $s\in\{1,\dots , r-1\}$, we have: \begin{eqnarray*}Z_{\omega,\,\Omega,\,h}^{(m)}({\cal F}) = W_{\omega,\,\Omega,\,h}^{(m)}({\cal E}/{\cal F})\end{eqnarray*} at every point of $X$.

\end{Prop}

\noindent {\it Proof.} There is a canonical isomorphism of vector bundles: \begin{eqnarray*}\Lambda^{r-s}E^\star\otimes\det E\simeq\Lambda^sE,\end{eqnarray*} the globalised version of the canonical isomorphism of $\C$-vector spaces (for a given $\C$-vector space $V$ of dimension $r$): \begin{eqnarray*}\Lambda^{r-s}V^\star\otimes\Lambda^r V\simeq\Lambda^sV,  \hspace{5ex} (\alpha,\,\eta)\longmapsto\alpha\lrcorner\eta,\end{eqnarray*} defined by the contraction of any $\eta\in\Lambda^r V$ by any $\alpha\in\Lambda^{r-s}V^\star$. This is the extension by linearity of the composition of contractions: $(\alpha_1\wedge\dots\wedge\alpha_{r-s})\lrcorner\eta := \alpha_1\lrcorner\dots\lrcorner\bigg(\alpha_{r-s-1}\lrcorner(\alpha_{r-s}\lrcorner\eta)\bigg)$, where $\alpha_1,\dots ,\alpha_{r-s}\in V^\star$ are arbitrary.

Now, since $\det{\cal E}\simeq\det{\cal F}\otimes\det({\cal E}/{\cal F})$, the above vector bundle isomorphism leads to: \begin{eqnarray*}\Lambda^{r-s} E^\star\otimes\det({\cal E}/{\cal F})\simeq\bigg(\Lambda^{r-s}E^\star\otimes\det E\bigg)\otimes\det {\cal F}^\star\simeq\Lambda^sE\otimes\det {\cal F}^\star.\end{eqnarray*}

Moreover, the canonical sections $\mu_{{\cal F}}\in H^0(X,\,\Lambda^{r-s} E^\star\otimes\det({\cal E}/{\cal F}))$ and $\sigma_{{\cal F}}\in H^0(X,\,\Lambda^s E\otimes\det{\cal F}^\star)$ of Facts \ref{Fact:canonical-section_quotient} and \ref{Fact:canonical-section} correspond to each other under the above vector bundle isomorphism. To see this, note that everything is pointwise, so we can work at a pregiven point $x\in X\setminus S$. Let $Q$ be the holomorphic vector bundle over $X\setminus S$ whose sheaf of germs of holomorphic sections is the restriction of ${\cal E}/{\cal F}$. Let $\{e_1,\dots , e_r\}$ be a local holomorphic frame of $E$ in a neighbourhood $U$ of $x$ such that $\{e_1,\dots , e_s\}$ is a holomorphic frame of $F$ on $U$. Then $Q$ has frame $\{\hat{e}_{s+1},\dots , \hat{e}_r\}$ and $\det Q$ has frame $\{\hat{e}_{s+1}\wedge\dots\wedge\hat{e}_r\}$ over $U$. In these frames, we have: \begin{eqnarray*}\sigma_{\cal F} = (e_1\wedge\dots\wedge e_s)\otimes(e_1\wedge\dots\wedge e_s)^\star \hspace{5ex}\mbox{and}\hspace{5ex} \mu_{\cal F} = (e^{s+1}\wedge\dots\wedge e^r)\otimes(\hat{e}_{s+1}\wedge\dots\wedge\hat{e}_r),\end{eqnarray*} where $\{e^1,\dots , e^r\}$ is the frame of $E^\star$ over $U$ that is dual to $\{e_1,\dots , e_r\}$. Since $\det E$ is generated by $e_1\wedge\dots\wedge e_r = (e_1\wedge\dots\wedge e_s)\otimes(\hat{e}_{s+1}\wedge\dots\wedge\hat{e}_r)$ (where the last equality is given by the canonical line bundle isomorphism $\det E\simeq\det F\otimes\det Q)$) and since \begin{eqnarray*}(e^{s+1}\wedge\dots\wedge e^r)\lrcorner(e_1\wedge\dots\wedge e_r) = \pm\, e_1\wedge\dots\wedge e_s,\end{eqnarray*} the image of $e^{s+1}\wedge\dots\wedge e^r$ under the canonical isomorphism $\Lambda^{r-s}E^\star\otimes\det E\simeq\Lambda^sE$ is $\pm\, e_1\wedge\dots\wedge e_s$. It follows that the image of $\mu_{\cal F}$ under the canonical isomorphism $\Lambda^{r-s} E^\star\otimes\det({\cal E}/{\cal F})\simeq\Lambda^sE\otimes\det {\cal F}^\star$ is $\pm\,\sigma_{\cal F}$.

From this and from the fact that all the fibre metrics that feature in the definitions of $Z_{\omega,\,\Omega,\,h}^{(m)}({\cal F})$ and $W_{\omega,\,\Omega,\,h}^{(m)}({\cal E}/{\cal F})$ are induced canonically from $h$ and from the $(\omega,\,\Omega)$-Hermite-Einstein property of the determinant line bundles, we conclude that  $Z_{\omega,\,\Omega,\,h}^{(m)}({\cal F}) = W_{\omega,\,\Omega,\,h}^{(m)}({\cal E}/{\cal F})$. \hfill $\Box$

\subsection{Simplicity and $m$-positive stability}\label{subsection:simplicity-stability}

The next result asserts that the holomorphic structure of any $(\omega,\,\Omega)$-$m$-positively stable vector bundle is {\it simple}, in the sense that its only holomorphic endomorphisms are the homotheties.

\begin{The}\label{The:m-pos-stable_implies_simple} Let $E\longrightarrow X$ be a holomorphic vector bundle of rank $r\geq 1$ over a compact complex manifold $X$ with $\mbox{dim}_\C X =n$ and let $m\in\{1,\dots , n\}$. Suppose there exists a K\"ahler metric $\omega$ on $X$. Let $\Omega\in C^\infty_{n-m,\,n-m}(X,\,\R)$ be a form such that \begin{eqnarray*}\Omega>0 \hspace{1ex} \mbox{(metrically weakly) \hspace{3ex} and} \hspace{3ex} \partial\bar\partial\Omega = 0.\end{eqnarray*}

  If $E$ is {\bf $(\omega,\,\Omega)$-$m$-positively stable}, then for every holomorphic endomorphism $f\in H^0(X,\,\operatorname{End} E)$ there exists a constant $\lambda\in\C$ such that $f = \lambda\,\mbox{Id}_E$.

\end{The}

\noindent {\it Proof.} Suppose that $E$ is $(\omega,\,\Omega)$-$m$-positively stable. Let $f\in H^0(X,\,\operatorname{End} E)$.

\vspace{1ex}

$\bullet$ We first prove that $f$ is either {\bf identically zero} or an {\bf automorphism} of $E$.

Let ${\cal F}:=f(E)$. Then, the torsion-free coherent sheaf ${\cal F}$ on $X$ is both a subsheaf and a quotient sheaf of ${\cal E}:={\cal O}(E)$. Indeed, we have short exact sequences of coherent sheaves on $X$: \begin{eqnarray}\label{eqn:seq_F-quotient_sub}0\longrightarrow\ker f\longrightarrow{\cal E}\stackrel{f}{\longrightarrow}{\cal F}\longrightarrow 0 \hspace{5ex}\mbox{and}\hspace{5ex} 0\longrightarrow{\cal F}\longrightarrow{\cal E}\longrightarrow {\cal E}/{\cal F}\longrightarrow 0.\end{eqnarray}

Let $s$ be the rank of ${\cal F}$ and let $A:=\Lambda^sE\otimes\det{\cal F}^\star$. This is a holomorphic vector bundle on $X$. We denote by $A^\star = \Lambda^sE^\star\otimes\det{\cal F}$ the dual vector bundle. Regarding ${\cal F}$ as a subsheaf of ${\cal E}$ (via the second exact sequence in (\ref{eqn:seq_F-quotient_sub})), we get a canonical section $\sigma_{{\cal F}}\in H^0(X,\,A)$ thanks to Fact \ref{Fact:canonical-section}. Regarding the same ${\cal F}$ as a quotient sheaf of ${\cal E}$ (via the first exact sequence in (\ref{eqn:seq_F-quotient_sub})), we get a canonical section $\mu_{{\cal F}}\in H^0(X,\,A^\star)$ thanks to Fact \ref{Fact:canonical-section_quotient} (in which the role of ${\cal E}/{\cal F}$ is played by ${\cal F}$).

Now, fix any $C^\infty$ Hermitian fibre metric $h$ on $E$. Let $\Lambda^sh$, respectively $\Lambda^sh^\star$, be the induced Hermitian fibre metric on $\Lambda^sE$, respectively $\Lambda^sE^\star$. Meanwhile, let $\det h_{\cal F}$ be the (unique, up to a positive multiplicative constant) $(\omega,\,\Omega)$-Hermite-Einstein fibre metric on the line bundle $\det\,{\cal F}$ and let $\det h_{\cal F}^\star$ be the induced fibre metric on the dual line bundle $\det\,{\cal F}^\star$. Finally, we denote by $h_A$, respectively $h_{A^\star}$, the $C^\infty$ Hermitian fibre metric induced on $A$, respectively $A^\star$, by $h$ and $\det h_{\cal F}$. 

From Definitions \ref{Def:Q-function} and \ref{Def:Z-function} , we get the following equalities at every point of $X$: \begin{eqnarray}\label{eqn:Q-Z_A-A-star}\nonumber\frac{1}{|\mu_{\cal F}|_{h_{A^\star}}^2}\,W_{\omega,\,\Omega,\,h}^{(m)}({\cal F}) & = & \frac{1}{|\mu_{\cal F}|_{h_{A^\star}}^2}\,\bigg\{\frac{i\Theta_{h_{A^\star}}(A^\star)\wedge\omega^{m-1}\wedge\Omega}{dV_\omega}\,\mu_{{\cal F}},\,\mu_{{\cal F}}\bigg\}_{h_{A^\star}}  \\
 -\frac{1}{|\sigma_{\cal F}|_{h_A}^2}\,Z_{\omega,\,\Omega,\,h}^{(m)}({\cal F}) & = & -\frac{1}{|\sigma_{\cal F}|_{h_A}^2}\,\bigg\{\frac{i\Theta_{h_A}(A)\wedge\omega^{m-1}\wedge\Omega}{dV_\omega}\,\sigma_{{\cal F}},\,\sigma_{{\cal F}}\bigg\}_{h_A}.\end{eqnarray}

\vspace{2ex}

Now, the sections $\sigma_{{\cal F}}\in H^0(X,\,A)$ and $\mu_{{\cal F}}\in H^0(X,\,A^\star)$ are canonically dual to each other, by construction. In particular, $\mu_{{\cal F}}(\sigma_{{\cal F}}) = 1$ at every point of $X$.

Since $A$ and $A^\star$ have been equipped with {\it dual} Hermitian fibre metrics $h_A$ and $h_{A^\star}$, we have: \begin{eqnarray}\label{eqn:norms_mu-sigma}|\mu_{\cal F}|_{h_{A^\star}}^2 = \frac{1}{|\sigma_{\cal F}|_{h_A}^2}  \hspace{5ex} \mbox{at every point of}\hspace{1ex} X.\end{eqnarray} Indeed, fix an arbitrary point $x\in X$ and a local holomorphic frame $\{e_1,\dots , e_r\}$ of $E$ in a neighbourhood $U$ of $x$ such that $\{e_1,\dots , e_s\}$ is a holomorphic frame of $F$ on $U$ and $\{e_1,\dots , e_r\}$ is orthogonal at $x$ for the fibre metric $h$. Then, \begin{eqnarray*}\sigma_{\cal F} = (e_1\wedge\dots \wedge e_s)\otimes\tau^\star \hspace{5ex} \mbox{and} \hspace{5ex} \mu_{\cal F} = (e^1\wedge\dots \wedge e^s)\otimes\tau,\end{eqnarray*} where $\tau:=e_1\wedge\dots \wedge e_s$ and $\{e^1,\dots , e^r\}$ is the frame of $E^\star$ dual to $\{e_1,\dots , e_r\}$.

Moreover, let $\sharp:(A^\star,\,h_{A^\star})\longrightarrow(A,\,h_A)$ be the metric identification of these two dual vector bundles. This yields: $\sharp\mu_{\cal F} = \sigma_{\cal F}/|\sigma_{\cal F}|_{h_A}^2$. On the other hand, the curvature forms of $(A^\star,\,h_{A^\star})$ and $(A,\,h_A)$ are related by the identity: \begin{eqnarray*}\{i\Theta_{h_{A^\star}}(A^\star)\,v,\,v\}_{h_{A^\star}} = -\{i\Theta_{h_A}(A)\,\sharp v,\,\sharp v\}_{h_A},   \hspace{5ex} v\in A^\star_z,\,z\in X.\end{eqnarray*} In particular, taking $v=\mu_{\cal F}$, we get: \begin{eqnarray*}\{i\Theta_{h_{A^\star}}(A^\star)\,\mu_{\cal F},\,\mu_{\cal F}\}_{h_{A^\star}} = -\frac{1}{|\sigma_{\cal F}|_{h_A}^4}\,\{i\Theta_{h_A}(A)\,\sigma_{\cal F},\,\sigma_{\cal F}\}_{h_A}, \hspace{5ex} \mbox{at every point of}\hspace{1ex} X.\end{eqnarray*} Thanks to (\ref{eqn:norms_mu-sigma}), this translates to \begin{eqnarray*}\frac{1}{|\mu_{\cal F}|_{h_{A^\star}}^2}\{i\Theta_{h_{A^\star}}(A^\star)\,\mu_{\cal F},\,\mu_{\cal F}\}_{h_{A^\star}} = -\frac{1}{|\sigma_{\cal F}|_{h_A}^2}\,\{i\Theta_{h_A}(A)\,\sigma_{\cal F},\,\sigma_{\cal F}\}_{h_A} \hspace{5ex} \mbox{at every point of}\hspace{1ex} X.\end{eqnarray*} Multiplying this identity by $\omega^{m-1}\wedge\Omega$ and then dividing the resulting $(n,\,n)$-form by $dV_\omega$, we see that the right-hand sides of the two equalities in (\ref{eqn:Q-Z_A-A-star}) coincide. Thus, we get the following identity at every point of $X$: \begin{eqnarray}\label{eqn:Q-Z_A-A-star_bis}\frac{1}{|\mu_{\cal F}|_{h_{A^\star}}^2}\,W_{\omega,\,\Omega,\,h}^{(m)}({\cal F}) = -\frac{1}{|\sigma_{\cal F}|_{h_A}^2}\,Z_{\omega,\,\Omega,\,h}^{(m)}({\cal F}).\end{eqnarray}

Now, recall that if $s\in\{1,\dots , r-1\}$, the function $Z_{\omega,\,\Omega,\,h}^{(m)}({\cal F})$ satisfies properties (\ref{eqn:Z-function_alternative}) and (\ref{eqn:Z-function_integral}). Similarly, the function $W_{\omega,\,\Omega,\,h}^{(m)}({\cal F})$, which thanks to Proposition \ref{Prop:subsheaves_quotient-sheaves_equiv} equals $Z_{\omega,\,\Omega,\,h}^{(m)}(\ker f)$ -- see the first exact sequence in (\ref{eqn:seq_F-quotient_sub}), satisfies the same properties (\ref{eqn:Z-function_alternative}) and (\ref{eqn:Z-function_integral}). In particular, property (\ref{eqn:Z-function_integral}) for both functions, combined with (\ref{eqn:Q-Z_A-A-star_bis}), implies \begin{eqnarray}\label{eqn:integrals_Z-Q_zero}\int\limits_XZ_{\omega,\,\Omega,\,h}^{(m)}({\cal F}) = \int\limits_XW_{\omega,\,\Omega,\,h}^{(m)}({\cal F}) =0\end{eqnarray} for every $C^\infty$ Hermitian fibre metric $h$ on $E$. Note that we haven't used any stability assumption on $E$ so far.

  However, if $E$ is {\it $(\omega,\,\Omega)$-$m$-positively stable}, Definition \ref{Def:m-pos-stability} ensures the existence of a $C^\infty$ Hermitian fibre metric $h$ on $E$ such that \begin{eqnarray*}Z_{\omega,\,\Omega,\,h}^{(m)}({\cal F})> 0 \hspace{2ex} \mbox{at every point of}\hspace{1ex} X.\end{eqnarray*} This implies $\int_X Z_{\omega,\,\Omega,\,h}^{(m)}({\cal F})>0$, contradicting (\ref{eqn:integrals_Z-Q_zero}). This contradiction has been reached in the case where $s=\mbox{rank}\,({\cal F})\in\{1,\dots , r-1\}$.

  We conclude that, if $E$ is {\it $(\omega,\,\Omega)$-$m$-positively stable}, we have either $s=0$ (which amounts to the map $f:E\to E$ vanishing identically) or $s=r$ (in which case the map $f:E\to E$ is bijective). Concerning the last statement, recall that $s=r$ is equivalent to $f$ being an isomorphism outside an analytic subset of $X$. In particular, $f:E\to E$ has maximal rank almost everywhere, hence $\det f\in H^0(X,\,(\det E)^\star\otimes\det E) =H^0(X,\,{\cal O}_X)$ is non-zero almost everywhere. Since $\det f$ is a global holomorphic function on the compact complex manifold $X$, it must be constant. Moreover, this constant cannot be zero because $\det f$ is generically non-zero. It follows that, thanks to the compactness of $X$, $f:E\to E$ is indeed an automorphism of $E$ when $s=r$.  

\vspace{1ex}

  $\bullet$ We conclude in the standard way. Fix an arbitrary point $x\in X$ and let $\lambda\in\C$ be an eigenvalue of $f_x:E_x\longrightarrow E_x$. Then, by the first step, the section $f-\lambda\,\mbox{Id}_E\in H^0(X,\,\operatorname{End} E)$ would be an automorphism of $E$ if it did not vanish identically. However, by construction, $(f-\lambda\,\mbox{Id}_E)_x:E_x\longrightarrow E_x$ is not an isomorphism. Therefore, $f-\lambda\,\mbox{Id}_E\equiv 0$.  \hfill $\Box$  

\section{Moduli spaces}\label{section:moduli}

Let us start by recalling a few standard definitions and results (cf. e.g. [LT95] whose notation and terminology we largely adopt). With a $C^\infty$ complex vector bundle $E$ of rank $r$ over a compact complex $n$-dimensional manifold $X$, one associates the {\it complex gauge group} \begin{eqnarray*}{\cal G}^{\C} = C^\infty(X,\,\mbox{Aut}\,E) = GL(E)\subset C^\infty(X,\,\operatorname{End} E)\end{eqnarray*} of $C^\infty$ automorphisms of $E$ and its right action \begin{eqnarray}\label{eqn:G-C_action_A}\bar{\cal A}(E)\times{\cal G}^{\C} \longrightarrow \bar{\cal A}(E), \hspace{5ex} (D'',\, T)\longmapsto T^{-1}\circ D''\circ T:=D''_T,\end{eqnarray} on the set \begin{eqnarray*}\bar{\cal A}(E):=\bigg\{D'':C^\infty(X,\,E)\longrightarrow C^\infty_{0,\,1}(X,\,E)\,\mid\, D'' \hspace{1ex} \mbox{is a $(0,\,1)$-connection on $E$}\bigg\},\end{eqnarray*} where by the differential operator $D''$ being a $(0,\,1)$-connection on $E$ we mean that it satisfies the Leibniz rule: \begin{eqnarray*}D''(fs) = (\bar\partial f)\otimes s + f\,D''s\end{eqnarray*} for all $f\in C^\infty(X,\,\C)$ and all $s\in C^\infty(X,\,E)$. This operator, originally defined in degree $0$, extends naturally to operators $D'':C^\infty_k(X,\,E)\longrightarrow C^\infty_{k+1}(X,\,E)$ defined in any degree $k$ by extending the Leibniz rule to $D''(f\wedge s) = \bar\partial f\wedge s + (-1)^l f\wedge D''s$, for all $f\in C^\infty_l(X,\,\C)$ and all $s\in C^\infty_k(X,\,E)$.

  It is equally standard that every $(0,\,1)$-connection $D''$ (denoted sometimes $D''_E$, for emphasis) on $E$ induces a $(0,\,1)$-connection $D''_{\operatorname{End} E}$ on the $\C$-vector bundle of endomorphisms of $E$ by requiring the following Leibniz rule to hold in degree $0$: \begin{eqnarray}\label{eqn:Leibniz-rule_D_End}D''_E (Ts) = \bigg(D''_{\operatorname{End} E}T\bigg)\,s + T\,D''_E s,\end{eqnarray} for all $T\in C^\infty(X,\,\operatorname{End}E)$ and all $s\in C^\infty(X,\,E)$. This naturally extends to all the degrees by inserting the coefficient $(-1)^k$ in front of the last term when $T\in C^\infty_k(X,\,\operatorname{End}E)$.

  A $(0,\,1)$-connection $D''\in\bar{\cal A}(E)$ on $E$ is said to be {\it integrable} if $D^{''2} = 0$. (It is standard that the integrability property is equivalent to $D''$ defining a holomorphic structure on $E$ in the sense that there exist local holomorphic trivialisations with holomorphic transition isomorphisms.) On the other hand, $D''$ is said to be {\it simple} if the only {\it holomorphic} endomorphisms of $E$ for the induced $(0,\,1)$-connection $D''_{\operatorname{End}E}\in\bar{\cal A}(\operatorname{End}E)$ are the homotheties: \begin{eqnarray*}\ker D''_{\operatorname{End}E} = \C\,\mbox{Id}_E.\end{eqnarray*} The subset of {\it integrable} $(0,\,1)$-connections on $E$ ($=$ holomorphic structures on $E$) is denoted by ${\cal H}(E)\subset\bar{\cal A}(E)$, while the subset of {\it simple} $(0,\,1)$-connections on $E$ is denoted by $\bar{\cal A}^s(E)\subset\bar{\cal A}(E)$. It is immediate to check that both ${\cal H}(E)$ and $\bar{\cal A}^s(E)$ are invariant under the ${\cal G}^{\C}$-action (\ref{eqn:G-C_action_A}). Hence, so is their intersection \begin{eqnarray*}{\cal H}^s(E):= {\cal H}(E)\cap\bar{\cal A}^s(E)\subset\bar{\cal A}(E),\end{eqnarray*} the subset of simple holomorphic structures on $E$.

  \vspace{1ex}

  Finally, recall that if $D_1''$ and $D_2''$ are holomorphic structures on a $C^\infty$ $\C$-vector bundle $E$ over $X$, a $C^\infty$ vector bundle morphism $T:E_1:=(E,\,D_1'')\longrightarrow E_2:=(E,\,D_2'')$ is said to be {\bf holomorphic} if $D''_{\operatorname{Hom} (E_1,\,E_2)}T =0$. We denote by $D''_{\operatorname{Hom} (E_1,\,E_2)}$ the holomorphic structure ($=$ the integrable $(0,\,1)$-connection) induced on the $C^\infty$ morphism bundle $\operatorname{Hom} (E_1,\,E_2)$ by $D_1''$ and $D_2''$ by requiring the analogue of (\ref{eqn:Leibniz-rule_D_End}) to hold, namely: \begin{eqnarray}\label{eqn:Leibniz-rule_D_Hom}D''_2 (Ts) = \bigg(D''_{\operatorname{Hom} (E_1,\,E_2)}T\bigg)\,s + T\,D''_1 s,\end{eqnarray} for all $T\in C^\infty(X,\,{Hom} (E_1,\,E_2))$ and all $s\in C^\infty(X,\,E_1)$. In particular, $T:E_1:\longrightarrow E_2$ is a holomorphic isomorphism of vector bundles if it is holomorphic and a $C^\infty$ isomorphism.

  The following statement is standard (see e.g. [LT95, Remark 4.3.2.]) and may help explain why the ${\cal G}^{\cal C}$-action (\ref{eqn:G-C_action_A}) was defined in that way. For the reader's convenience, we give a proof.

\begin{Prop}\label{Prop:hol-isom_equiv-action} For any $D_1'', D_2''\in{\cal H}(E)$ and any $T\in{\cal G}^{\cal C}$, the following equivalence holds: \begin{eqnarray*}D_2'' = T^{-1}\circ D_1''\circ T \iff T^{-1}:(E,\,D_1'')\longrightarrow(E,\,D_2'') \hspace{2ex} \mbox{is a {\bf holomorphic} isomorphism}.\end{eqnarray*}

\end{Prop}

\noindent {\it Proof.} Applying formula (\ref{eqn:Leibniz-rule_D_Hom}) with $T^{-1}$ in place of $T$, we get: \begin{eqnarray*}D''_2 (T^{-1}s) = \bigg(D''_{\operatorname{Hom} (E_1,\,E_2)}T^{-1}\bigg)\,s + T^{-1}\,D''_1 s, \hspace{5ex} s\in C^\infty(X,\,E).\end{eqnarray*} This shows that $D''_{\operatorname{Hom} (E_1,\,E_2)}T^{-1} = 0$ (i.e. $T^{-1}:E_1\longrightarrow E_2$ is holomorphic) if and only if $D_2''\circ T^{-1} = D_1''\circ T$. This proves the contention.  \hfill $\Box$

\subsection{Moduli spaces of $m$-positively stable vector bundles}\label{subsection:moduli_m-pos-stable}

  Using this terminology, Theorem \ref{The:m-pos-stable_implies_simple} can be restated as follows:

  \begin{Cor}\label{Cor:m-pos-stable_implies_simple} Let $X$ be a compact complex manifold with $\mbox{dim}_\C X =n$ and let $m\in\{1,\dots , n\}$. Suppose there exists a K\"ahler metric $\omega$ on $X$. Let $\Omega\in C^\infty_{n-m,\,n-m}(X,\,\R)$ be a form such that \begin{eqnarray*}\Omega>0 \hspace{1ex} \mbox{(metrically weakly) \hspace{3ex} and} \hspace{3ex} \partial\bar\partial\Omega = 0.\end{eqnarray*}

    Then, for every $C^\infty$ $\C$-vector bundle $E$ of rank $r\geq 1$ on $X$, the following inclusion holds:  \begin{eqnarray*}{\cal H}_{m\mbox{-stable}}(E)\subset{\cal H}^s(E),\end{eqnarray*} where ${\cal H}_{m\mbox{-stable}}(E) = {\cal H}_{\omega-\Omega-m\mbox{-stable}}(E)\subset{\cal H}(E)$ is the subset of holomorphic structures $D''$ on $E$ such that the holomorphic vector bundle $(E,\,D'')$ is {\bf $(\omega,\,\Omega)$-$m$-positively stable}.

\end{Cor}

Next, we prove the invariance  under the ${\cal G}^{\C}$-action (\ref{eqn:G-C_action_A}) of the subset ${\cal H}_{m\mbox{-stable}}(E)\subset{\cal H}^s(E)$. To this end, we need the following preliminary observation.

\begin{Lem}\label{Lem:gauge-transf_curvature} Let $E$ be a $C^\infty$ $\C$-vector bundle of rank $r\geq 1$ over a compact complex manifold $X$ with $\mbox{dim}_\C X = n$. Let $D''\in{\cal H}(E)$ be a holomorphic structure and $h$ a $C^\infty$ Hermitian fibre metric on $E$. We denote by $D_h=D_h' + D''$ the Chern connection of $(E,\,D'',\,h)$.

  For every $T\in{\cal G}^{\C}$, let $D''_T:=T^{-1}\circ D''\circ T$ be the holomorphic structure on $E$ induced by $(D'',\,T)$ and let $h_T$ be the $C^\infty$ Hermitian fibre metric on $E$ induced by $(h,\,T)$ that makes $T^{-1}:(E,\,h)\longrightarrow(E,\,h_T)$ {\bf isometric}: \begin{eqnarray*}\langle u,\,v\rangle_{h_T}:= \langle Tu,\,Tv\rangle_h, \hspace{5ex} u,v\in E_x,\hspace{1ex} x\in X.\end{eqnarray*} We denote by $D_{h_T}= D'_{h_T} + D''_T$ the Chern connection of $(E,\,D_T'',\,h_T)$.

  Then, the curvature forms of $(E,\, D_h)$ and $(E,\,D_{h_T})$ are related in the following way: \begin{eqnarray}\label{eqn:gauge-transf_curvature}i\Theta(D_{h_T}) = T^{-1}\circ i\Theta(D_h)\circ T.\end{eqnarray}

\end{Lem}  

  \noindent {\it Proof.} $\bullet$ We first compute $D'_{h_T}$ in terms of $D'_h$ and $T$.

  For all $u,v\in E_x$ (for some fixed $x\in X$), the Leibniz rule for $h_T$ yields: \begin{eqnarray}\label{eqn:del-Leibniz_proof-curvatures_1}\partial\{u,\,v\}_{h_T} = \bigg\{D'_{h_T}u,\,v\bigg\}_{h_T} + \bigg\{u,\,D''_Tv\bigg\}_{h_T} = \bigg\{TD'_{h_T}u,\,Tv\bigg\}_h + \bigg\{Tu,\,D''(Tv)\bigg\}_h,\end{eqnarray} where for the last equality we used the identity $T\circ D''_T = D''\circ T$. Meanwhile, the Leibniz rule for $h$ yields the second equality below: \begin{eqnarray}\label{eqn:del-Leibniz_proof-curvatures_2}\partial\{u,\,v\}_{h_T}  & = & \partial\{Tu,\,Tv\}_h = \bigg\{D'_h(Tu),\,Tv\bigg\}_h + \bigg\{Tu,\,D''(Tv)\bigg\}_h.\end{eqnarray}

We deduce from (\ref{eqn:del-Leibniz_proof-curvatures_1}) and (\ref{eqn:del-Leibniz_proof-curvatures_2}) that $\{TD'_{h_T}u,\,Tv\}_h = \{D'_h(Tu),\,Tv\}_h$ for all $u,v$. This yields the identity $D'_{h_T} = T^{-1}\circ D'_h\circ T$. Thus, $D'_{h_T}$ transforms from $D'_h$ in the same way as $D''_T$ transforms from $D''$. In particular, we get: \begin{eqnarray}\label{Lem:gauge-transf_Chern-connection}D_{h_T} = T^{-1}\circ D_h\circ T.\end{eqnarray} 
  
$\bullet$ We can now compare the curvature forms using their definitions and (\ref{Lem:gauge-transf_Chern-connection}). We get: \begin{eqnarray*}i\Theta(D_{h_T}) = i\,D_{h_T}^2 = i\,(T^{-1}\circ D_h\circ T)\circ(T^{-1}\circ D_h\circ T) = i\,T^{-1}\circ D_h^2\circ T = T^{-1}\circ i\Theta(D_h)\circ T.  \end{eqnarray*} This proves (\ref{eqn:gauge-transf_curvature}).  \hfill $\Box$

\vspace{1ex}

Note that the change of fibre metric from $h$ to $h_T$ (in addition to the change of holomorphic structure from $D''$ to $D''_T$ under the ${\cal G}^{\cal C}$-action) was key in Lemma \ref{Lem:gauge-transf_curvature} and will continue to be so in the proof of Proposition \ref{Prop:m-pos-stability_G-C-action_invariance}. For the sake of context, we recall the following standard result (see e.g. [LT95, Remark 1.1.20., (i)]):

\begin{Prop}\label{Prop:D'_T_D'_same-metric} In the setting of Lemma \ref{Lem:gauge-transf_curvature}, let $D_T = D_T' + D_T''$ be the Chern connection of $(E,\,D_T'',\,h)$, where $D''_T:=T^{-1}\circ D''\circ T$ for some $T\in{\cal G}^{\C}$. We continue to denote by $D_h=D_h' + D''$ the Chern connection of $(E,\,D'',\,h)$.

  Then, $D_T' = T^\star\circ D_h'\circ(T^\star)^{-1}$, where $T^\star$ is the adjoint of $T$ with respect to the fibre metric $h$.

\end{Prop}

\vspace{1ex}

This standard result shows that if the fibre metric $h$ is kept fixed on $E$ when the holomorphic structure $D''$ is changed to $D''_T$ under the ${\cal G}^{\cal C}$-action, the $(1,\,0)$-part $D_T'$ of the resulting Chern connection $D_T$ transforms by a different rule compared to the $(0,\,1)$-part. (Note, however, that if ${\cal G}^{\cal C}$ is replaced by the {\it $h$-unitary gauge group} $U(E,\,h)$ -- the group of $h$-unitary elements $T\in C^\infty(X,\,\mbox{Aut}\,E)$ -- then $T^\star = T^{-1}$ for every $T\in U(E,\,h)$, so $D'$ transforms similarly to $D''$.) By contrast, the formula $D'_{h_T} = T^{-1}\circ D'_h\circ T$ we obtained in the proof of Lemma \ref{Lem:gauge-transf_curvature} when $h$ too is changed to $h_T$ obeys the same transformation law as $D''_T$.

\begin{Prop}\label{Prop:m-pos-stability_G-C-action_invariance} Let $X$ be a compact complex manifold with $\mbox{dim}_\C X =n$ and let $m\in\{1,\dots , n\}$. Suppose there exists a K\"ahler metric $\omega$ on $X$. Let $\Omega\in C^\infty_{n-m,\,n-m}(X,\,\R)$ be a form such that \begin{eqnarray*}\Omega>0 \hspace{1ex} \mbox{(metrically weakly) \hspace{3ex} and} \hspace{3ex} \partial\bar\partial\Omega = 0.\end{eqnarray*}

  Then, for every $C^\infty$ $\C$-vector bundle $E$ of rank $r\geq 1$ on $X$, the subset ${\cal H}_{m\mbox{-stable}}(E) = {\cal H}_{\omega-\Omega-m\mbox{-stable}}(E)\subset{\cal H}(E)$ defined in Corollary \ref{Cor:m-pos-stable_implies_simple} is {\bf invariant} under the ${\cal G}^{\C}$-action (\ref{eqn:G-C_action_A}).


\end{Prop}

\noindent {\it Proof.} Let us fix $T\in{\cal G}^{\C}$ and $D''\in{\cal H}(E)$. Assume that the holomorphic vector bundle $(E,\,D'')$ is {\bf $(\omega,\,\Omega)$-$m$-positively stable}. We have to show that the holomorphic vector bundle $(E,\,D_T'')$ has the same property, where $D_T'' = T^{-1}\circ D''\circ T$.

\vspace{1ex}

$\bullet$ By Proposition \ref{Prop:hol-isom_equiv-action}, $T^{-1}:(E,\,D'')\longrightarrow(E,\,D''_T)$ is a {\it holomorphic} isomorphism of vector bundles. 

Let ${\cal E}$ and ${\cal E}_T$ be the ${\cal O}_X$-modules associated respectively with the isomorphic holomorphic vector bundles $(E,\,D'')$ and $E_T:=(E,\,D_T'')$. Thus, ${\cal E} = \ker D''$ and ${\cal E}_T = \ker D''_T$. Moreover, $T^{-1}:{\cal E}\longrightarrow{\cal E}_T$ is a sheaf isomorphism and every coherent subsheaf ${\cal F}\subset{\cal E}$ induces a coherent subsheaf ${\cal F}_T:=T^{-1}{\cal F}\subset{\cal E}_T$. Conversely, every coherent subsheaf ${\cal F}_T\subset{\cal E}_T$ arises uniquely in this way. 

This can also be seen directly. Indeed, if $s$ is a local section of ${\cal E}$, then $D''s=0$, hence $D''_T(T^{-1}s) = T^{-1}D''(TT^{-1}s) = T^{-1}(D''s) = 0$, so $T^{-1}s$ is a local section of ${\cal E}_T$. Conversely, if $T^{-1}s$ is a local section of ${\cal E}_T$, then $s$ is a local section of ${\cal E}$. 

\vspace{1ex}

$\bullet$ Now, fix an arbitrary coherent subsheaf ${\cal F}\subset{\cal E}$ of rank $s\in\{1,\dots , r-1\}$. Let $A:=\Lambda^sE\otimes\det{\cal F}^\star$ and $A_T:=\Lambda^sE_T\otimes\det{\cal F}_T^\star$ be the induced vector bundles, holomorphic with respect to the holomorphic structure induced by $D''$, respectively $D''_T$. Thus, $A$ and $A_T$ coincide as $C^\infty$ vector bundles and are isomorphic as holomorphic vector bundles under the holomorphic isomorphism \begin{eqnarray}\label{eqn:T-tilde_isom_def}\widetilde{T}^{-1}=(\Lambda^sT^{-1})\otimes\det(T^{-1})^\star : A \longrightarrow A_T.\end{eqnarray}

Let $h$ be an arbitrary $C^\infty$ Hermitian fibre metric on $E$ and let $k$ be an arbitrary $C^\infty$ Hermitian fibre metric on $\det{\cal F}^\star$ (e.g. the unique -- up to a positive multiplicative constant -- $(\omega,\,\Omega)$-Hermitian-Einstein metric). The sheaf isomorphism $T^{-1}:{\cal F}\longrightarrow{\cal F}_T$ induces a holomorphic line bundle isomorphism: \begin{eqnarray*}S:\det{\cal F}^\star\longrightarrow\det{\cal F}_T^\star.\end{eqnarray*} We transport the fibre metric $k$ on $\det{\cal F}^\star$ through $S$ to a $C^\infty$ Hermitian fibre metric $k_T$ on $\det{\cal F}_T^\star$ by requiring \begin{eqnarray*}S:(\det{\cal F}^\star,\,k)\longrightarrow(\det{\cal F}_T^\star,\,k_T)\end{eqnarray*} to be isometric. (Note that if $k$ is $(\omega,\,\Omega)$-Hermitian-Einstein, so is $k_T$ because curvature is preserved under holomorphic isometries.) Let $h_A$ be the $C^\infty$ Hermitian fibre metric induced by $(h,\,k)$ on $A$.

On the other hand, we define the $C^\infty$ Hermitian fibre metric $h_T$ on $E_T$ as in Lemma \ref{Lem:gauge-transf_curvature} by \begin{eqnarray*}\langle u,\,v\rangle_{h_T}:= \langle Tu,\,Tv\rangle_h, \hspace{5ex} u,v\in E_x,\hspace{1ex} x\in X.\end{eqnarray*} Let $(h_T)_{A_T}$ be the $C^\infty$ Hermitian fibre metric induced on $A_T$ by $(h_T,\,k_T)$.

With respect to these metrics, the holomorphic isomorphism (\ref{eqn:T-tilde_isom_def}) becomes {\it isometric}: \begin{eqnarray*}\widetilde{T}^{-1} : (A,\,h_A)\longrightarrow(A_T,\,(h_T)_{A_T}).\end{eqnarray*}

Let $D_A$, respectively $D_{A_T}$, be the Chern connection of $(A,\,D'',\,h_A)$, respectively $(A_T,\,D''_T,\,(h_T)_{A_T})$. By the analogue of (\ref{Lem:gauge-transf_Chern-connection}), we get: \begin{eqnarray*}D_{A_T} = \widetilde{T}^{-1}\circ D_A\circ\widetilde{T}.\end{eqnarray*} By Lemma \ref{Lem:gauge-transf_curvature}, the corresponding curvature forms are related by: \begin{eqnarray}\label{eqn:gauge-transf_curvature_proof_stability}i\Theta(D_{A_T}) = \widetilde{T}^{-1}\circ i\Theta(D_A)\circ\widetilde{T}.\end{eqnarray}

$\bullet$ Now, let $\sigma_{\cal F}\in H^0(X,\,A)$, respectively $\sigma_{{\cal F}_T}\in H^0(X,\,A_T)$, be the canonical section induced by the inclusion ${\cal F}\hookrightarrow{\cal E}$, respectively ${\cal F}_T\hookrightarrow{\cal E}_T$. We have $\sigma_{{\cal F}_T} = \widetilde{T}^{-1}\sigma_{\cal F}$. Together with (\ref{eqn:gauge-transf_curvature_proof_stability}), this yields: \begin{eqnarray*}i\Theta(D_{A_T})\,\sigma_{{\cal F}_T} = \widetilde{T}^{-1}\bigg(i\Theta(D_A)\,\sigma_{\cal F}\bigg).\end{eqnarray*} We get: \begin{eqnarray*}\bigg\{i\Theta(D_{A_T})\,\sigma_{{\cal F}_T},\,\sigma_{{\cal F}_T}\bigg\}_{(h_T)_{A_T}} = \bigg\{\widetilde{T}^{-1}\bigg(i\Theta(D_A)\,\sigma_{\cal F}\bigg),\,\widetilde{T}^{-1}\sigma_{\cal F}\bigg\}_{(h_T)_{A_T}} = \bigg\{i\Theta(D_A)\,\sigma_{\cal F},\,\sigma_{\cal F}\bigg\}_{h_A},\end{eqnarray*} where the last equality follows from the isometry $\widetilde{T}^{-1} : (A,\,h_A)\longrightarrow(A_T,\,(h_T)_{A_T})$.

We conclude that $Z_{\omega,\,\Omega,\,h_T}^{(m)}({\cal F}_T) = Z_{\omega,\,\Omega,\,h}^{(m)}({\cal F})$ for every coherent subsheaf ${\cal F}\subset{\cal E}$, every $T\in{\cal G}^{\cal C}$ and every $C^\infty$ Hermitian fibre metric $h$ on $E$. This proves the contention.  \hfill $\Box$

\vspace{2ex}

It follows from Proposition \ref{Prop:m-pos-stability_G-C-action_invariance} that the object defined in the next statement is meaningful and from Theorem \ref{The:m-pos-stable_implies_simple} that it is contained in a well-known space.

\begin{Cor-Def}\label{Cor-def:moduli-stable_set} In the setting of Proposition \ref{Prop:m-pos-stability_G-C-action_invariance}, we define the set-theoretic {\bf moduli space of $(\omega,\,\Omega)$-$m$-positively stable holomorphic structures} on $E$ by the first quotient in the inclusion: \begin{eqnarray}\label{eqn:moduli-stable_set}{\cal M}_{m\mbox{-stable}}(E):={\cal H}_{m\mbox{-stable}}(E)/{\cal G}^{\cal C}\xhookrightarrow { } {\cal M}^s(E)={\cal H}^s(E)/{\cal G}^{\cal C},\end{eqnarray} where the second quotient is the moduli space of simple holomorphic structures on $E$.

\end{Cor-Def}

\subsection{Group actions and uniformly stable vector bundles}\label{subsection:groups_uniformly-stable}

We now turn our attention to the {\bf uniform $m$-positive (semi-)stability} introduced in (b) of Definition \ref{Def:m-pos-stability}. The proof of Proposition \ref{Prop:m-pos-stability_G-C-action_invariance} also yields the following statement:

\begin{Cor}\label{Cor:uniform-stab_G-C-action_invariance} The setting is the same as in Proposition \ref{Prop:m-pos-stability_G-C-action_invariance}. Let $(D'',\,h)$ be a pair consisting of a holomorphic structure and a $C^\infty$ Hermitian fibre metric on a given $C^\infty$ complex vector bundle $E$ of rank $r\geq 1$ on $X$.

  Then, for every $T\in{\cal G}^{\cal C}$, the following equivalence holds: \begin{eqnarray*}D'' \hspace{1ex} \mbox{is {\bf uniformly stable} for}\hspace{1ex} h \iff  D''_T:=T^{-1}\circ D''\circ T \hspace{1ex} \mbox{is {\bf uniformly stable} for}\hspace{1ex} h_T,\end{eqnarray*} where $h_T$ is the $C^\infty$ Hermitian fibre metric on $E$ defined as in Lemma \ref{Lem:gauge-transf_curvature}.

\end{Cor}

\noindent {\it Proof.} The conclusion of the proof of Proposition \ref{Prop:m-pos-stability_G-C-action_invariance} was that $Z_{\omega,\,\Omega,\,h_T}^{(m)}(T^{-1}({\cal F})) = Z_{\omega,\,\Omega,\,h}^{(m)}({\cal F})$ for every coherent subsheaf ${\cal F}\subset{\cal E}$, every $T\in{\cal G}^{\cal C}$ and every $C^\infty$ Hermitian fibre metric $h$ on $E$. This also proves the current statement. \hfill $\Box$

\vspace{2ex}

In particular, this corollary shows that if a ${\cal G}^{\cal C}$-orbit of holomorphic structures on $E$ contains a {\it uniformly stable} element, then all its elements are {\it uniformly stable} (for possibly different fibre metrics). Thus, uniform stability is a property for the points of the classical moduli space ${\cal H}(E)/{\cal G}^{\cal C}$ of isomorphism classes of holomorphic structures on $E$: there may be {\it uniformly stable} ${\cal G}^{\cal C}$-orbits and {\it non-uniformly stable} ones. This does not rule out the existence of $C^\infty$ complex vector bundles $E$ whose all isomorphism classes of holomorphic structures are of a single of the above two types. 

The special case of this statement for positive definite automorphisms of $E$, when we start off with two Hermitian fibre metrics and consider the transition automorphism between them, can be expressed as follows:  

\begin{Cor}\label{Cor:uniform-stab_class-intrinsic} Let $(E,\,D'')\longrightarrow (X,\,\omega,\,\Omega)$ be a holomorphic vector bundle of rank $r\geq 1$ over an $n$-dimensional compact complex manifold supposed to carry forms $\omega$ and $\Omega$ with the properties of Definition \ref{Def:m-pos-stability}. Let $h_1$ and $h_2$ be arbitrary $C^\infty$ Hermitian fibre metrics on $E$ and let $T\in C^\infty(X,\,\operatorname{Herm}^{+}(E,\,h_1))$ be the (unique) positive definite (for $h_1$) endomorphism of $E$ such that \begin{eqnarray*}\langle s,\,t\rangle_{h_2} = \langle Ts,\,t\rangle_{h_1}, \hspace{5ex} x\in X, \hspace{1ex} s,t\in E_x.\end{eqnarray*}

  Let $S\in C^\infty(X,\,\operatorname{Herm}^{+}(E,\,h_1))$ be the positive definite (for $h_1$) endomorphism of $E$ such that $T=S^2$. Define a holomorphic structure on $E$ by $D''_{S^{-1}}:=S\circ D''\circ S^{-1}$.

  \vspace{1ex}

  Then, the following identity holds at every point of $X$: \begin{eqnarray}\label{eqn:Z-functions_equality_F-SF}Z^{(m)}_{\omega,\,\Omega,\,h_2}({\cal F}) = Z^{(m)}_{\omega,\,\Omega,\,h_1}(S({\cal F}))\end{eqnarray} for every coherent subsheaf ${\cal F}$ of ${\cal O}(E,\,D'')$, where $S({\cal F})\subset{\cal O}(E,\,D''_{S^{-1}})$ denotes the image sheaf of ${\cal F}$ under the holomorphic bundle isomorphism $S:(E,\,D'')\longrightarrow(E,\,D''_{S^{-1}})$.

In particular, the following equivalence holds: \begin{eqnarray*}\label{eqn:uniform-stab_class-intrinsic_equiv}h_2 \hspace{1ex} \mbox{is {\bf uniformly stable} for} \hspace{1ex} D'' \iff h_1 \hspace{1ex} \mbox{is {\bf uniformly stable} for} \hspace{1ex} D''_{S^{-1}}.\end{eqnarray*} 

\end{Cor}  

\noindent {\it Proof.} It suffices to observe that $h_1 = (h_2)_{S^{-1}}$ and to apply Corollary \ref{Cor:uniform-stab_G-C-action_invariance}. \hfill $\Box$

\vspace{2ex}

We now investigate to what extent uniqueness up to scale can be expected from fibre metrics that are uniformly stable for a given holomorphic structure. The first of the two notions introduced below uses the notion {\it $(\omega,\,\Omega)$-$m$-semi-negativity} introduced in (c) of Definition \ref{Def:m-pos}. 

\begin{Def}\label{Def:relations_metrics} The setting is the same as in Corollary \ref{Cor:uniform-stab_class-intrinsic}. We introduce the following two relations between arbitrary $C^\infty$ Hermitian fibre metrics $h_1$ and $h_2$ on $E$: \begin{eqnarray*}h_2\prec_{pt} h_1 & \iff & \bigg\{\bigg(i\Theta_{\widetilde{h}}(\widetilde{E})\wedge\omega^{m-1}\wedge\Omega\bigg)\,\widetilde{u},\,\widetilde{u}\bigg\}_{\widetilde{h}}\leq 0, \hspace{5ex} x\in X, \hspace{1ex} \widetilde{u}\in\widetilde{E}_x, \\
 h_2\prec_{int} h_1 & \iff & \int\limits_X\bigg\{\bigg(i\Theta_{\widetilde{h}}(\widetilde{E})\wedge\omega^{m-1}\wedge\Omega\bigg)\,\widetilde{u},\,\widetilde{u}\bigg\}_{\widetilde{h}}\leq 0, \hspace{5ex}  \widetilde{u}\in H^0(X,\,\widetilde{E}),\end{eqnarray*} where the holomorphic vector bundle $\widetilde{E}$ is defined as \begin{eqnarray*}\widetilde{E}:=\operatorname{Hom}\bigg((E,\,D''),\,(E,\,D''_{S^{-1}})\bigg)\simeq (E,\,D'')^\star\otimes(E,\,D''_{S^{-1}})\end{eqnarray*} and it is equipped with the $C^\infty$ Hermitian fibre metric $\widetilde{h}:=h_1^\star\otimes h_1$. By $i\Theta_{\widetilde{h}}(\widetilde{E})$ we mean the Chern curvature form of $(\widetilde{E},\,\widetilde{h})$.

\end{Def}  

With these relations in place, we obtain the following

\begin{The}\label{The:h_1-h_2_comparison_proportinality} The setting is the same as in Corollary \ref{Cor:uniform-stab_class-intrinsic}.

\vspace{1ex}

$(1)$\, If $(E,\,D'')$ is {\bf simple} and $h_2\prec_{int} h_1$, then there exists a constant $\lambda>0$ such that $h_2 = \lambda\,h_1$.

\vspace{1ex}

$(2)$\, If either $h_1$ or $h_2$ is {\bf uniformly stable} for $D''$ and if they are {\bf integrably comparable} (in the sense that either $h_2\prec_{int} h_1$ or $h_1\prec_{int} h_2$), then there exists a constant $\lambda>0$ such that $h_2 = \lambda\,h_1$.

\vspace{1ex}

$(3)$\, If $h$ is a $C^\infty$ Hermitian fibre metric on $E$, the following equivalence holds: \begin{eqnarray}\label{eqn:h_prec_h_H-E}h\prec_{pt} h \iff h \hspace{1ex} \mbox{is} \hspace{1ex} \mbox{\bf weakly} \hspace{1ex} (\omega,\,\Omega)\mbox{\bf -Hermite-Einstein}.\end{eqnarray}

\vspace{1ex}

$(4)$\, If $(E,\,D'')$ is {\bf simple} and $h_2\prec_{pt} h_1$, then there exists a constant $\lambda>0$ such that $h_2 = \lambda\,h_1$ and $h_1$, $h_2$ are both {\bf weakly $(\omega,\,\Omega)$-Hermite-Einstein}.

\end{The}

\noindent {\it Proof.} $\bullet$ We know from Corollary \ref{Cor:uniform-stab_class-intrinsic} that $S:(E,\,D'',\,h_2)\longrightarrow(E,\,D''_{S^{-1}},\,h_1)$ is a {\it holomorphic} and {\it isometric} vector bundle isomorphism and that $h_1 = (h_2)_{S^{-1}}$. In particular, \begin{eqnarray}\label{eqn:S_hol-section_E-tilde}S\in H^0(X,\,\widetilde{E}).\end{eqnarray} Thus, denoting by $D_{\widetilde{h}} = D'_{\widetilde{h}} + D''_{\widetilde{h}}$ the Chern connection of $(\widetilde{E},\,\widetilde{h})$, we get $D''_{\widetilde{h}} S =0$ on $X$.

$\bullet$ Let $D_1 = D_1' + D''$, respectively $D_2 = D_2' + D''$, be the Chern connection of $(E,\,D'',\,h_1)$, respectively of $(E,\,D'',\,h_2)$. Then \begin{eqnarray}\label{eqn:D'_same-D''_different-h}D_2' = T^{-1}\circ D_1'\circ T.\end{eqnarray} Indeed, for any smooth local sections $s,t$ of $E$, we can compute $\partial\langle s,\,t\rangle_{h_2}$ in two ways and get: \begin{eqnarray*}\partial\langle s,\,t\rangle_{h_2} & = & \{D'_2s,\,t\}_{h_2} + \{s,\,D''t\}_{h_2} = \{(T\circ D'_2)\,s,\,t\}_{h_1} + \{Ts,\,D''t\}_{h_1},\\
 \partial\langle s,\,t\rangle_{h_2} & = & \partial\langle Ts,\,t\rangle_{h_1} = \{D'_1(Ts),\,t\}_{h_1} + \{Ts,\,D''t\}_{h_1}.\end{eqnarray*} Comparing the two expressions, we infer $T\circ D'_2 = D'_1\circ T$, which is (\ref{eqn:D'_same-D''_different-h}).

$\bullet$ Since $D''_{S^{-1}}=S\circ D''\circ S^{-1}$, the standard Proposition \ref{Prop:D'_T_D'_same-metric} ensures that, for the Chern connection $D_{S^{-1}} = D'_{S^{-1}} + D''_{S^{-1}}$ of $(E,\,D''_{S^{-1}},\,h_1)$, the $(1,\,0)$-part is given by \begin{eqnarray}\label{eqn:D'_S-1_D'_1}D'_{S^{-1}} = (S^{-1})^\star_{h_1}\circ D_1'\circ S^\star_{h_1} = S^{-1}\circ D_1'\circ S,\end{eqnarray} where the second equality follows from $S$ (hence also $S^{-1}$) being $h_1$-positive definite (hence also $h_1$-self-adjoint).

$\bullet$ A link between these objects is provided by the formula: \begin{eqnarray}\label{eqn:D'_h-tilde_D'_S-1_formula}D'_{\widetilde{h}}S = D'_{S^{-1}}\circ S - S\circ D'_1,\end{eqnarray} which expresses the Leibniz rule $D'_{S^{-1}}(Su) = (D'_{\widetilde{h}}S)\,u + S(D'_1u)$ applied to an arbitrary smooth local section $u$ of $(E,\,D'',\,h_1)$ whose image $Su$ under the holomorphic (possibly non-isometric) bundle isomorphism $S:(E,\,D'',\,h_1)\longrightarrow(E,\,D''_{S^{-1}},\,h_1)$ is a smooth local section of $(E,\,D''_{S^{-1}},\,h_1)$.  

$\bullet$ A straightforward computation of the Chern curvature form of $(E,\,D''_{S^{-1}},\,h_1)$ yields: \begin{eqnarray*}i\Theta(D_{S^{-1}}) & = & i\,D'_{S^{-1}}\circ D''_{S^{-1}} + i\,D''_{S^{-1}}\circ D'_{S^{-1}} \\
  & = & i\,S^{-1}\circ (D_1'\circ T)\circ D''\circ S^{-1} + i\,S\circ D''\circ (T^{-1} \circ D_1')\circ S \\
  & = & i\,S^{-1}\circ (T\circ D_2')\circ D''\circ S^{-1} + i\,S\circ D''\circ (D_2'\circ T^{-1})\circ S \\
  & = & i\,S\circ(D_2'\circ D'')\circ S^{-1} + i\,\,S\circ(D''\circ D_2')\circ S^{-1} \\
  & = & S\circ i\,(D_2'\circ D'' + D''\circ D_2')\circ S^{-1},\end{eqnarray*} where the third line followed from the second thanks to (\ref{eqn:D'_same-D''_different-h}). We have thus got the following formula: \begin{eqnarray}\label{eqn:curvature_D_S-1_D_2}i\Theta(D_{S^{-1}}) = S\circ i\Theta(D_2)\circ S^{-1}.\end{eqnarray}

$\bullet$ On the other hand, since $S$ is a global holomorphic section of $\widetilde{E}$ (cf. (\ref{eqn:S_hol-section_E-tilde})), the computation we carried out in the proof of Theorem \ref{The:Kobayashi-vanishing_gen} (in which we now replace $E$ with $\widetilde{E}$, $s$ with $S$ and $h$ with $\widetilde{h}$) yields the following identity at every point of $X$: \begin{eqnarray}\label{eqn:dd-bar_S-S_curvature_Bochner_pointwise}i\bar\partial\partial\{S,\,S\}_{\widetilde{h}}\wedge\omega^{m-1}\wedge\Omega  =  \bigg\{\bigg(i\Theta_{\widetilde{h}}(\widetilde{E})\wedge\omega^{m-1}\wedge\Omega\bigg)\,S,\,S\bigg\}_{\widetilde{h}} - i\,\{D'_{\widetilde{h}}S,\,D'_{\widetilde{h}}S\}_{\widetilde{h}}\wedge\omega^{m-1}\wedge\Omega.\end{eqnarray} Integrating over $X$ and using the compactness of $X$ and the $\partial\bar\partial$-closedness of $\omega^{m-1}\wedge\Omega$ (which together imply the vanishing of the integral of the l.h.s. term above), we get: \begin{eqnarray}\label{eqn:dd-bar_S-S_curvature_Bochner_integrated}\int\limits_X\bigg\{\bigg(i\Theta_{\widetilde{h}}(\widetilde{E})\wedge\omega^{m-1}\wedge\Omega\bigg)\,S,\,S\bigg\}_{\widetilde{h}} = \int\limits_X  i\,\{D'_{\widetilde{h}}S,\,D'_{\widetilde{h}}S\}_{\widetilde{h}}\wedge\omega^{m-1}\wedge\Omega\geq 0,\end{eqnarray} where the last inequality follows from the pointwise non-negativity of the integrand.

Note that no special assumption on the fibre metrics $h_1$ and $h_2$ has been used so far.

\vspace{1ex}

$(1)$\, The assumption $h_2\prec_{int} h_1$ implies the non-positivity of the l.h.s. of (\ref{eqn:dd-bar_S-S_curvature_Bochner_integrated}) by taking $\widetilde{u} = S\in H^0(X,\,\widetilde{E})$. Thus, under this assumption, we obtain the vanishing of the integral of $i\,\{D'_{\widetilde{h}}S,\,D'_{\widetilde{h}}S\}_{\widetilde{h}}\wedge\omega^{m-1}\wedge\Omega$, hence the vanishing of this semi-positive $(n,\,n)$-form at every point of $X$. As already argued for other bundle-valued $(1,\,0)$-forms (see e.g. proof of Theorem \ref{The:Kobayashi-vanishing_gen}), this is equivalent to $D'_{\widetilde{h}} S =0$ everywhere on $X$.

Consequently, we get: \begin{eqnarray*}0 & = & D'_{\widetilde{h}} S \stackrel{(a)}{=} D'_{S^{-1}}\circ S - S\circ D'_1 \stackrel{(b)}{=} (S^{-1}\circ D_1'\circ S)\circ S - S\circ D'_1 = S^{-1}\circ D_1'\circ T - S\circ D'_1 \\
 & \stackrel{(c)}{=} & S^{-1}\circ\bigg(D'_{1,\,\operatorname{End}E}T\bigg) + S^{-1}\circ(T\circ D'_1)  - S\circ D'_1 = S^{-1}\circ\bigg(D'_{1,\,\operatorname{End}E}T\bigg),\end{eqnarray*} where (a) is (\ref{eqn:D'_h-tilde_D'_S-1_formula}), (b) follows from (\ref{eqn:D'_S-1_D'_1}), while (c) follows from the Leibniz rule $D_1'(Tv) = (D'_{1,\,\operatorname{End}E}T)\,v + T(D_1'v)$ applied to an arbitrary smooth local
section $v$ of $(E,\,D'',\,h_1)$ and using the $(1,\,0)$-part $D'_{1,\,\operatorname{End}E}$ of the Chern connection of $\bigg(\operatorname{End}(E,\,D''),\,\widetilde{h}=h_1^\star\otimes h_1\bigg)$. Note that this last vector bundle has a different holomorphic structure to $\widetilde{E}$ with which it shares the same underlying $C^\infty$ bundle and the same fibre metric.

Since $S$ is an isomorphism, we get $D'_{1,\,\operatorname{End}E}T=0$ everywhere on $X$. Taking adjoints with respect to $h_1$, this amounts to the first equality below: \begin{eqnarray*}0 = \bigg(D'_{1,\,\operatorname{End}E}T\bigg)^\star_{h_1} = D''_{\operatorname{End}E}\bigg(T^\star_{h_1}\bigg) = D''_{\operatorname{End}E}T,\end{eqnarray*} where $D''_{\operatorname{End}E}$ is the holomorphic structure induced on $\operatorname{End}E$ by $D''$ ($=$ the holomorphic structure of $\operatorname{End}(E,\,D'')$). For the last equality, we used the $h_1$-self-adjointness of $T$, which follows from its $h_1$-positive definiteness. We have thus proved that $T$ is a {\it holomorphic} endomorphism of $(E,\,D'')$.

It remains to invoke the {\it simplicity} assumption on $(E,\,D'')$ to conclude that $T$ is a positive constant multiple of the identity automorphism of $E$. Since $T$ expresses $h_2$ in terms of $h_1$, we conclude that these two fibre metrics are proportional. This proves $(1)$. 

\vspace{1ex}

$(2)$\, If either of the fibre metrics $h_1$ and $h_2$ on $E$ is {\it uniformly stable} for $D''$, the holomorphic vector bundle $(E,\,D'')$ is, in particular, {\it $(\omega,\,\Omega)$-$m$-positively stable}. Therefore, thanks to Theorem \ref{The:m-pos-stable_implies_simple} (see also its rewording as Corollary \ref{Cor:m-pos-stable_implies_simple}), $(E,\,D'')$ is {\it simple}.

Thus, $(2)$ follows at once from $(1)$.

\vspace{1ex}

$(3)$\, Writing $(\widetilde{E},\,\widetilde{h})\simeq (E,\,D'',\,h_1)^\star\otimes(E,\,D''_{S^{-1}},\,h_1)$ and applying the standard formula for the curvature form of a tensor product vector bundle, we get the first line below: \begin{eqnarray*}i\Theta_{\widetilde{h}}(\widetilde{E}) & = & i\Theta_{h_1^\star}((E,\,D'')^\star)\otimes\operatorname{Id}_{(E,\,D''_{S^{-1}})} + \operatorname{Id}_{(E,\,D'')^\star}\otimes i\Theta_{h_1}((E,\,D''_{S^{-1}})) \\
  & = & \operatorname{Id}_{(E,\,D'')^\star}\otimes\bigg(S\circ i\Theta(D_2)\circ S^{-1}\bigg) - i\Theta(D_1)^{\dagger}\otimes\operatorname{Id}_{(E,\,D''_{S^{-1}})},\end{eqnarray*} where the second line follows from formula (\ref{eqn:curvature_D_S-1_D_2}) for $i\Theta_{h_1}((E,\,D''_{S^{-1}})) = i\Theta(D_{S^{-1}})$ and from the classical formula (\ref{eqn:curvature_E-star_dagger}) for the curvature form of a dual bundle.

Thus, for every smooth local section of $\widetilde{E}$ of the shape $\widetilde{u} = u^\star\otimes v$ with $u,v$ smooth local sections of $E$ and $u^\star$ the $h_1$-dual of $u$ (as a section of $E^\star$), we get: \begin{eqnarray*}\bigg\{i\Theta_{\widetilde{h}}(\widetilde{E})\,\widetilde{u},\,\widetilde{u}\bigg\}_{\widetilde{h}} & = & \bigg\{\bigg(S\circ i\Theta(D_2)\circ S^{-1}\bigg)\,v,\,v\bigg\}_{h_1}\,|u^\star|^2_{h_1^\star} - \bigg\{i\Theta(D_1)^{\dagger}\,u^\star,\,u^\star\bigg\}_{h_1^\star}\,|v|^2_{h_1} \\
  & = & \bigg\{i\Theta(D_2)\,(S^{-1}v),\,S^{-1}v\bigg\}_{h_2}\,|u|^2_{h_1} - \bigg\{i\Theta(D_1)\,u,\,u\bigg\}_{h_1}\,|v|^2_{h_1},\end{eqnarray*} where, in order to get the first term on the second line, we wrote the second factor $v$ on the previous line as $S(S^{-1}v)$ and used the fact that $S$ is an isometry from $(E,\,h_2)$ to $(E,\,h_1)$.

Now, suppose that $h_2\prec_{pt} h_1$. This implies that the above quantity, after multiplication by $\omega^{m-1}\wedge\Omega$, is $\leq 0$ for every $x\in X$ and for all sections $u,v$ as above in a neighbourhood of $x$. In other words, \begin{eqnarray*}\frac{\bigg\{\bigg(i\Theta(D_2)\wedge\omega^{m-1}\wedge\Omega\bigg)\,(S^{-1}v),\,S^{-1}v\bigg\}_{h_2}}{|v|^2_{h_1}} \leq \frac{\bigg\{i\Theta(D_1)\,u,\,u\bigg\}_{h_1}}{|u|^2_{h_1}}\end{eqnarray*} for all non-vanishing smooth local sections $u,v$ of $E$.

In the special case when $h_1 = h_2:=h$, from $h\prec_{pt} h$ we infer: \begin{eqnarray*}\frac{\bigg\{\bigg(i\Theta_h(E)\wedge\omega^{m-1}\wedge\Omega\bigg)\,v\,\,v\bigg\}_h}{|v|^2_h}\leq\frac{\bigg\{\bigg(i\Theta_h(E)\wedge\omega^{m-1}\wedge\Omega\bigg)\,u\,\,u\bigg\}_h}{|u|^2_h}, \hspace{5ex} \mbox{for all}\hspace{1ex} u,v,\end{eqnarray*} because $S=\operatorname{Id}_E$ and $D_1=D_2$. We have set $i\Theta_h(E):= i\Theta(D_1) = i\Theta(D_2)$. Interchanging $u$ and $v$, we infer that, for every point $x\in X$, there exists $\lambda(x)\in\R$ such that \begin{eqnarray*}\frac{i\Theta_h(E)\wedge\omega^{m-1}\wedge\Omega}{dV_\omega}(x) = \lambda(x)\,\operatorname{Id}_{E_x},  \hspace{5ex} x\in X.\end{eqnarray*} This means precisely that the fibre metric $h$ is {\it weakly $(\omega,\,\Omega)$-Hermite-Einstein}.

\vspace{1ex}

$(4)$\, The first part of the statement follows from $(1)$ thanks to the implication: $h_2\prec_{pt}h_1 \implies h_2\prec_{int}h_1$. Then, the second part of the statement follows from $(3)$ after rescaling by a positive constant one of the metrics $h_1$ and $h_2$ to make them equal.  \hfill $\Box$

\subsection{Moduli spaces of uniformly stable vector bundles}\label{subsection:moduli_unif-stable}

Based on the above considerations and on (b) of Definition \ref{Def:m-pos-stability} (see also the lightening of the terminology that immediately follows that definition -- recall that a holomorphic structure $D''$ on $E$ is termed {\it uniformly stable} if it has this property with respect to some  $C^\infty$ Hermitian fibre metric $h$ on $E$), we propose the following:

\begin{Cor-Def}\label{Cor-Def:m-pos-stable_moduli} Let $E\longrightarrow X$ be a $C^\infty$ $\C$-vector bundle of rank $r\geq 1$ on an $n$-dimensional compact complex manifold. Fix $m\in\{1,\dots , n\}$ and suppose that there exist forms $\omega$ and $\Omega$ on $X$ satisfying the properties of Definition \ref{Def:m-pos-stability}.

\vspace{1ex}

(a)\, We consider the following sets: \begin{eqnarray*}{\cal H}_{m\mbox{-u-stable}}(E) &:= &\bigg\{D''\,\mid\, D''\in{\cal H}(E) \hspace{1ex} \mbox{is {\bf uniformly stable}}\bigg\}\\
  {\cal HM}_{m\mbox{-u-stable}}(E) &:= & \bigg\{(D'',\,h)\,\mid\, D''\in{\cal H}(E) \hspace{1ex} \mbox{is {\bf uniformly stable} for the}\hspace{1ex} C^\infty \hspace{1ex} \mbox{fibre metric} \hspace{1ex} h\bigg\} \\
 {\cal H}_{m\mbox{-u-stable}}(E,\,h) &:=& \bigg\{D''\,\mid\, D''\in{\cal H}(E) \hspace{1ex} \mbox{is {\bf uniformly stable for} h}\bigg\},\end{eqnarray*} where for the last set, $h$ is a fixed $C^\infty$ Hermitian fibre metric on $E$, supposed uniformly stable for at least one holomorphic structure $D''$ on $E$.  

\vspace{1ex}

(b)\, The right actions: \begin{eqnarray*}{\cal H}_{m\mbox{-u-stable}}(E)\times{\cal G}^{\cal C} \longrightarrow {\cal H}_{m\mbox{-u-stable}}(E), & \hspace{5ex} & (D'',\, T)\longmapsto D''_T,\\
  {\cal HM}_{m\mbox{-u-stable}}(E)\times{\cal G}^{\cal C} \longrightarrow {\cal HM}_{m\mbox{-u-stable}}(E), & \hspace{5ex} & \bigg((D'',\, h),\, T)\bigg)\longmapsto(D''_T,\, h_T),\\
 {\cal H}_{m\mbox{-u-stable}}(E,\,h)\times U(E,\,h)\longrightarrow{\cal H}_{m\mbox{-u-stable}}(E,\,h), & \hspace{5ex} & (D'',\, T)\longmapsto D''_T,\end{eqnarray*} are well defined, where $U(E,\,h)$ denotes the group of $h$-unitary automorphisms of $E$ ($=$ the group of $C^\infty$ isomorphisms of $E$ that define isometries of each fibre $E_x$ for the inner product $\langle\,\cdot\,,\,\cdot\,\rangle_{h(x)}$).

\vspace{1ex}

(c)\, We define the set-theoretic {\bf moduli space of $(\omega,\,\Omega)$-uniformly-$m$-positively stable} ({\bf uniformly stable}, for short) {\bf holomorphic structures} on $E$ to be the quotient: \begin{eqnarray*}{\cal M}_{m\mbox{-u-stable}}(E):={\cal H}_{m\mbox{-u-stable}}(E)/{\cal G}^{\cal C}.\end{eqnarray*}

Similarly, we define the set-theoretic {\bf moduli space of $(\omega,\,\Omega)$-uniformly-$m$-positively stable} ({\bf uniformly stable}, for short) {\bf pairs of holomorphic structures/fibre metrics} on $E$ to be the quotient: \begin{eqnarray*}\widetilde{\cal M}_{m\mbox{-u-stable}}(E):={\cal HM}_{m\mbox{-u-stable}}(E)/{\cal G}^{\cal C}.\end{eqnarray*}

Finally, we define the set-theoretic {\bf moduli space of $(\omega,\,\Omega)$-$h$-uniformly-$m$-positively stable} ({\bf $h$-uniformly stable}, for short) {\bf holomorphic structures} on $E$ to be the quotient: \begin{eqnarray*}{\cal M}_{m\mbox{-u-stable}}(E,\,h):={\cal H}_{m\mbox{-u-stable}}(E,\,h)/U(E,\,h).\end{eqnarray*}

\end{Cor-Def}

\noindent {\it Proof.} Part (b) follows at once from Corollary \ref{Cor:uniform-stab_G-C-action_invariance} and also, in the case of the third set, from the fact that $h_T = h$ for every $h$ and every $T\in U(E,\,h)$. \hfill $\Box$

\vspace{2ex}

The first observation is that the moduli space of uniformly stable pairs is a natural refinement of its counterpart for holomorphic structures alone.

\begin{Cor}\label{Cor:can-surjection_u-stable_moduli-spaces} The setting is the same as in Corollary and Definition \ref{Cor-Def:m-pos-stable_moduli}.

  There is a {\bf canonical surjection}: \begin{eqnarray*}\widetilde{p}:\widetilde{\cal M}_{m\mbox{-u-stable}}(E)\longrightarrow{\cal M}_{m\mbox{-u-stable}}(E), \hspace{5ex} [(D'',\,h)]\longmapsto[D''],\end{eqnarray*} where $[\,\,]$ denotes ${\cal G}^{\cal C}$-orbits.

\end{Cor}

\noindent {\it Proof.} Let $p:{\cal HM}_{m\mbox{-u-stable}}(E)\longrightarrow{\cal H}_{m\mbox{-u-stable}}(E)$ be the map defined by $(D'',\,h)\mapsto D''$. It is ${\cal G}^{\cal C}$-equivariant: \begin{eqnarray*}p\bigg((D'',\,h)\cdot T\bigg) = p\bigg((D'',\,h)\bigg)\cdot T = D''_T,  \hspace{5ex} T\in{\cal G}^{\cal C}.\end{eqnarray*} Hence, it descends to the map $\widetilde{p}$ of the statement.

Its surjectivity is immediate from the definition. \hfill $\Box$

\vspace{2ex}

We do not know at this stage whether $\widetilde{p}$ is injective, but that possibility seems unlikely. Indeed, let $(D''_1,\,h_1), (D''_2,\,h_2)\in{\cal HM}_{m\mbox{-u-stable}}(E)$ such that $\widetilde{p}([(D''_1,\,h_1)]) = \widetilde{p}([(D''_2,\,h_2)])$. Thus, there exists $T\in{\cal G}^{\cal C}$ such that $D''_2 = (D_1'')_T$.

On the other hand, there exists a unique $S\in{\cal G}^{\cal C}$ such that $h_2 = (h_1)_{S^{-1}}$.

Meanwhile, we have the equivalence: \begin{eqnarray*}[(D''_1,\,h_1)] = [(D''_2,\,h_2)] \iff \exists\,R\in{\cal G}^{\cal C} \hspace{1ex}\mbox{such that}\hspace{1ex} (D''_1,\,h_1) = \bigg((D''_2)_R,\,(h_2)_R)  \bigg).\end{eqnarray*} So, the injectivity of $\widetilde{p}$ would force the equality $[(D''_1,\,h_1)] = [(D''_2,\,h_2)]$, which in turn would force $h_1 = (h_2)_R$. We would then get $R = S^{-1}$ by the uniqueness of $S^{-1}$ with this property. Then, we would also get $D''_1 = (D''_2)_{S^{-1}}$, or equivalently $D_2'' = (D''_1)_S$.

Since we also have $D''_2 = (D_1'')_T$, we would infer that $ST^{-1} : (E,\,D_1'')\longrightarrow(E,\,D_1'')$ is holomorphic. However, the holomorphic structure $D''$ of $E$ is {\it simple} because it is {\it $(\omega,\,\Omega)$-$m$-positively stable} (see Corollary \ref{Cor:m-pos-stable_implies_simple}). Thus, there must exist a constant $\lambda\in\C$ such that $ST^{-1} = \lambda\,\operatorname{Id}_E$, or equivalently $S = \lambda\, T$. However, $S$ is uniquely determined by the metrics $h_1$ and $h_2$, while $T$ is determined (uniquely up to a constant factor) by the holomorphic structures $D_1''$ and $D_2''$. The possible proportionality of $S$ and $T$ seems to be a rather restrictive condition.

\vspace{2ex}

The last part of the above argument, applied to the case where $D_1'' = D_2'' = D''$, yields the following:

\begin{Cor}\label{Cor:pairs_same-orbit_condition} The setting is the same as in Corollary and Definition \ref{Cor-Def:m-pos-stable_moduli}. Let $D''$ be a holomorphic structure on $E$, supposed to be {\bf uniformly stable} for two $C^\infty$ Hermitian fibre metrics $h_1$ and $h_2$ on E. (In other words, $(D'',\,h_1), (D'',\,h_2)\in{\cal HM}_{m\mbox{-u-stable}}(E)$.)

  \vspace{1ex}

  Then, $(D'',\,h_1)$ and $(D'',\,h_2)$ lie in the same ${\cal G}^{\cal C}$-orbit
if and only if the metrics $h_1$ and $h_2$ are proportional: \begin{eqnarray*}[(D'',\,h_1)] = [(D'',\,h_2)]\in\widetilde{\cal M}_{m\mbox{-u-stable}}(E) \iff \exists\, \lambda>0 \hspace{1ex} \mbox{constant, \hspace{2ex} such that} \hspace{2ex} h_2 = \lambda\,h_1.\end{eqnarray*}

\end{Cor}

\vspace{2ex}

The above result tells us when two points $[(D'',\,h_1)], [(D'',\,h_2)]\in\widetilde{\cal M}_{m\mbox{-u-stable}}(E)$ sharing the holomorphic structure $D''$ are equal. We can go further and wonder whether they can ever be distinct. In other words, do we have {\bf uniqueness}, up to scale, of the uniformly stable metric for a given holomorphic structure?

\begin{Question}\label{Question:uniqueness_unif-stable-metrics} The setting is the same as in Corollary and Definition \ref{Cor-Def:m-pos-stable_moduli}. Let $D''$ be a holomorphic structure on $E$, supposed to be {\bf uniformly stable} for two $C^\infty$ Hermitian fibre metrics $h_1$ and $h_2$ on E.

  Does it follow that there exists a constant $\lambda>0$ such that $h_2 = \lambda\,h_1$?

\end{Question}

We will outline a strategy for a possible answer to this question further down, but we now observe one of its consequences.

\begin{Cor}\label{Cor:can-injection_u-stable_moduli-spaces} The setting is the same as in Corollary and Definition \ref{Cor-Def:m-pos-stable_moduli}. If the answer to Question \ref{Question:uniqueness_unif-stable-metrics} is {\bf affirmative}, the following {\bf canonical} map: \begin{eqnarray*}\iota_h:{\cal M}_{m\mbox{-u-stable}}(E,\, h)\longrightarrow{\cal M}_{m\mbox{-u-stable}}(E), \hspace{5ex} [D'']_{U(E,\, h)}\longmapsto[D'']_{{\cal G}^{\cal C}},\end{eqnarray*} is {\bf injective}, where $[\,\,]_{U(E,\, h)}$ denotes a $U(E,\, h)$-orbit and $[\,\,]_{{\cal G}^{\cal C}}$ denotes a ${\cal G}^{\cal C}$-orbit.

\end{Cor}  

\noindent {\it Proof.} The map $\iota_h$ is well defined and canonical as it is induced by the inclusions ${\cal H}_{m\mbox{-u-stable}}(E,\,h)\subset{\cal H}_{m\mbox{-u-stable}}(E)$ and $U(E,\,h)\subset{\cal G}^{\cal C}$.

Now, let $D_1'', D_2''\in{\cal H}_{m\mbox{-u-stable}}(E,\,h)$ such that $[D_1'']_{{\cal G}^{\cal C}} = [D_2'']_{{\cal G}^{\cal C}}$. This means that there exists $T\in{\cal G}^{\cal C}$ such that $D_2'' = (D_1'')_T = T^{-1}\circ D_1''\circ T$. To prove injectivity of $\iota_h$, we need to prove the existence of $R\in U(E,\,h)$ such that $D_2'' = (D_1'')_R$. If such an $R$ exists, then $(D_1'')_{TR^{-1}} = D_1''$, which amounts to the smooth automorphism $TR^{-1}:(E,\,D_1'')\longrightarrow(E,\,D_1'')$ being holomorphic. However, $(E,\,D_1'')$ is uniformly stable (by hypothesis), hence also {\it $(\omega,\,\Omega)$-$m$-positively stable}. By Corollary \ref{Cor:m-pos-stable_implies_simple}, $(E,\,D_1'')$ is then {\it simple}. Therefore, its only holomorphic endomorphisms are the homotheties. We deduce the existence of a constant $\lambda\in\C$ such that $T=\lambda R$ if an $R$ as above exists.

Thus, the question reduces to whether $T$ is constrained to be a scalar multiple of an $h$-unitary automorphism by the $h$-uniform stability hypotheses on $D_1''$ and $D_2''$. Let $T = US$ be the (unique) polar decomposition of $T$ w.r.t. $h$. This means that $U\in U(E,\,h)$ and $S\in{\cal G}^{\cal C}$ is positive definite w.r.t. $h$ at every point of $X$. We get: \begin{eqnarray*}D_2'' = (D_1'')_T = ((D_1'')_U)_S\end{eqnarray*} is at once {\it uniformly stable} for $h$ (by the assumption made on $D_2''$) and {\it uniformly stable} for $h_T = (h_U)_S = h_S$ (by the assumption made on $D_1''$ combined with Corollary \ref{Cor:uniform-stab_G-C-action_invariance}).

  Thus, $h$ and $h_S$ are uniformly stable fibre metrics for the same holomorphic structure $D_2''$ on $E$. If the answer to Question \ref{Question:uniqueness_unif-stable-metrics} is affirmative, there exists a constant $\lambda>0$ such that $h_S = \lambda^2\,h$. This amounts to $\lambda^2\,\langle u,\,v\rangle_h = \langle S^2u,\,v\rangle_h$ for every $x\in X$ and all $u,v\in E_x$. We conclude that $S^2 =\lambda^2\,\operatorname{Id}_E$, hence $S =\lambda\,\operatorname{Id}_E$. (We have used several times the fact that $S$ is positive definite, hence also self-adjoint, w.r.t. $h$.) Thus, the polar decompsition of $T$ becomes $T = \lambda\,U$ and we are done. \hfill $\Box$

\vspace{2ex}

We equip each of the three set-theoretic moduli spaces introduced in Corollary and Definition \ref{Cor-Def:m-pos-stable_moduli} with the quotient topology inherited from the original topological space, ${\cal H}_{m\mbox{-u-stable}}(E)$, ${\cal HM}_{m\mbox{-u-stable}}(E)$, respectively $ {\cal H}_{m\mbox{-u-stable}}(E,\,h)$, whose quotient it is.

\begin{Prop}\label{Prop:Hausdorff_moduli_h-unif-stable} The setting is the same as in Corollary and Definition \ref{Cor-Def:m-pos-stable_moduli}. We fix a $C^\infty$ Hermitian fibre metric $h$ on $E$.

  Then, the moduli space ${\cal M}_{m\mbox{-u-stable}}(E,\,h)={\cal H}_{m\mbox{-u-stable}}(E,\,h)/U(E,\,h)$ of {\bf $h$-uniformly stable holomorphic structures} on $E$, equipped with the quotient topology, is {\bf Hausdorff}.

\end{Prop}  

\noindent {\it Proof.} This follows at once from the following classical result (see e.g. [Kob87, Corollary 7.1.15]): the moduli space ${\cal H}(E,\,h)/U(E,\,h)$ of $h$-compatible Chern connections on $E$ (or, equivalently, of pairs $(D'',\,h)$ of a holomorphic structure on $E$ and the fixed Hermitian fibre metric $h$) is {\it Hausdorff}. The main step in the proof of this classical result consists in proving that the right action of $U(E,\,h)$ on ${\cal H}(E,\,h)$ is {\it proper}. The main ingredient that yields this properness is the {\it compactness} of the unitary group $U(r)$.

Our topological space ${\cal H}_{m\mbox{-u-stable}}(E,\,h)$ of $h$-uniformly stable holomorphic structures on $E$ is a topological subspace of ${\cal H}(E,\,h)$. Since we quotient both topological spaces by the same group $U(E,\,h)$ acting in the same way -- recall that we have already proved that ${\cal H}_{m\mbox{-u-stable}}(E,\,h)$ is invariant under the right action of $U(E,\,h)$ -- and we equip both quotients with the quotient topologies, our moduli space ${\cal M}_{m\mbox{-u-stable}}(E,\,h)={\cal H}_{m\mbox{-u-stable}}(E,\,h)/U(E,\,h)$ is a topological subspace of the classical moduli space ${\cal H}(E,\,h)/U(E,\,h)$.

The contention follows from the elementary fact according to which every topological subspace of a Hausdorff topological space is Hausdorff.  \hfill $\Box$

\section{The weakly holomorphic subbundle formalism in our case}\label{section:w-h-subbundle_formalism}

Let $(E,\,D'',\,h)\longrightarrow(X,\,\omega,\,\Omega)$ be a Hermitian holomorphic vector bundle of complex rank $r\geq 1$ over a compact complex $n$-dimensional manifold supposed to carry forms $\omega$ and $\Omega$ on $X$ satisfying the properties of Definition \ref{Def:m-pos-stability}. The following statement is standard:

\begin{Fact}\label{Fact:sheaf_pi} There is a {\bf canonical bijection} between:

\vspace{1ex}

\noindent (a)\, the set ${\cal A}$ of {\bf coherent analytic subsheaves} ${\cal F}$ of ${\cal O}(E)$;

\vspace{1ex}

\noindent (b)\, the set ${\cal B}$ of {\bf weakly holomorphic subbundles} of $E$, namely sections $\pi\in L^2_1(X,\,\operatorname{End}E)$ satisfying the following conditions:

\vspace{1ex}

(i)\, there exists an analytic subset $S\subset X$ with $\mbox{codim}_XS\geq 2$ such that the restriction $\pi_{|X\setminus S}\in C^\infty(X\setminus S,\,\operatorname{End}E)$;

\vspace{1ex}

(ii)\, $\pi = \pi^\star = \pi^2$ on $X\setminus S$;

\vspace{1ex}

(iii)\, $(\operatorname{Id}_E - \pi)\circ D''_{\operatorname{End}E}\pi = 0$  on $X\setminus S$. 

\vspace{1ex}

This bijection is given by: \begin{eqnarray*}{\cal B}\ni\pi\longmapsto{\cal F}_{|X\setminus S}:=\pi_{|X\setminus S}(E_{|X\setminus S})\in{\cal A}.\end{eqnarray*}

\end{Fact}  

\vspace{2ex}

Thus, on $X\setminus S$, $\pi$ is the $h$-orthogonal projection onto the holomorphic vector subbundle $F:={\cal F}_{|X\setminus S}\subset E_{|X\setminus S}$. The fact that every element ${\cal F}\in{\cal A}$ defines an element $\pi\in{\cal B}$ can be seen easily. Conversely, the fact that every element $\pi\in{\cal B}$ defines an element ${\cal F}\in{\cal A}$ was first stated and proved by Uhlenbeck and Yau in [UY86] and [UY89]. It was subsequently given a different proof in [Pop05].

\vspace{2ex}

Before applying this formalism to our setting, we recall, for the reader's convenience, a few standard facts as they are presented in [Dem97, V-$\S14$]. We group them in the following:

\begin{Prop}({\bf standard})\label{Prop:standard_exact-seq_2nd-fundamental} Let $0\longrightarrow F {\stackrel{j}{\longrightarrow}} E {\stackrel{g}{\longrightarrow}} Q \longrightarrow 0$ be a short exact sequence of holomorphic vector bundles over a complex manifold $X$. Equip $E$ with a $C^\infty$ Hermitian fibre metric $h=h_E$ and consider the induced, respectively quotient, fibre metrics $h_F$, respectively $h_Q$, on $F$, respectively $Q$. One lets $j^\star:E\longrightarrow F$ and $g^\star:Q\longrightarrow E$ be the adjoint morphisms of $j$, respectively $g$, with respect to these fibre metrics. Thus, $j$ and $g$ are holomorphic morphisms, while $j^\star$ and $g^\star$ are $C^\infty$.

\vspace{1ex}

$(1)$\, According to the $C^\infty$ isomorphism $j^\star\oplus g:E\longrightarrow F\oplus Q$, the Chern connection $D_E = D_E' + D_E''$ of $(E,\, D''_E=\bar\partial,\, h)$ has the shape: \begin{eqnarray*}D_E = \begin{pmatrix} D_F & -\beta^\star \\
    \beta & D_Q\end{pmatrix},\end{eqnarray*} where $D_F = D_F' + D_F''$ and $D_Q = D_Q' + D_Q''$ are the Chern connections of $(F,\,D''_F=\bar\partial,\, h_F)$, respectively $(Q,\,D''_Q=\bar\partial,\, h_Q)$, while $\beta\in C^\infty_{1,\,0}(X,\,\operatorname{Hom}(F,\,Q))$ is called the {\bf second fundamental form} of $F$ in $E$ and $\beta^\star\in C^\infty_{0,\,1}(X,\,\operatorname{Hom}(Q,\,F))$ is its adjoint.

\vspace{1ex}

$(2)$\, The morphisms $j$, $g$ and their adjoints have the following properties: \begin{align}\label{eqn:j-g_adjoints_prop}\nonumber D'_{\operatorname{Hom}(F,\,E)}j & =  g^\star\circ\beta;  &    \bar\partial j & =  0, \\
\nonumber  D'_{\operatorname{Hom}(E,\,Q)}g & =  -\beta\circ j^\star;  &  \bar\partial g & = 0 \\ 
\nonumber D'_{\operatorname{Hom}(E,\,F)}j^\star & =  0;  &    \bar\partial j^\star & =  \beta^\star\circ g; \\
 D'_{\operatorname{Hom}(Q,\,E)}g^\star & =  0;   &  \bar\partial g^\star & = -j\circ\beta^\star.\end{align} 
  
\vspace{1ex}

$(3)$\, We always have $\bar\partial\beta^\star = D''_{\operatorname{Hom}(Q,\,F)}\beta^\star = 0$ and the Chern curvature form of $(E,\, D''_E=\bar\partial,\, h)$ is given by: \begin{eqnarray*}i\Theta_h(E) = \begin{pmatrix} i\Theta_{h_F}(F) - i\beta^\star\wedge\beta   &  iD'_{\operatorname{Hom}(Q,\,F)}\beta^\star  \\
    i\bar\partial\beta &   i\Theta_{h_Q}(Q) - i\beta\wedge\beta^\star\end{pmatrix}.\end{eqnarray*}

Explicitly, this means, in particular, that the following formulae hold for the $(F\to F)$ and the $(Q\to Q)$ components of $i\Theta_h(E)\in C^\infty_{1,\,1}(X,\,\operatorname{End}E)$ w.r.t. the $C^\infty$ isomorphism $j^\star\oplus g:E\longrightarrow F\oplus Q$: \begin{eqnarray}\label{eqn:curvature_subbundle}i\Theta_h(E)_{|F}: & = & j^\star\circ i\Theta_h(E)\circ j = i\Theta_{h_F}(F) - i\beta^\star\wedge\beta \\
 \label{eqn:curvature_quot-bundle} i\Theta_h(E)_{|Q}: & = & g\circ i\Theta_h(E)\circ g^\star = i\Theta_{h_Q}(Q) - i\beta\wedge\beta^\star.\end{eqnarray}

\end{Prop}

\vspace{2ex}

Now, let us go back to our setting and suppose that ${\cal F}$ is a coherent subsheaf of a Hermitian holomorphic vector bundle $(E,\,h)$ of rank $r\geq 1$ over an $n$-dimensional complex manifold $X$. Let $S\subset X$ be the analytic subset of codimension $\geq 2$ outside which ${\cal F}$ is locally free (and given by a holomorphic vector bundle $F$). We denote by $Q:=E/F$ the quotient bundle over $X\setminus S$ and we will use the notation of Proposition \ref{Prop:standard_exact-seq_2nd-fundamental}. In particular, the short exact sequence of holomorphic vector bundles of Proposition \ref{Prop:standard_exact-seq_2nd-fundamental} is now defined only over $X\setminus S$.

\begin{Cor}\label{Cor:pi_j_g_curvature} Let $\pi\in L^2_1(X,\,\operatorname{End}E)$ be the element of the set ${\cal B}$ of Fact \ref{Fact:sheaf_pi} that corresponds to ${\cal F}$ under that canonical bijection. The following identities hold everywhere on $X\setminus S$:

\vspace{1ex}

$(1)$\, (a)\, $\pi = j\circ j^\star$ \hspace{2ex} and \hspace{2ex} $\operatorname{Id}_F = j^\star\circ j$; \hspace{6ex} (b)\, $\operatorname{Id}_E - \pi = g^\star\circ g$ \hspace{2ex} and \hspace{2ex} $\operatorname{Id}_Q = g\circ g^\star$;

\vspace{1ex}

$(2)$\,  (a)\, $D'_{\operatorname{End}E}\pi = g^\star\circ\beta\circ j^\star$, \hspace{2ex} hence \hspace{2ex} $\beta = g\circ(D'_{\operatorname{End}E}\pi)\circ j$;

\vspace{1ex}

\hspace{3ex} (b)\, $D''_{\operatorname{End}E}\pi = j\circ\beta^\star\circ g$, \hspace{2ex} hence \hspace{2ex} $\beta^\star = j^\star\circ(D''_{\operatorname{End}E}\pi)\circ g^\star$;

\vspace{1ex}

$(3)$\, (a)\, $j\circ i\Theta_{h_F}(F)\circ j^\star = j\circ i\Theta_h(E)_{|F}\circ j^\star + iD''_{\operatorname{End}E}\pi\wedge D'_{\operatorname{End}E}\pi$, hence also \begin{eqnarray*}i\Theta_{\det h_F}(\det F) = \mbox{Tr}_F\bigg(i\Theta_h(E)_{|F}\bigg) + \mbox{Tr}_E\bigg(iD''_{\operatorname{End}E}\pi\wedge D'_{\operatorname{End}E}\pi\bigg);\end{eqnarray*}  

 \vspace{1ex}

\hspace{3ex} (b)\, $g^\star\circ i\Theta_{h_Q}(Q)\circ g = g^\star\circ i\Theta_h(E)_{|Q}\circ g + iD'_{\operatorname{End}E}\pi\wedge D''_{\operatorname{End}E}\pi$, hence also \begin{eqnarray*}i\Theta_{\det h_Q}(\det Q) = \mbox{Tr}_Q\bigg(i\Theta_h(E)_{|Q}\bigg) + \mbox{Tr}_E\bigg(iD'_{\operatorname{End}E}\pi\wedge D''_{\operatorname{End}E}\pi\bigg).\end{eqnarray*}

\end{Cor}
  
\noindent {\it Proof.} (1)\, The first identity in (a) follows from the fact that $\pi$ is the $h$-orthogonal projection onto $F$. Similarly, the first identity in (b) follows from the fact that $\operatorname{Id}_E - \pi$ is the $h$-orthogonal projection onto $Q$. The second identity in both (a) and (b) follows at once from the definitions.

$(2)$\, Taking $D'_{\operatorname{End}E}$ in the identity $\pi = j\circ j^\star$, we get: \begin{eqnarray*}D'_{\operatorname{End}E}\pi = \bigg(D'_{\operatorname{Hom}(F,\,E)}j\bigg)\circ j^\star + j\circ\bigg(D'_{\operatorname{Hom}(E,\,F)}j^\star\bigg) = (g^\star\circ\beta)\circ j^\star,\end{eqnarray*} where we have used the values of the $D'$-derivatives of $j$ and $j^\star$ given by (\ref{eqn:j-g_adjoints_prop}). This proves the first identity in (a). The second identity follows from the first by composing by $g$ on the left and by $j$ on the right and using the identities $g\circ g^\star = \operatorname{Id}_Q$ and $j^\star\circ j = \operatorname{Id}_F$ observed under $(1)$ above.

To prove part (b), we take $D''_{\operatorname{End}E}$ in the identity $\pi = j\circ j^\star$ to get: \begin{eqnarray*}D''_{\operatorname{End}E}\pi = (\bar\partial j)\circ j^\star + j\circ\bar\partial j^\star = j\circ(\beta^\star\circ g),\end{eqnarray*} where we have used the values of the $\bar\partial$-derivatives of $j$ and $j^\star$ given by (\ref{eqn:j-g_adjoints_prop}). This proves the first identity in (b). The second identity follows from the first by composing by $j^\star$ on the left and by $g^\star$ on the right and using the identities $g\circ g^\star = \operatorname{Id}_Q$ and $j^\star\circ j = \operatorname{Id}_F$ observed under $(1)$ above.

$(3)$\, Putting together the identities proved under $(2)$, we get: \begin{eqnarray*}iD''_{\operatorname{End}E}\pi\wedge D'_{\operatorname{End}E}\pi = i(j\circ\beta^\star\circ g)\wedge(g^\star\circ\beta\circ j^\star) = j\circ(i\beta^\star\wedge\beta)\circ j^\star,\end{eqnarray*} since $g\circ g^\star = \operatorname{Id}_Q$.

On the other hand, (\ref{eqn:curvature_subbundle}) yields: $i\Theta_{h_F}(F) = i\Theta_h(E)_{|F} + i\beta^\star\wedge\beta$. Composing this identity with $j$ on the left and with $j^\star$ on the right and using the above identity for $j\circ(i\beta^\star\wedge\beta)\circ j^\star$, we get the first identity stated under $(3)(a)$. The second identity in $(3)(a)$ follows from the first after taking $\operatorname{Tr}_E$ and observing that $\operatorname{Tr}_E\bigg(j\circ i\Theta_h(E)_{|F}\circ j^\star\bigg) = \operatorname{Tr}_F\bigg(i\Theta_h(E)_{|F}\bigg)$ and \begin{eqnarray*}\operatorname{Tr}_E\bigg(j\circ i\Theta_{h_F}(F)\circ j^\star\bigg) = \operatorname{Tr}_F\bigg(i\Theta_{h_F}(F)\bigg) = i\Theta_{\det h_F}(\det F).\end{eqnarray*} This proves part (a).

Part (b) can be proved in a similar way, starting from \begin{eqnarray*}iD'_{\operatorname{End}E}\pi\wedge D''_{\operatorname{End}E}\pi = i(g^\star\circ\beta\circ j^\star)\wedge(j\circ\beta^\star\circ g) = g^\star\circ(i\beta\wedge\beta^\star)\circ g,\end{eqnarray*} since $j^\star\circ j = \operatorname{Id}_F$. We then use the identity $i\Theta_{h_Q}(Q) = i\Theta_h(E)_{|Q} + i\beta\wedge\beta^\star$ given by (\ref{eqn:curvature_quot-bundle}), that we compose with $g^\star$ on the left and $g$ on the right to get the first identity claimed under under $(3)(b)$. The second identity in $(3)(b)$ follows from the first after taking $\operatorname{Tr}_E$ and observing that $\operatorname{Tr}_E\bigg(g^\star\circ i\Theta_h(E)_{|Q}\circ g\bigg) = \operatorname{Tr}_Q\bigg(i\Theta_h(E)_{|Q}\bigg)$ and \begin{eqnarray*}\operatorname{Tr}_E\bigg(g^\star\circ i\Theta_{h_Q}(Q)\circ g\bigg) = \operatorname{Tr}_Q\bigg(i\Theta_{h_Q}(Q)\bigg) = i\Theta_{\det h_Q}(\det Q).\end{eqnarray*} \hfill $\Box$

\vspace{2ex}

Now, recall that $\det{\cal F}$, defined to be the double dual $(\Lambda^s{\cal F})^{\vee\vee}$ of the top exterior power of ${\cal F}$, is locally free on the whole of $X$. Moreover, it is the unique extension across $S$ of the holomorphic line bundle $\Lambda^sF$ (a priori defined on $X\setminus S$). The role of the next statement is to translate the Uhlenbeck-Yau projection formalism -- which involves objects that are $C^\infty$ only on $X\setminus S$ and have only $L^2_1$-regularity on $X$ -- into globally $C^\infty$ objects on the whole of $X$, by passing from ${\cal F}$ to its holomorphic determinant line bundle.

\begin{Cor}\label{Cor:Lambda_s_exact-seq} The setting is the same as in Corollary \ref{Cor:pi_j_g_curvature}. Let $s\in\{1,\dots , r-1\}$ be the rank of ${\cal F}$. We consider the following short exact sequence of holomorphic vector bundles on $X$: \begin{eqnarray}\label{eqn:det_F_short-exact-seq}0\longrightarrow \det{\cal F} {\stackrel{\Lambda^s j}{\longrightarrow}} \Lambda^sE {\stackrel{\widetilde{g}}{\longrightarrow}} \widetilde{Q}:= \Lambda^sE/\det {\cal F}\longrightarrow 0.\end{eqnarray}

  Let $\widetilde\beta\in C^\infty_{1,\,0}(X,\,\operatorname{Hom}(\det{\cal F},\,\widetilde{Q}))$ be the second fundamental form of $\det{\cal F}$ in $\Lambda^sE$ and let $\sigma_{\cal F}\in H^0(X,\,\Lambda^sE\otimes\det{\cal F}^\star)$ be the canonical section associated with ${\cal F}$ as in Fact \ref{Fact:canonical-section}.

  We equip $\Lambda^sE$ with the $C^\infty$ Hermitian fibre metric $\Lambda^sh$ induced by $h$. We equip $\det{\cal F}$ and $\widetilde{Q}$ with the fibre metrics $h_{\det\cal F}$, respectively $h_{\widetilde{Q}}$, induced by $\Lambda^sh$. We equip the holomorphic vector bundle $A:=\Lambda^sE\otimes\det{\cal F}^\star$ with the fibre metric $h_A:=\Lambda^sh\otimes h_{\det\cal F}^\star$ induced by $\Lambda^sh$ and $h_{\det\cal F}$. Let $D_A:=D'_A + \bar\partial$ be the corresponding Chern connection.  

\vspace{1ex}

Then, the following identities hold everywhere on $X$:

\vspace{1ex}

$(1)$\, (a)\, $\sigma_{\cal F} = \Lambda^s j$; \hspace{5ex} (b)\, $\Lambda^s\pi = \sigma_{\cal F}\circ\sigma_{\cal F}^\star\in C^\infty(X,\,\operatorname{End}(\Lambda^sE))$;   \hspace{5ex} (c)\, $\sigma_{\cal F}^\star\circ\sigma_{\cal F} = \operatorname{Id}_{\det\cal F}$;

\vspace{1ex}

\hspace{3ex} (d)\, $\operatorname{Id}_{\Lambda^sE} - \Lambda^s\pi = \widetilde{g}^\star\circ\widetilde{g}$; \hspace{5ex} (e)\, $\widetilde{g}\circ\widetilde{g}^\star = \operatorname{Id}_{\widetilde{Q}}$;

\vspace{1ex}

$(2)$\, (a)\, $D'_A\sigma_{\cal F} = \widetilde{g}^\star\circ\widetilde\beta$, \hspace{2ex} hence \hspace{2ex} $\widetilde\beta = \widetilde{g}\circ D'_A\sigma_{\cal F}$;
  
\vspace{1ex}

\hspace{3ex} (b)\, $\bar\partial\sigma_{\cal F}^\star = \widetilde\beta^\star\circ\widetilde{g}$, \hspace{4ex} hence \hspace{4ex} $\widetilde\beta^\star = (\bar\partial\sigma_{\cal F}^\star)\circ\widetilde{g}^\star$;

\vspace{2ex}

$(3)$\, $\displaystyle\sigma_{\cal F}\circ i\Theta_{h_{\det\cal F}}(\det{\cal F})\circ\sigma_{\cal F}^\star = (\Lambda^s\pi)\circ i\Theta_{\Lambda^sh}(\Lambda^sE)\circ(\Lambda^s\pi) + i D''_{\operatorname{End}(\Lambda^sE)}(\Lambda^s\pi)\wedge D'_{\operatorname{End}(\Lambda^sE)}(\Lambda^s\pi)$;

\vspace{2ex}

$(4)$\, $\displaystyle\bigg\{i\Theta_{h_A}(A)\sigma_{\cal F},\,\sigma_{\cal F}\bigg\}_{h_A} = i\,\{\widetilde\beta,\,\widetilde\beta\}_{h_{\operatorname{Hom}(\det{\cal F},\,\widetilde{Q})}} = i\,\{D'_A\sigma_{\cal F},\,D'_A\sigma_{\cal F}\}_{h_A} \geq 0$,

\noindent where $h_{\operatorname{Hom}(\det{\cal F},\,\widetilde{Q})}$ is the $C^\infty$ fibre metric induced on $\operatorname{Hom}(\det{\cal F},\,\widetilde{Q})$ by $h_{\det\cal F}$ and $h_{\widetilde{Q}}$.

\end{Cor}  

\noindent {\it Proof.} $(1)$\, Since $j$ is the inclusion map of ${\cal F}$ into ${\cal O}(E)$, $\Lambda^sj$ is the inclusion map of $\det{\cal F}$ into $\Lambda^sE$. It is intially defined as a holomorphic map of holomorphic vector bundles over $X\setminus S$ and then it extends holomorphically and uniquely across $S$. As this is the definition of $\sigma_{\cal F}$, we get identity (a). 

Identities (b)-(e) are the analogues for the short exact sequence considered in this corollary of their counterparts of Corollary \ref{Cor:pi_j_g_curvature}.

\vspace{1ex}

$(2)$\, (a) is the analogue in this case of the standard first identity in (\ref{eqn:j-g_adjoints_prop}), with $j$ replaced by $\Lambda^sj= \sigma_{\cal F}$, $g$ replaced by $\widetilde{g}$ and $\beta$ replaced by $\widetilde\beta$. Similarly, (b) is the analogue of the second identity on the third line of (\ref{eqn:j-g_adjoints_prop}).

\vspace{1ex}

$(3)$\, The analogue of the first identity in $(3)$(a) of Corollary \ref{Cor:pi_j_g_curvature} reads: \begin{eqnarray*}(\Lambda^sj)\circ i\Theta_{h_{\det F}}(\det{\cal F})\circ (\Lambda^sj)^\star & = & (\Lambda^sj)\circ\bigg(i\Theta_{\Lambda^sh}(\Lambda^sE)_{|\det {\cal F}}\bigg)\circ(\Lambda^sj)^\star \\
  & + & iD''_{\operatorname{End}(\Lambda^sE)}(\Lambda^s\pi)\wedge D'_{\operatorname{End}(\Lambda^sE)}(\Lambda^s\pi).\end{eqnarray*} This proves $(3)$ after writing: $\Lambda^sj = \sigma_{\cal F}$ and $i\Theta_{\Lambda^sh}(\Lambda^sE)_{|\det {\cal F}} = (\Lambda^sj)^\star\circ i\Theta_{\Lambda^sh}(\Lambda^sE)\circ (\Lambda^sj)$ and $(\Lambda^sj)\circ(\Lambda^sj)^\star = \sigma_{\cal F}\circ\sigma_{\cal F}^\star = \Lambda^s\pi$.

\vspace{1ex}

$(4)$\, We will use the following elementary fact. If $T,S:(V,\,h_V)\longrightarrow(W,\,h_W)$ are $\C$-linear maps of Hermitian $\C$-vector spaces, when we view $T,S\in V^\star\otimes W$ we have the following identity: \begin{eqnarray}\label{eqn:T-S-S-star-T-Id_inner-prod}\langle T,\,S\rangle_{h_{V^\star\otimes W}} = \langle S^\star T,\,\operatorname{Id}_V\rangle_{h_{\operatorname{End}V}},\end{eqnarray} $h_{V^\star\otimes W}$ and $h_{\operatorname{End}V}$ being the inner products induced on the respective vector spaces by $h_V$ and $h_W$.

Applying this to our case, we get: \begin{eqnarray}\label{Cor:Lambda_s_exact-seq_proof_1}\bigg\{i\Theta_{h_A}(A)\sigma_{\cal F},\,\sigma_{\cal F}\bigg\}_{h_A} = \bigg\{\sigma_{\cal F}^\star\circ i\Theta_{h_A}(A)\sigma_{\cal F},\,\operatorname{Id}_{\det{\cal F}}\bigg\}_{h_{\operatorname{End}(\det{\cal F})}},\end{eqnarray} where $i\Theta_{h_A}(A)\sigma_{\cal F}\in C^\infty_{1,\,1}(X,\,A)\simeq C^\infty_{1,\,1}(X,\,\operatorname{Hom}(\det{\cal F},\,\Lambda^sE))$ is obtained by evaluating the $(\operatorname{End}A)$-part of $i\Theta_{h_A}(A)$ on the section $\sigma_{\cal F}$ of $A$.

On the other hand, a standard formula yields: \begin{eqnarray*}i\Theta_{h_A}(A) = i\Theta_{\Lambda^sh}(\Lambda^sE)\otimes\operatorname{Id}_{\det{\cal F}^\star} + \operatorname{Id}_{\Lambda^sE}\otimes i\Theta_{h_{\det{\cal F}^\star}}(\det{\cal F}^\star).\end{eqnarray*} Hence: \begin{eqnarray}\label{Cor:Lambda_s_exact-seq_proof_2}\nonumber\sigma_{\cal F}^\star\circ i\Theta_{h_A}(A)\sigma_{\cal F} & = & \sigma_{\cal F}^\star\circ i\Theta_{\Lambda^sh}(\Lambda^sE)\circ\sigma_{\cal F} - \bigg(\sigma_{\cal F}^\star\circ\operatorname{Id}_{\Lambda^sE}\circ\sigma_{\cal F}\bigg)\otimes i\Theta_{h_{\det{\cal F}}}(\det{\cal F}) \\
\nonumber  & = & i\Theta_{\Lambda^sh}(\Lambda^sE)_{|\det{\cal F}} - (\sigma_{\cal F}^\star\circ\sigma_{\cal F})\otimes i\Theta_{h_{\det{\cal F}}}(\det{\cal F}) \\
& = & \bigg(i\Theta_{h_{\det{\cal F}}}(\det{\cal F}) - i\widetilde\beta^\star\wedge\widetilde\beta\bigg) - i\Theta_{h_{\det{\cal F}}}(\det{\cal F}) = - i\,\widetilde\beta^\star\wedge\widetilde\beta,\end{eqnarray} where we have used the identity $\sigma_{\cal F} = \Lambda^sj$ to go from the first to the second lines and then the analogue in the context of the short exact sequence (\ref{eqn:det_F_short-exact-seq}) of the standard identity (\ref{eqn:curvature_subbundle}) and the identity $\sigma_{\cal F}^\star\circ\sigma_{\cal F} = \operatorname{Id}_{\det\cal F}$ to go from the second to the third lines.

Putting together (\ref{Cor:Lambda_s_exact-seq_proof_1}) and (\ref{Cor:Lambda_s_exact-seq_proof_2}), we get: \begin{eqnarray*}\bigg\{i\Theta_{h_A}(A)\sigma_{\cal F},\,\sigma_{\cal F}\bigg\}_{h_A} = -\bigg\{i\,\widetilde\beta^\star\wedge\widetilde\beta,\,\operatorname{Id}_{\det{\cal F}}\bigg\}_{h_{\operatorname{End}(\det{\cal F})}} = -\bigg(-i\,\{\widetilde\beta,\,\widetilde\beta\}_{h_{\operatorname{Hom}(\det{\cal F},\,\widetilde{Q})}}\bigg)\geq 0,\end{eqnarray*} where to get the last equality we reapplied the elementary identity (\ref{eqn:T-S-S-star-T-Id_inner-prod}) with an extra $(-1)$-factor introduced by the fcat that $\widetilde\beta$ is an odd-degree form (a $1$-form). The last inequality is standard for every vector-bundle-valued $(1,\,0)$-form.

Now, using the second identity in $(2)$(a) of the statement, we further get: \begin{eqnarray*}\bigg\{i\Theta_{h_A}(A)\sigma_{\cal F},\,\sigma_{\cal F}\bigg\}_{h_A} = i\,\bigg\{\widetilde{g}\circ D_A'\sigma_{\cal F},\,\widetilde{g}\circ D_A'\sigma_{\cal F}\bigg\}_{h_{\operatorname{Hom}(\det{\cal F},\,\widetilde{Q})}} =  i\,\{D'_A\sigma_{\cal F},\,D'_A\sigma_{\cal F}\}_{h_A},\end{eqnarray*} where the last equality follows from the first identity in $(2)$(a) of the statement -- which ensures that $D'_A\sigma_{\cal F}$ is the element of minimal norm in the $\widetilde{g}$-fibre above $\widetilde\beta$ -- and the definition of the quotient metric.   \hfill $\Box$

\vspace{2ex}

Note that $(4)$ in Corollary \ref{Cor:Lambda_s_exact-seq} implies that if one were to equip $\det{\cal F}^\star$ with the fibre metric induced by the given fibre metric $h$ on $E$, then the analogue of the function $Z_{\omega,\,\Omega,\,h}^{(m)}({\cal F})$ introduced in Definition \ref{Def:Z-function} would be non-negative everywhere on $X$, for every coherent subsheaf ${\cal F}\subset{\cal O}(E)$. In particular, every Hermitian holomorphic vector bundle $(E,\,h)$ would be uniformly stable in the sense of Definition \ref{Def:m-pos-stability}.

The key point is that, in our definitions, $\det{\cal F}^\star$ is instead equipped with the (unique, up to a positive multiplicative constant) $(\omega,\,\Omega)$-Hermite-Einstein fibre metric. Consequently, our notions of stability depend on the discrepancy between these two fibre metrics on $\det{\cal F}^\star$, as spelt out in the following result.

\begin{The-Def}\label{The-Def:discrepancy-equation} Let ${\cal F}$ be a coherent subsheaf of rank $s\in\{1,\dots , r-1\}$ of a Hermitian holomorphic vector bundle $(E,\,h)$ of rank $r\geq 1$ over an $n$-dimensional compact complex manifold $X$. We suppose that $X$ carries forms $\omega$ and $\Omega$ satisfying the properties of Definition \ref{Def:m-pos-stability}.

  Let $h_{\det{\cal F}}^{H-E}$ be the (unique, up to a positive multiplicative constant) {\bf $(\omega,\,\Omega)$-Hermite-Einstein} fibre metric on the line bundle $\det{\cal F}$. (It was also denoted by $\det h_{{\cal F}}$ earlier.) Let $h_{\det{\cal F}}$ be the fibre metric on $\det{\cal F}$ induced by the fibre metric $\Lambda^sh$ of $\Lambda^sE$ when viewing $\det{\cal F}$ as a holomorphic {\bf subbundle} of $\Lambda^sE$. The unique $C^\infty$ function $f = f_{\cal F}:X\longrightarrow\R$ such that \begin{eqnarray}\label{eqn:discrepancy-fucntion_def}h_{\det{\cal F}}^{H-E} = e^{-f}\,h_{\det{\cal F}}\end{eqnarray} is called the {\bf discrepancy function} of the triple $(E,\,h,\,{\cal F})$.

  We equip the holomorphic vector bundle $A:=\Lambda^sE\otimes\det{\cal F}^\star$ with the fibre metric $h_A:=\Lambda^sh\otimes h_{\det\cal F}^\star$ induced by $\Lambda^sh$ and $h_{\det\cal F}$. On the other hand, we equip $A$ with the fibre metric $h_A^{H-E}:=\Lambda^sh\otimes (h_{\det\cal F}^{H-E})^\star$ induced by $\Lambda^sh$ and $h_{\det\cal F}^{H-E}$.

\vspace{1ex}

Then, the following identities hold: \begin{eqnarray}\label{eqn:discrepancy-equation_1}\widehat{Z}_{\omega,\,\Omega,\,h}^{(m)}({\cal F}) = e^{-f}\,\frac{Z_{\omega,\,\Omega,\,h}^{(m)}({\cal F})}{|\sigma_{{\cal F}}|^2_{h_A}} & = & P_{\omega,\,\Omega}(f) + \frac{1}{|\sigma_{{\cal F}}|^2_{h_A}}\,\frac{i\,\{D'_A\sigma_{\cal F},\,D'_A\sigma_{\cal F}\}_{h_A}\wedge\omega^{m-1}\wedge\Omega}{dV_\omega} \\
  \label{eqn:discrepancy-equation_2} & = & P_{\omega,\,\Omega}(f) - Q(\pi),\end{eqnarray} where $Q(\pi):=$ \begin{eqnarray*}\frac{1}{|\sigma_{{\cal F}}|^2_{h_A}}\,\frac{\bigg\{\sigma_{{\cal F}}^\star\circ(\Lambda^s\pi)\circ\bigg(iD''_{\operatorname{End}(\Lambda^sE)}(\Lambda^s\pi)\wedge D'_{\operatorname{End}(\Lambda^sE)}(\Lambda^s\pi)\bigg)\circ(\Lambda^s\pi)\circ\sigma_{{\cal F}},\,\operatorname{Id}_{\det{\cal F}}\bigg\}_{h_{\operatorname{End}({\det\cal F})}}\wedge\omega^{m-1}\wedge\Omega}{dV_\omega}\end{eqnarray*} \hspace{10ex} $\leq 0$,

\noindent where $h_{\operatorname{End}({\det\cal F})}$ is the fibre metric induced on $\operatorname{End}({\det\cal F})$ by $h_{\det\cal F}$ and $P_{\omega,\,\Omega}$ is the elliptic second-order differential operator of (\ref{eqn:P_operator_def}). As in Corollary \ref{Cor:Lambda_s_exact-seq}, $D_A:=D'_A + \bar\partial$ denotes the Chern connection of $(A,\,h_A)$. 

\end{The-Def} 

\noindent {\it Proof.} The first equality in (\ref{eqn:discrepancy-equation_1}) follows at once from the definitions. We are left with the other two.

From (\ref{eqn:discrepancy-fucntion_def}), we get $h_A^{H-E} = e^f\,h_A$, hence:  \begin{eqnarray*}i\Theta_{h_A^{H-E}}(A) = i\Theta_{h_A}(A) - i\partial\bar\partial f\otimes\operatorname{Id}_{\operatorname{End}(A)}.\end{eqnarray*} From this and from $(4)$ of Corollary \ref{Cor:Lambda_s_exact-seq}, we get: \begin{eqnarray*}\bigg\{i\Theta_{h_A^{H-E}}(A)\sigma_{\cal F},\,\sigma_{\cal F}\bigg\}_{h_A^{H-E}} = e^f\,i\,\{D'_A\sigma_{\cal F},\,D'_A\sigma_{\cal F}\}_{h_A} - e^f\,|\sigma_{\cal F}|^2_{h_A}i\partial\bar\partial f.\end{eqnarray*} Multiplying the last equality by $\omega^{m-1}\wedge\Omega$ and then dividing the result by $dV_\omega$, we get: \begin{eqnarray*}Z_{\omega,\,\Omega,\,h}^{(m)}({\cal F}) & = & \frac{\bigg\{i\Theta_{h_A^{H-E}}(A)\sigma_{\cal F},\,\sigma_{\cal F}\bigg\}_{h_A^{H-E}}\wedge\omega^{m-1}\wedge\Omega}{dV_\omega} \\
  & = & e^f\,\frac{i\,\{D'_A\sigma_{\cal F},\,D'_A\sigma_{\cal F}\}_{h_A}\wedge\omega^{m-1}\wedge\Omega}{dV_\omega} - e^f\,|\sigma_{\cal F}|^2_{h_A}\,\frac{i\partial\bar\partial f\wedge\omega^{m-1}\wedge\Omega}{dV_\omega}.\end{eqnarray*} This proves (\ref{eqn:discrepancy-equation_1}).

\vspace{1ex}

To deduce (\ref{eqn:discrepancy-equation_2}) from (\ref{eqn:discrepancy-equation_1}), we will express the above quantities in terms of $\Lambda^s\pi$.

In the proof of $(3)$ of Corollary \ref{Cor:Lambda_s_exact-seq}, we obtained the first equality below: \begin{eqnarray*}iD''_{\operatorname{End}(\Lambda^sE)}(\Lambda^s\pi)\wedge D'_{\operatorname{End}(\Lambda^sE)}(\Lambda^s\pi) & = & \sigma_{\cal F}\circ\bigg(i\Theta_{h_{\det F}}(\det{\cal F}) - i\Theta_{\Lambda^sh}(\Lambda^sE)_{|\det {\cal F}}\bigg)\circ\sigma_{\cal F}^\star \\
  & = & \sigma_{\cal F}\circ\bigg(i\,\widetilde\beta^\star\wedge\widetilde\beta\bigg)\circ\sigma_{\cal F}^\star,\end{eqnarray*} where the second equality follows from the analogue of the standard identity (\ref{eqn:curvature_subbundle}) in the context of the short exact sequence (\ref{eqn:det_F_short-exact-seq}). Now, composing on either side with $\Lambda^s\pi = \sigma_{\cal F}\circ\sigma_{\cal F}^\star$ and using the identity $\sigma_{\cal F}^\star\circ\sigma_{\cal F} = \operatorname{Id}_{\det{\cal F}}$, we get: \begin{eqnarray*}(\Lambda^s\pi)\circ\bigg(iD''_{\operatorname{End}(\Lambda^sE)}(\Lambda^s\pi)\wedge D'_{\operatorname{End}(\Lambda^sE)}(\Lambda^s\pi)\bigg)\circ(\Lambda^s\pi) =  \sigma_{\cal F}\circ\bigg(i\,\widetilde\beta^\star\wedge\widetilde\beta\bigg)\circ\sigma_{\cal F}^\star.\end{eqnarray*} Further composing with $\sigma_{\cal F}^\star$ on the left and $\sigma_{\cal F}$ on the right and re-using the identity $\sigma_{\cal F}^\star\circ\sigma_{\cal F} = \operatorname{Id}_{\det{\cal F}}$, we get: \begin{eqnarray*}\sigma_{\cal F}^\star\circ(\Lambda^s\pi)\circ\bigg(iD''_{\operatorname{End}(\Lambda^sE)}(\Lambda^s\pi)\wedge D'_{\operatorname{End}(\Lambda^sE)}(\Lambda^s\pi)\bigg)\circ(\Lambda^s\pi)\circ\sigma_{\cal F} =  i\,\widetilde\beta^\star\wedge\widetilde\beta.\end{eqnarray*} Hence: \begin{eqnarray*} & & \bigg\{\sigma_{\cal F}^\star\circ(\Lambda^s\pi)\circ\bigg(iD''_{\operatorname{End}(\Lambda^sE)}(\Lambda^s\pi)\wedge D'_{\operatorname{End}(\Lambda^sE)}(\Lambda^s\pi)\bigg)\circ(\Lambda^s\pi)\circ\sigma_{\cal F},\,\operatorname{Id}_{\det{\cal F}}\bigg\}_{h_{\operatorname{End}({\det\cal F})}} \\
  & = & -i\,\{\widetilde\beta,\,\widetilde\beta\}_{h_{\operatorname{Hom}(\det{\cal F},\,\widetilde{Q})}} = -i\,\{D'_A\sigma_{\cal F},\,D'_A\sigma_{\cal F}\}_{h_A},\end{eqnarray*} where the last equality was obtained at the end of the proof of Corollary \ref{Cor:Lambda_s_exact-seq}.

This shows that (\ref{eqn:discrepancy-equation_2}) follows from (\ref{eqn:discrepancy-equation_1}). \hfill $\Box$

\vspace{2ex}

We now see explicitly that if a given Hermitian fibre metric $h$ on $E$ induces an $(\omega,\,\Omega)$-Hermite-Einstein metric on the determinant line bundle of every coherent subsheaf ${\cal F}\subset{\cal O}(E)$ (equivalently, if the discrepancy function $f$ is constant for every ${\cal F}$), then $P_{\omega,\,\Omega}(f)\equiv 0$, hence $Z_{\omega,\,\Omega,\,h}^{(m)}({\cal F})\geq 0$ everywhere on $X$, for every such ${\cal F}$, making $(E,\,h)$ into a {\it uniformly semi-stable} vector bundle.

\begin{Cor}\label{Cor:discrepancy-equation} The setting is the same as in Theorem and Definition \ref{The-Def:discrepancy-equation}.

\vspace{1ex}  

 (i)\, The Hermitian holomorphic vector bundle $(E,\,h)$ is {\bf uniformly stable} (respectively {\bf uniformly semi-stable}) if and only if \begin{eqnarray}\label{eqn:discrepancy-inequality}P_{\omega,\,\Omega}(f) > Q(\pi)  \hspace{3ex} \bigg(\mbox{respectively} \hspace{2ex} P_{\omega,\,\Omega}(f) \geq Q(\pi)\bigg) \hspace{2ex} \mbox{on} \hspace{1ex} X, \hspace{1ex} \mbox{for every} \hspace{1ex} {\cal F}\subset{\cal O}(E),\end{eqnarray} where $Q(\pi):X\longrightarrow(-\infty,\, 0]$ is a $C^\infty$ function depending on $h$ and ${\cal F}$.

\vspace{1ex}  

 (ii)\, Let ${\cal F}\subset{\cal O}(E)$ be a coherent subsheaf of rank $s\in\{1,\dots , r-1\}$. If $Z_{\omega,\,\Omega,\,h}^{(m)}({\cal F})\equiv 0$, then:

\vspace{1ex}  

\hspace{2ex} (a)\, the fibre metric $h_{\det{\cal F}}$ induced on $\det{\cal F}$ by the fibre metric $\Lambda^sh$ of $\Lambda^sE$ (when viewing $\det{\cal F}$ as a holomorphic subbundle of $\Lambda^sE$) is {\bf $(\omega,\,\Omega)$-Hermite-Einstein};

\vspace{1ex}  

\hspace{2ex} (b)\, $D^{H-E}_A\sigma_{{\cal F}}\equiv 0$ \hspace{3ex} and \hspace{3ex} $D_A\sigma_{{\cal F}}\equiv 0$;

\vspace{1ex}

\noindent (i.e. $\sigma_{{\cal F}}$ is {\bf parallel} with respect to the Chern connections $D^{H-E}_A = (D^{H-E}_A)' + \bar\partial$ of $(A,\,h_A^{H-E})$ and $D_A = D_A' + \bar\partial$ of $(A,\,h_A)$.)

\vspace{1ex}  

(iii)\, Let ${\cal F}\subset{\cal O}(E)$ be a coherent subsheaf of rank $s\in\{1,\dots , r-1\}$. If the fibre metric $h_{\det{\cal F}}$ induced on $\det{\cal F}$ by the fibre metric $\Lambda^sh$ of $\Lambda^sE$ (when viewing $\det{\cal F}$ as a holomorphic subbundle of $\Lambda^sE$) is {\bf $(\omega,\,\Omega)$-Hermite-Einstein}, then  $Z_{\omega,\,\Omega,\,h}^{(m)}({\cal F})\geq 0$ everywhere on $X$.

\vspace{1ex}  

(iv)\, Suppose that $(E,\,h)$ is {\bf not uniformly semi-stable} and let ${\cal F}\subset{\cal O}(E)$ be a {\bf destabilising} coherent subsheaf of rank $s\in\{1,\dots , r-1\}$. Then, the {\bf destabilising open subset} \begin{eqnarray}\label{eqn:U_F_inclusion}U_{\cal F}:=\bigg\{x\in X\,\mid\, \widehat{Z}_{\omega,\,\Omega,\,h}^{(m)}({\cal F})(x)<0\bigg\}\subset\bigg\{x\in X\,\mid\, P_{\omega,\,\Omega}(f)(x)<0\bigg\} \subset X\end{eqnarray} is non-empty and every point of $X$ at which the discrepancy function $f = f_{\cal F}:X\longrightarrow\R$ attains a local maximum lies outside $U_{\cal F}$. 


\end{Cor}   

\noindent {\it Proof.} Only (ii), (iii) and (iv) still need proofs.

\vspace{1ex}

(ii)\, We have already proved conclusion (ii)(b) in Proposition \ref{Prop:Z-function_alternative} under the weaker assumption $\int_XZ_{\omega,\,\Omega,\,h}^{(m)}({\cal F})\,dV_\omega = 0$. From (\ref{eqn:discrepancy-equation_1}) we deduce that the hypothesis $Z_{\omega,\,\Omega,\,h}^{(m)}({\cal F})\equiv 0$ implies \begin{eqnarray*}P_{\omega,\,\Omega}(f) = Q(\pi)\leq 0   \hspace{10ex} \mbox{everywhere on}\hspace{1ex} X,\end{eqnarray*} where the last inequality was seen in Theorem and Definition \ref{The-Def:discrepancy-equation} to always hold. 

  As in the proof of theorem \ref{The:Kobayashi-vanishing_gen}, the ellipticity of $P_{\omega,\,\Omega}$, the compactness of $X$ and the maximum principle yield: \begin{eqnarray*}P_{\omega,\,\Omega}(f)\equiv 0 \hspace{5ex}\mbox{and}\hspace{5ex} f = \mbox{Const}.\end{eqnarray*} The constancy of $f$ implies that $h_{\det{\cal F}}$ is $(\omega,\,\Omega)$-Hermite-Einstein. This proves (ii)(a).

  We also get $Q(\pi)\equiv 0$, which is equivalent, thanks to (\ref{eqn:discrepancy-equation_1}) and (\ref{eqn:discrepancy-equation_2}), to $i\,\{D'_A\sigma_{\cal F},\,D'_A\sigma_{\cal F}\}_{h_A}\wedge\omega^{m-1}\wedge\Omega = 0$. As noted in the proof of Proposition \ref{Prop:Z-function_alternative}, this implies $D'_A\sigma_{\cal F}\equiv 0$.

  Meanwhile, since $f$ is constant, the fibre metrics $h_A^{H-E}$ and $h_A$ differ by a positive multiplicative constant, so we also have $(D^{H-E}_A)'\sigma_{\cal F}\equiv 0$. On the other hand, we already know that $\bar\partial\sigma_{\cal F} \equiv 0$.

\vspace{1ex}

(iii)\, The hypothesis means that the discrepancy function $f$ of the triple $(E,\,h,\,{\cal F})$ is {\it constant} (because the $(\omega,\,\Omega)$-Hermite-Einstein metric on $\det{\cal F}$ is unique up to a positive multiplicative constant). Therefore, $P_{\omega,\,\Omega}(f)\equiv 0$, so (\ref{eqn:discrepancy-equation_2}) reduces to: \begin{eqnarray*}e^{-f}\,\frac{Z_{\omega,\,\Omega,\,h}^{(m)}({\cal F})}{|\sigma_{{\cal F}}|^2_{h_A}} = - Q(\pi)\geq 0,\end{eqnarray*} at every point of $X$. This amounts to $Z_{\omega,\,\Omega,\,h}^{(m)}({\cal F})\geq 0$ everywhere on $X$.

\vspace{1ex}

(iv)\, Since $(E,\,h)$ is not uniformly semi-stable, there exists a proper coherent subsheaf ${\cal F}\subset{\cal O}(E)$ such that \begin{eqnarray*}c_h({\cal F})=\min\limits_X\widehat{Z}_{\omega,\,\Omega,\,h}^{(m)}({\cal F}) = \min\limits_X\bigg(P_{\omega,\,\Omega}(f) - Q(\pi)\bigg) <0.\end{eqnarray*} Thus, there exists $x_0\in X$ at which $\widehat{Z}_{\omega,\,\Omega,\,h}^{(m)}({\cal F})$ is negative. By continuity, $\widehat{Z}_{\omega,\,\Omega,\,h}^{(m)}({\cal F})$ is negative on a whole neighbourhood of $x_0$ and the subset $U_{\cal F}$ is open  in $X$ and contains $x_0$, hence it is non-empty.

Now, let $a\in X$ be such that $f$ attains a (local) maximum at $a$. (By continuity of $f$ and compactness of $X$, $f$ attains a global maximum at least at one point of $X$.) The maximum principle applied to the elliptic second-order differential operator $P_{\omega,\,\Omega}$ with no zero$^{th}$-order terms yields: \begin{eqnarray*}0\leq P_{\omega,\,\Omega}(f)(a).\end{eqnarray*} If we had $a\in U_{\cal F}$, then we would also have: $P_{\omega,\,\Omega}(f)(a)<Q(\pi)(a)\leq 0$, a contradiction.

As for the inclusion of $U_{\cal F}$ stated in (\ref{eqn:U_F_inclusion}), we note that whenever $\widehat{Z}_{\omega,\,\Omega,\,h}^{(m)}({\cal F})(x) = P_{\omega,\,\Omega}(f)(x) - Q(\pi)(x)< 0$, we get $P_{\omega,\,\Omega}(f)(x) < Q(\pi)(x)\leq 0$. The contention follows. \hfill $\Box$

\vspace{2ex}

The following corollary is the analogue for the normalised stability function $\widehat{Z}_{\omega,\,\Omega,\,h}^{(m)}({\cal F})$ (and the induced fibre metric $h_A$ on $A$) of Proposition \ref{Prop:Z-function_alternative} (which dealt with $Z_{\omega,\,\Omega,\,h}^{(m)}({\cal F})$ and $h_A^{H-E}$).

\begin{Cor}\label{Cor:discrepancy-equation_Z-hat} The setting is the same as in Theorem and Definition \ref{The-Def:discrepancy-equation}. We always have: \begin{eqnarray}\label{eqn:Z-hat-function_integral}\int\limits_X \widehat{Z}_{\omega,\,\Omega,\,h}^{(m)}({\cal F})\,dV_\omega \geq 0.\end{eqnarray}

   Moreover, if $\int_X\widehat{Z}_{\omega,\,\Omega,\,h}^{(m)}({\cal F})\,dV_\omega = 0$, then $D_A\sigma_{{\cal F}}\equiv 0$ (i.e. $\sigma_{{\cal F}}$ is parallel with respect to the Chern connection $D_A = D_A' + \bar\partial$ of $(A,\,h_A)$).

\end{Cor}

\noindent {\it Proof.} Integrating identity (\ref{eqn:discrepancy-equation_2}), we get the first equality below: \begin{eqnarray*}\int\limits_X \widehat{Z}_{\omega,\,\Omega,\,h}^{(m)}({\cal F})\,dV_\omega = \int\limits_X  P_{\omega,\,\Omega}(f)\,dV_\omega - \int\limits_X  Q(\pi)\,dV_\omega = - \int\limits_X  Q(\pi)\,dV_\omega\geq 0,\end{eqnarray*} where the second equality is a consequence of the following $L^2_\omega$-orthogonal decomposition: \begin{eqnarray*}C^\infty(X,\,\R) = \R\oplus\mbox{Im}\,P_{\omega,\,\Omega}\end{eqnarray*} given in Corollary 2.9. of [DP25] (and proved there as a consequence of the fact that the adjoint operator $P_{\omega,\,\Omega}^\star$ has no zero$^{th}$-order terms, combined with the ellipticity of $P_{\omega,\,\Omega}$ and the compactness of $X$). The last inequality follows from the pointwise inequality $-Q(\pi)\geq 0$ on $X$ (seen above).

This proves (\ref{eqn:Z-hat-function_integral}) and shows that it becomes an equality if and only if $Q(\pi) = 0$ everywhere on $X$. We have seen in the proof of (ii) of Corollary \ref{Cor:discrepancy-equation} that this implies $D'_A\sigma_{\cal F}\equiv 0$. Since $\sigma_{\cal F}$ is already holomorphic, we have $\bar\partial\sigma_{\cal F} =0$, hence $D_A\sigma_{\cal F} = 0$ everywhere on $X$.  \hfill $\Box$

\section{Link with the Hermite-Einstein condition}\label{section:H-E_link}

We now investigate how the notion of $m$-positive semi-stability interacts with the $(\omega,\,\Omega)$-Hermite-Einstein condition introduced in [Pop25, Definition 2.1.].

\begin{The}\label{The:H-E_implies_m-pos-stability} In the setting of Definition \ref{Def:m-pos-stability}, suppose there exists an {\bf $(\omega,\,\Omega)$-Hermite-Einstein} $C^\infty$ Hermitian fibre metric $h$ on $E$. Then $(E,\,h)$ is {\bf uniformly $(\omega,\,\Omega)$-$m$-positively semi-stable}.

  Moreover, if $(E,\,h)$ is not uniformly $(\omega,\,\Omega)$-$m$-positively stable, then $E$ is a holomorphic, $h$-orthogonal direct sum \begin{eqnarray*}(E,\,h) = \bigoplus\limits_i (E_i,\,h_i)\end{eqnarray*} of {\bf uniformly $(\omega,\,\Omega)$-$m$-positively stable} holomorphic subbundles $(E_i,\,h_i)\subset (E,\,h)$, the $C^\infty$ Hermitian fibre metric $h_i$ induced by $h$ on each $E_i$ being {\bf $(\omega,\,\Omega)$-Hermite-Einstein} with the same Einstein factor as $(E,\,h)$.

\end{The}
  
\noindent {\it Proof.} The hypothesis means that there exists a $C^\infty$ Hermitian fibre metric $h$ on $E$ and a constant $\lambda_E\in\R$ such that \begin{eqnarray*}i\Theta_h(E)\wedge\omega^{m-1}\wedge\Omega = \lambda_E\,dV_\omega\otimes\mbox{Id}_E.\end{eqnarray*}

Let ${\cal F}\subset{\cal O}(E)$ be a coherent subsheaf  of rank $s\in\{1,\dots , r-1\}$. Since every holomorphic line bundle admits a (unique, up to a positive multiplicative constant) $(\omega,\,\Omega)$-Hermite-Einstein fibre metric and $\det{\cal F}^\star$ is a holomorphic line bundle, there exists a $C^\infty$ Hermitian fibre metric $h^{H-E}_{\det{\cal F}^\star} =\det h_{{\cal F}}^\star$ on $\det{\cal F}^\star$ and a constant $\lambda_{\det{\cal F}^\star}\in\R$ such that \begin{eqnarray*}i\Theta_{h^{H-E}_{\det{\cal F}^\star}}(\det{\cal F}^\star)\wedge\omega^{m-1}\wedge\Omega = \lambda_{\det{\cal F}^\star}\,dV_\omega.\end{eqnarray*}

Then, $h$ induces a metric $\Lambda^sh$ on $\Lambda^sE$. The latter metric and $h^{H-E}_{\det{\cal F}^\star}$ induce a metric $\Lambda^sh\otimes h^{H-E}_{\det{\cal F}^\star}$ on $\Lambda^sE\otimes\det{\cal F}^\star$ whose Chern curvature form satisfies the equality: \begin{eqnarray*}i\Theta_{\Lambda^sh\otimes h^{H-E}_{\det{\cal F}^\star}}(\Lambda^sE\otimes\det{\cal F}^\star)\wedge\omega^{m-1}\wedge\Omega = \bigg(s\lambda_E + \lambda_{\det{\cal F}^\star}\bigg)\,dV_\omega\otimes\mbox{Id}_E\end{eqnarray*} at every point of $X$. (Indeed, as in the classical case, $(E,\,h)$ being $(\omega,\,\Omega)$-Hermite-Einstein with Einstein factor $\lambda_E$ implies that $(\Lambda^sE,\,\Lambda^sh)$ is $(\omega,\,\Omega)$-Hermite-Einstein with Einstein factor $s\lambda_E$. Moreover, the Einstein factors add up in a tensor product of vector bundles. See Lemma 2.1.4 in [LT95] for the standard case and Lemma 2.7. in [Pop25] for the current case.)

  Combined with Definition \ref{Def:Z-function}, this yields: \begin{eqnarray}\label{The:H-E_implies_m-pos-stability_proof_1}Z_{\omega,\,\Omega,\,h}^{(m)}({\cal F}) = \bigg(s\lambda_E + \lambda_{\det{\cal F}^\star}\bigg)\,|\sigma_{{\cal F}}|^2_{\Lambda^sh\otimes h^{H-E}_{\det{\cal F}^\star}}\end{eqnarray} at every point of $X$. Since the canonical section $\sigma_{{\cal F}}$ is non-vanishing, $|\sigma_{{\cal F}}|^2> 0$ (we have dropped the subscript of the pointwise norm for the sake of brevity) everywhere on $X$. Therefore, from conclusion (\ref{eqn:Z-function_alternative}) of Proposition \ref{Prop:Z-function_alternative} we infer that the constant $s\lambda_E + \lambda_{\det{\cal F}^\star}$ is non-negative, hence $Z_{\omega,\,\Omega,\,h}^{(m)}({\cal F})\geq 0$ at every point of $X$.

  This proves that $E$ is {\bf $m$-positively semi-stable} with a {\bf uniform} $C^\infty$ Hermitian fibre metric $h$. Indeed, the $(\omega,\,\Omega)$-Hermite-Einstein metric $h$ on $E$ is independent of ${\cal F}\subset{\cal O}(E)$.

\vspace{1ex}

Now, suppose that $(E,\,h)$ is {\it not uniformly $m$-positively stable}. Then, there is a destabilising subsheaf, namely a coherent subsheaf ${\cal F}\subset{\cal O}(E)$ of some rank $s\in\{1,\dots , r-1\}$ such that $Z_{\omega,\,\Omega,\,h}^{(m)}({\cal F})\geq 0$ everywhere on $X$ (as we have seen in the first part of the proof) but there exists a point $x_0\in X$ such that \begin{eqnarray*}Z_{\omega,\,\Omega,\,h}^{(m)}({\cal F})(x_0) = 0.\end{eqnarray*}

Since (\ref{The:H-E_implies_m-pos-stability_proof_1}) holds (at every point of $X$, hence also at $x_0$) and $\sigma_{\cal F}(x_0)\neq 0$, we infer the vanishing of the constant $s\lambda_E + \lambda_{\det{\cal F}^\star}$. Hence \begin{eqnarray*}Z_{\omega,\,\Omega,\,h}^{(m)}({\cal F})\equiv 0.\end{eqnarray*} Thus, the last conclusion of Proposition \ref{Prop:Z-function_alternative} ensures that $\sigma_{{\cal F}}$ is parallel with respect to the Chern connection $D_{\Lambda^sh\otimes h^{H-E}_{\det{\cal F}^\star}} = D_{\Lambda^sh\otimes h^{H-E}_{\det{\cal F}^\star}}' + \bar\partial$ of the holomorphic Hermitian vector bundle $(\Lambda^s E\otimes\det{\cal F}^\star,\,\Lambda^sh\otimes h^{H-E}_{\det{\cal F}^\star})$. Hence, the holomorphic line bundle $\det{\cal F}$ must be a parallel subbundle of $\Lambda^s E$. This implies that the $(\omega,\,\Omega)$-Hermite-Einstein fibre metric $h^{H-E}_{\det{\cal F}} = (h^{H-E}_{\det{\cal F}^\star})^\star$ of $\det{\cal F}$ equals (up to a positive multiplicative constant) the metric induced from the one on $\Lambda^s E$ by the inclusion $\det\varphi:\det{\cal F}\hookrightarrow\Lambda^s E$, where $\varphi:{\cal F}\hookrightarrow{\cal O}(E)$ is the original inclusion.

From this point onwards, the proof proceeds in the same (by now standard) way as in the classical case treated by Kobayashi in [Kob82] and by L\"ubke in [Lub83], presented as Theorem 2.3.2 in [LT95] and, in the slightly more general context of this paper, as Theorem 2.15. in [Pop25].

\vspace{1ex}

Here are the main steps of the argument.

\vspace{1ex}

$\bullet$ One shows that the destabilising subsheaf ${\cal F}\subset{\cal O}(E)$ must be locally free and a direct summand of $E$, namely the quotient sheaf ${\cal O}(E)/{\cal F}$ is itself locally free and, denoting by $G$ the associated holomorphic vector bundle, one has a holomorphic direct-sum decomposition: \begin{eqnarray*}E = F\oplus G\end{eqnarray*} in which the subbundles $F$ and $G$ of $E$ are $(\omega,\,\Omega)$-Hermite-Einstein with the same Einstein factor $\lambda_E$ as $(E,\,h)$.

  $\bullet$ This is done by associating with every torsion-free coherent sheaf ${\cal F}$ on $X$ and every Hermitian fibre metric $h$ on the determinant line bundle $\det{\cal F}$ the function: \begin{eqnarray*}d^{(m)}_{\omega,\,\Omega,\,h}({\cal F}):X\longrightarrow\R, \hspace{5ex}  d^{(m)}_{\omega,\,\Omega,\,h}({\cal F}):=\frac{i\Theta_h(\det{\cal F})\wedge\omega^{m-1}\wedge\Omega}{dV_\omega},\end{eqnarray*} and by showing that in any short exact sequence of holomorphic vector bundles on $X$: \begin{eqnarray*}0\longrightarrow (S,\,h_S) \longrightarrow (E,\,h_E) \longrightarrow (Q,\,h_Q)\longrightarrow 0\end{eqnarray*} in which $h_E$ is an $(\omega,\,\Omega)$-Hermite-Einstein fibre metric on $E$ and $h_S$, resp. $h_Q$, are the induced, resp. quotient, Hermitian fibre metrics on $S$, resp. $Q$, one has \begin{eqnarray*}\frac{d^{(m)}_{\omega,\,\Omega,\,h_S}(S)}{\rk S}\stackrel{(i)}{\leq} \frac{d^{(m)}_{\omega,\,\Omega,\,h_E}(E)}{\rk E}\stackrel{(ii)}{\leq} \frac{d^{(m)}_{\omega,\,\Omega,\,h_Q}(Q)}{\rk Q},\end{eqnarray*} with equality in either (i) or (ii) being equivalent to the short exact sequence splitting holomorphically. In this case, $h_S$ and $h_Q$ are again $(\omega,\,\Omega)$-Hermite-Einstein with the same Einstein factors as $(E,\,h_E)$.

    $\bullet$ To see this, one uses the well-known formula $i\Theta_{h_E}(\det E) = \mbox{Trace}_E(i\Theta_{h_E}(E))$ (where we still denote by $h_E$ the fibre metric induced on $\det E$ by the fibre metric $h_E$ on $E$) and its analogues for $S$ and $Q$, as well as the standard formulae (see e.g. [Dem97, V-$\S14$]): \begin{eqnarray}\label{eqn:curvatures_S-Q_exact-seq}i\Theta_{h_S}(S) = i\Theta_{h_E}(E)_{|S} + i\beta^\star\wedge\beta \hspace{5ex}\mbox{and}\hspace{5ex} i\Theta_{h_Q}(Q) = i\Theta_{h_E}(E)_{|Q} + i\beta\wedge\beta^\star,\end{eqnarray} where $\beta\in C^\infty_{1,\,0}(X,\,\mbox{Hom}\,(S,\,Q))$ is the second fundamental form of $S$ in $E$ and $\beta^\star\in C^\infty_{0,\,1}(X,\,\mbox{Hom}\,(Q,\,S))$ is its adjoint. As is well known, $\bar\partial\beta^\star = 0$ and the above short exact sequence splits holomorphically if and only if the extension class $\{\beta^\star\}\in H^{0,\,1}_{\bar\partial}(X,\,\mbox{Hom}\,(Q,\,S))$ vanishes.

    $\bullet$ Under the $(\omega,\,\Omega)$-Hermite-Einstein assumption on $(E,\,h_E)$, inequalities (i) and (ii) are then seen to be respectively equivalent to the inequalities: \begin{eqnarray*}\mbox{Trace}_S\bigg(\frac{i\beta^\star\wedge\beta\wedge\omega^{m-1}\wedge\Omega}{dV_\omega}\bigg)\leq 0\leq\mbox{Trace}_Q\bigg(\frac{i\beta\wedge\beta^\star\wedge\omega^{m-1}\wedge\Omega}{dV_\omega}\bigg),\end{eqnarray*} which hold thanks to Lemma 2.9. of [Pop25]. Moreover, either of these inequalities is an equality if and only if $\beta=0$ (again by  Lemma 2.9. of [Pop25]), which implies the holomorphic splitting of the exact sequence of vector bundles.

\vspace{1ex}    

An induction on the rank of $E$ then completes the proof.    \hfill $\Box$

\vspace{3ex}


Recall that the generalisation to Gauduchon metrics of the main Uhlenbeck-Yau result of [UY86] and [UY89], obtained in [LY86] (see also [Buc88]), ensures that every {\it slope stable} (with respect to $\rho$, in the sense of Mumford-Takemoto) holomorphic vector bundle $E$ carries a $\rho$-Hermite-Einstein $C^\infty$ Hermitian fibre metric $h$. Combined with our Theorem \ref{The:H-E_implies_m-pos-stability}, this proves the implication (\ref{eqn:introd_slope-stab_implies_unif-stab}) highlighted in the introduction. 




The result that proves implication (\ref{eqn:introd_unif-semi-stab_not-implies_slope-semi}) (see introduction) is the following:

\begin{The}\label{The:unif-semi-stab_not-slope-semi-stab} There exists a Hermitian holomorphic vector bundle $(E,\,h)$ on a compact complex manifold $(X,\,\omega,\,\Omega)$ such that $(E,\,h)$ is {\bf uniformly $(\omega,\,\Omega)$-$m$-positively semi-stable}, but it is {\bf not $(\omega,\,\Omega)$-slope semi-stable}.

\end{The}

By $E$ being {\it $(\omega,\,\Omega)$-slope semi-stable} (or simply {\it $(\omega,\,\Omega)$-semi-stable}, as it was called in (iii) of Definition 2.10. of [Pop25]) we mean that $\mu_{\omega,\,\Omega}({\cal F})\leq\mu_{\omega,\,\Omega}(E)$ for every proper coherent subsheaf ${\cal F}$ of $E$, where (cf. (i)-(iii) of Definition 2.10. of [Pop25]): \begin{eqnarray*}\mu_{\omega,\,\Omega}({\cal F})=\frac{\deg_{\omega,\,\Omega}({\cal F})}{\rk {\cal F}}: = \frac{\int\limits_X c_1(\det {\cal F})\wedge\omega^{m-1}\wedge\Omega}{\rk {\cal F}}.\end{eqnarray*}

Note that in the case of complex curves, namely when $n=m=1$, the auxiliary forms $\omega$ and $\Omega$ appearing in the general definition play no essential role. Indeed, the factor $\omega^{m-1}$ reduces to $1$, while $\Omega$ has bidegree $(n-m,\,n-m) = (0,\,0)$, so it is merely a smooth positive function on $X$, which we may normalise to be identically equal to $1$. Thus, on curves, these notions reduce to their classical one-dimensional versions. On the other hand, for explicit computations of pointwise quantities on $\Proj^1$ (see the proof of Theorem \ref{The:unif-semi-stab_not-slope-semi-stab} below), we will choose $(\omega,\,\Omega) = (\omega_{FS},\,1)$, where $\omega_{FS}$ is the Fubini-Study metric of the complex projective line.  

\vspace{2ex}
 
Before giving an explicit example of a vector bundle that will prove Theorem \ref{The:unif-semi-stab_not-slope-semi-stab}, we describe a more general set-up that may lead to further examples and counter-examples in future. Let $C$ be a compact complex curve of genus $g\geq 0$ and let \begin{eqnarray}\label{eqn:extension_curve_examples}0\longrightarrow L_1\longrightarrow E \longrightarrow L_2\longrightarrow 0\end{eqnarray} be a short exact sequence of holomorphic vector bundles over $C$ such that $\rk(L_1) = \rk(L_2) = 1$ (so $\rk(E) = 2$) and $d_1:=\deg (L_1)> d_2:=\deg (L_2) > 0$.

\begin{Obs}\label{Obs:example_not-slope-semi-stable} The following inequalities hold: \begin{eqnarray}\label{eqn:ranks-slopes_ineq_extension_curve_examples}\mu (L_2) = \deg (L_2)<\mu (E) < \deg (L_1) = \mu (L_1),\end{eqnarray} where $\mu(F)=\deg (F)/\rk (F)$ denotes the classical {\bf slope} of a vector bundle $F$.

    In particular, $E$ is {\bf not slope semi-stable} since it is destabilised by $L_1$.   

\end{Obs}

\noindent {\it Proof.} From $\det E = \det L_1\otimes\det L_2 =  L_1\otimes L_2$ we infer $c_1(E) = c_1(\det E) = c_1(L_1) + c_1(L_2)$, where $c_1(F)$ denotes the first Chern class of a vector bundle $F$. Hence \begin{eqnarray*}\mu (E) = \frac{\deg (E)}{2} = \frac{\int\limits_C c_1(E)}{2} = \frac{\deg (L_1) + \deg (L_2)}{2}.\end{eqnarray*} Thus, (\ref{eqn:ranks-slopes_ineq_extension_curve_examples}) follows from this and from the hypothesis $\deg (L_1)> \deg (L_2)$. \hfill $\Box$

\vspace{2ex}

Our example that will prove Theorem \ref{The:unif-semi-stab_not-slope-semi-stab} will be a certain {\it split} extension of the type (\ref{eqn:extension_curve_examples}). However, it may be possible to construct other examples for different situations from {\it non-split} such extensions. As is well known, the extension is non-split if and only if its extension class $\{\beta^\star\}_{\bar\partial}\in H^{0,\,1}(C,\,\operatorname{Hom}(L_2,\,L_1))\simeq H^{0,\,1}(C,\, L_1-L_2)$ is non-zero, where $\beta^\star$ is the adjoint of the second fundamental form of $L_1$ in $E$ with respect to a given fibre metric on $E$ and the induced fibre metrics on $L_1$, $L_2$. Let $h^{0,\,1}(L_1-L_2)$ denote the complex dimension of this cohomology vector space.

By Serre duality, we get the first equality below: \begin{eqnarray*}h^{0,\,1}(L_1-L_2) = h^{1,\,0}(L_2-L_1) = h^0(K_C + L_2-L_1),\end{eqnarray*} where $h^0(F)$ denotes the dimension of the complex vector space $H^0(C,\,F)$ of global holomorphic sections of a vector bundle $F\to C$ and $K_C$ is the canonical bundle of $C$.

If $C$ is an elliptic curve, then $K_C$ is trivial. Since $\deg (L_2-L_1)<0$, we get $h^{0,\,1}(L_1-L_2) = 0$.

\begin{Conc}\label{Conc:elliptic-curve_solit-extension} If $C$ is an {\bf elliptic curve}, every extension of the type (\ref{eqn:extension_curve_examples}) with $\deg (L_1)> \deg (L_2)$ {\bf splits} holomorphically.

\end{Conc}

Thus, in order to get a non-split extension (\ref{eqn:extension_curve_examples}), we have to increase the genus $g$ of the curve $C$. The classical Riemann-Roch formula yields: \begin{eqnarray*}h^0(L_1-L_2) - h^{0,\,1}(L_1-L_2) = \deg (L_1-L_2) - g +1.\end{eqnarray*} Since $\deg (L_1-L_2) > 0$ in our case, this shows that $h^{0,\,1}(L_1-L_2) > 0$ (a condition equivalent to the existence of non-split extensions (\ref{eqn:extension_curve_examples})) whenever $g$ is large enough (explicitly, whenever $g > \deg (L_1-L_2) + 1 - h^0(L_1-L_2)$).

\vspace{2ex}

\noindent {\it Proof of Theorem \ref{The:unif-semi-stab_not-slope-semi-stab}.} $\bullet$ We take $C:=\Proj^1$ and the extension (\ref{eqn:extension_curve_examples}) to be {\it split}. In other words, we take $E:=L_1\oplus L_2\longrightarrow\Proj^1$, where $L_1={\cal O}(d_1)$ and $L_2={\cal O}(d_2)$ are holomorphic line bundles over $\Proj^1$ with $d_1>d_2>0$. We further take $\omega = \omega_{FS}$, the Fubini-Study metric of $\Proj^1$. Since $n=1$, $m=1$ and there is no $\Omega$. For $i\in\{1,\,2\}$, let $h_i$ be a $C^\infty$ Hermitian fibre metric, that we choose to be $\omega_{FS}$-Hermite-Einstein, on $L_i$. Finally, we equip $E$ with the $C^\infty$ Hermitian fibre metric $h:=h_1\oplus h_2$.

Since $C=\Proj^1$ is a curve and $E$ is of rank $2$, every proper coherent subsheaf ${\cal F}$ of $E$ is necessarily locally free of rank $1$. Equivalently, it is a holomorphic line bundle ${\cal O}(d)$ for some integer $d$.

\vspace{1ex}

$\bullet$ We will now determine the integers $d$ such that ${\cal O}(d)$ is a holomorphic subbundle of $E$. This is equivalent to finding all $d\in\Z$ such that there exist {\it non-zero} bundle morphisms: \begin{eqnarray*}\Phi = (s_1,\,s_2):{\cal O}(d)\longrightarrow {\cal O}(d_1)\oplus{\cal O}(d_2).\end{eqnarray*} Any such $\Phi$ is determined by a pair of sections $s_1\in H^0(\Proj^1,\,{\cal O}(d_1-d))$ and $s_2\in H^0(\Proj^1,\,{\cal O}(d_2-d))$ that do not have common zeroes in $\Proj^1$.

\vspace{1ex}

{\it Case $1$.} If $d>d_1$, then $d_i-d<0$ for $i=1,2$, so $H^0(\Proj^1,\,{\cal O}(d_1-d)) = H^0(\Proj^1,\,{\cal O}(d_2-d)) = \{0\}$.

We conclude that there exists {\it no non-zero} bundle morphism $\Phi$ as above in this case.

\vspace{1ex}

{\it Case $2$.} If $d_2<d\leq d_1$, then $d_2-d<0$, so $H^0(\Proj^1,\,{\cal O}(d_2-d)) = \{0\}$. This forces $s_2\equiv 0 $. We get $\Phi = (s_1,\,0)$, so we need $s_1$ (a global holomorphic section of the line bundle ${\cal O}(d_1-d)$) to be non-vanishing. This forces ${\cal O}(d_1-d)$ to be trivial. Equivalently, we must have $d=d_1$.

Conversely, when $d=d_1$, we can choose $s_1$ to be any non-zero constant (for example, $1$). Then, we get $\Phi=(1,\,0)$ and its image is the obvious line subbundle $L_1={\cal O}(d_1)\hookrightarrow{\cal O}(d_1)\oplus{\cal O}(d_2)$.

\vspace{1ex}

{\it Case $3$.} If $d\leq d_2$, then $d_1-d>0$ and $d_2-d\geq 0$, so $H^0(\Proj^1,\,{\cal O}(d_1-d))\neq\{0\}$ and $H^0(\Proj^1,\,{\cal O}(d_2-d))\neq\{0\}$. We can then find non-zero sections $s_1$ and $s_2$, one in each of these two vector spaces, that have no common zeroes. Indeed, we can always arrange for their respective zero divisors (which are finite linear combinations with positive integer coefficients of points) to have disjoint supports in $\Proj^1$.   

\vspace{1ex}

The conclusion of this case-by-case discussion is summed up in the following equivalence: \begin{eqnarray*}{\cal O}(d)\subset{\cal O}(d_1)\oplus{\cal O}(d_2) \iff d\leq d_2 \hspace{3ex}\mbox{or}\hspace{3ex} d=d_1.\end{eqnarray*}

\vspace{1ex}

$\bullet$ Let $d$ be an integer such that $d\leq d_2$ or $d=d_1$. Let $\sigma_d\in H^0(\Proj^1,\,{\cal O}(d_1-d)\oplus{\cal O}(d_2-d))$ be the canonical section induced by the inclusion ${\cal O}(d)\hookrightarrow{\cal O}(d_1)\oplus{\cal O}(d_2)$. Let $h_d$ be a $C^\infty$ $\omega_{FS}$-Hermite-Einstein fibre metric on ${\cal O}(d)$. 

Since each of the fibre metrics $h_1$, $h_2$, $h_d$ is $\omega_{FS}$-Hermite-Einstein, there exist constants $\lambda_1, \lambda_2,\lambda_d\in\R$ such that the following identities hold everywhere on $\Proj^1$: \begin{eqnarray*}i\Theta_{h_i}\bigg({\cal O}(d_i)\bigg) = \lambda_i\,\omega_{FS}, \hspace{5ex} i=1,2, \hspace{5ex}\mbox{and}\hspace{5ex} i\Theta_{h_d}\bigg({\cal O}(d)\bigg) = \lambda_d\,\omega_{FS}.\end{eqnarray*}

Integrating over $C$, we get: \begin{eqnarray*}\lambda_i = c\,d_i > 0, \hspace{5ex} i=1,2, \hspace{5ex}\mbox{and}\hspace{5ex} \lambda_d = c\,d,\end{eqnarray*} where $c:=1/\int_{\Proj^1}\omega_{FS}>0$ is a constant. This yields: \begin{eqnarray*}i\Theta:=i\Theta_{h\otimes h_d^{-1}}\bigg({\cal O}(d_1-d)\oplus{\cal O}(d_2-d)\bigg) = \begin{pmatrix}\lambda_1 - \lambda_d & 0 \\
           0 & \lambda_2 - \lambda_d\end{pmatrix}\,\omega_{FS} = c\,\begin{pmatrix}d_1 - d & 0 \\
           0 & d_2 - d\end{pmatrix}\,\omega_{FS}.\end{eqnarray*}

\vspace{1ex}

{\it Case $(a)$.} If $d<d_2$, then $d_1-d>0$ and $d_2-d>0$, so $i\Theta>0$. Therefore, for the stability function associated with ${\cal O}(d)\hookrightarrow{\cal O}(d_1)\oplus{\cal O}(d_2)$, we get: \begin{eqnarray*}Z_d(x):=\bigg\{(i\Theta\,\sigma_d)(x),\,\sigma_d(x)\bigg\}_{h\otimes h_d^{-1}}>0, \hspace{5ex} x\in\Proj^1.\end{eqnarray*}

\vspace{1ex}

{\it Case $(b)$.} If $d=d_2$, then $d_1-d>0$ and $d_2-d=0$. So $\sigma_{d_2} = (0,\,1)$, hence $Z_{d_2}\equiv 0$.

\vspace{1ex}

{\it Case $(c)$.} If $d=d_1$, then $d_1-d=0$ and $d_2-d<0$. We get $\sigma_{d_1} = (1,\,0)$, hence $Z_{d_1}\equiv 0$.

\vspace{1ex}

This case-by-case discussion shows that $Z_d\geq 0$ everywhere on $\Proj^1$ for every integer $d$ for which ${\cal O}(d)$ is a vector subbundle of $E$. We conclude that $(E,\,h) = ({\cal O}(d_1)\oplus{\cal O}(d_2),\,h_1\oplus h_2)$ is {\it uniformly semi-stable}.

On the other hand, we saw in Observation \ref{Obs:example_not-slope-semi-stable} that $E$ is {\it not slope semi-stable}.  \hfill $\Box$

\vspace{2ex}

\begin{Rem}\label{Rem:example_u-semi-stable_not-stable} The above proof shows that in the example $E={\cal O}(d_1)\oplus{\cal O}(d_2)\longrightarrow\Proj^1$ highlighted to prove Theorem \ref{The:unif-semi-stab_not-slope-semi-stab}, the only coherent subsheaves ${\cal F} = {\cal O}(d)\subset{\cal O}(E)$ whose associated stability function $Z_{\cal F} = Z_d$ is not positive everywhere on $\Proj^1$ are the two direct summands $L_1={\cal O}(d_1)$ and $L_2={\cal O}(d_2)$, for which $Z_d$ vanishes identically.

  Thus, $L_1$ and $L_2$ are the only obstructions to $(E,\,h)$ being uniformly stable.

\end{Rem}

\vspace{6ex}

\noindent {\bf References.} \\

\noindent [Buc88]\, N. P. Buchdahl --- {\it Hermitian-Einstein Connections and Stable Vector Bundles over Compact Complex Surfaces} --- Math. Ann. {\bf 280} (1988), 625-648. 

\vspace{1ex}

\noindent [Dem97]\, J.-P. Demailly --- {\it Complex Analytic and Algebraic Geometry} --- \url{https://www-fourier.univ-grenoble-alpes.fr/~demailly/manuscripts/agbook.pdf}

\vspace{1ex}


\noindent [DP25]\, S. Dinew, D. Popovici --- {\it $m$-Pseudo-effectivity and a Monge-Amp\`ere-Type Equation for Forms of Positive Degree} --- arXiv:2510.27362v1 [math.DG].

\vspace{1ex}

\noindent [Kob82]\, S. Kobayashi --- {\it Curvature and Stability of Vector Bundles} --- Proc. Jap. Acad. {\bf 58} (1982), 158-162.

\vspace{1ex}

\noindent [Kob87]\, S. Kobayashi --- {\it Differential Geometry of Complex Vector Bundles} --- Princeton University Press, 1987.

\vspace{1ex}

\noindent [Lub83]\, M. L\"ubke --- {\it Stability of Einstein-Hermitian Vector Bundles} --- Manuscripta Math. {\bf 42} (1983), 245-257.

\vspace{1ex}

\noindent [LT95]\, M. L\"ubke, A. Teleman --- {\it The Kobayashi-Hitchin Correspondence} --- World Scientific, 1995.

\vspace{1ex}

\noindent [LY86]\, J. Li, S-T Yau --- {\it Hermitian-Yang-Mills Connection on Non-K\"ahler Manifolds} --- Mathematical aspects of string theory (San Diego, Calif., 1986), 560–573, Adv. Ser. Math. Phys., {\bf 1}, World Sci. Publishing, Singapore, 1987.

\vspace{1ex}

\noindent [Pop05]\, D. Popovici --- {\it A Simple Proof of a Theorem by Uhlenbeck and Yau} --- Math. Z. {\bf 250} (2005), 855-872.

\vspace{1ex}

\noindent [Pop25]\, D. Popovici --- {\it Generalised Hermite-Einstein Fibre Metrics and Slope Stability for Holomorphic Vector Bundles} --- arXiv:2512.24932v2 [math.AG].

\vspace{1ex}

\noindent [UY86]\, K. Uhlenbeck, S.-T. Yau --- {\it On the Existence of Hermitian-Yang-Mills Connections in Stable Vector Bundles} --- Communications in Pure and Applied Mathematics, Vol. {\bf XXXIX} (1986), Supplement, S257-S293. 

\vspace{1ex}

\noindent [UY89]\, K. Uhlenbeck, S.-T. Yau --- {\it A Note on Our Previous Paper: On the Existence of Hermitian-Yang-Mills Connections in Stable Vector Bundles} --- Communications in Pure and Applied Mathematics, Vol. {\bf XLII} (1989), 703-707.

\vspace{6ex}

\noindent Institut de Math\'ematiques de Toulouse, Universit\'e de Toulouse,

\noindent 118 route de Narbonne, 31062 Toulouse, France

\noindent Email: popovici@math.univ-toulouse.fr

\end{document}